\documentclass[12pt,a4paper,twoside,reqno]{amsart}
\title[Non-commutative Courant algebroids and Quiver algebras]{Non-commutative Courant algebroids\\ and Quiver algebras}

\author[L. \'Alvarez-C\'onsul]{Luis \'Alvarez-C\'onsul}
\address{Instituto de Ciencias Matem\'aticas (CSIC-UAM-UC3M-UCM)\\
Nicol\'as Cabrera 13-15\\
Cantoblanco UAM\\ 28049 Madrid, Spain}
\email{l.alvarez-consul@icmat.es}

\author[D. Fern\'andez]{David Fern\'andez}
  \address{Instituto de Matem\'atica Pura e Aplicada (IMPA) \\
Estrada Dona Castorina, 110\\
22460-320\\ Rio de Janeiro\\ Brazil
    }
  \email{davidfa@impa.br}

\usepackage{hyperref,url}
\usepackage{amssymb,latexsym}
\usepackage{amsmath,amsthm}
\usepackage[latin1]{inputenc}
\usepackage{color}

\usepackage{enumerate}
\usepackage[all]{xy} \CompileMatrices
\SelectTips{cm}{12} 
\usepackage{bbm}

\theoremstyle{plain}
\newtheorem{theorem}{Theorem}[section]
\newtheorem{lemma}[theorem]{Lemma}

\newtheorem{proposition}[theorem]{Proposition}

\theoremstyle{definition}
\newtheorem{definition}[theorem]{Definition}
\newtheorem{definition-theorem}[theorem]{Definition-Theorem}
\newtheorem{example}[theorem]{Example}

\theoremstyle{remark}
\newtheorem{remark}[theorem]{Remark}

\newtheorem{claim}[theorem]{Claim}

\newtheorem{framework}[theorem]{Framework}

\newcommand{\secref}[1]{\S\ref{#1}}

\numberwithin{equation}{section} 

  \newcommand{\surj}{\to\kern-1.8ex\to}
\newcommand{\hra}{\hookrightarrow}
\newcommand{\lto}{\longrightarrow}
\newcommand{\lra}[1]{\stackrel{#1}{\longrightarrow}}

\newcommand{\defeq}{\mathrel{\mathop:}=} 
\renewcommand{\(}{\left(}
\renewcommand{\)}{\right)}

\renewcommand{\mod}{\operatorname{mod}}
\newcommand{\DD}{\mathbb{D}}
\newcommand{\NN}{\mathbb{N}}
\newcommand{\XX}{\mathbb{X}}
\newcommand{\ZZ}{\mathbb{Z}}
\newcommand{\kk}{{\Bbbk}} 
\newcommand{\op}{{\operatorname{op}}} 
\newcommand{\e}{{\operatorname{e}}} 
\newcommand{\du}{\operatorname{d}\!}
\newcommand{\out}{{\operatorname{out}}} 
\newcommand{\inn}{{\operatorname{inn}}} 
\newcommand{\Id}{{\operatorname{Id}}} 

\newcommand{\Tor}{{\operatorname{Tor}}}

\newcommand{\aad}{\operatorname{ \mathbbm{a}d}}

\newcommand{\lr}[1]{
  \left\{\mkern-6mu\left\{#1\right\}\mkern-6mu\right\}}

  \newcommand{\cc}[1]{
  [\mkern-3mu[#1]\mkern-3mu]}

  \newcommand{\mm}[1]{
  \mid \mkern-9mu\mid#1\mid\mkern-9mu\mid}

\newcommand{\QQ}{{Q}} 
\newcommand{\Q}[1]{{Q_{#1}}} 
\newcommand{\PP}{{P}} 
\renewcommand{\P}[1]{{P_{#1}}} 
\newcommand{\Pp}{{P^+_1}} 
\newcommand{\dQ}[1]{{\overline{Q}_{#1}}} 
\newcommand{\dPP}{{\overline{P}}} 
\newcommand{\cas}{\operatorname{\mathtt{cas}}}

\newcommand{\End}{\operatorname{End}}

\newcommand{\T}{\operatorname{T}}
\newcommand{\Hom}{\operatorname{Hom}}

\newcommand{\Der}{\operatorname{Der}}
\newcommand{\diff}[1]{{\Omega_R^1{#1}}}
\newcommand{\DR}[2]{{\operatorname{DR}_R^{#1}{#2}}}
\newcommand{\D}{\operatorname{\mathbb{D}er}}

\newcommand{\Tr}{\operatorname{Tr}}

\newcommand{\EEnd}{\operatorname{\mathbb{E}nd}}

\newcommand{\Eu}{{\operatorname{Eu}}}
\newcommand{\AAt}{\operatorname{ \mathbb{A}t}}
\newcommand{\LL}{{\operatorname{L}}}
\newcommand{\iii}{{\operatorname{i}}}
\newcommand{\ab}{{\operatorname{ab}}}
\newcommand{\dR}{{\operatorname{DR}}}

\thanks{Partial support of the first author was provided the Spanish
  MINECO, through the Severo Ochoa Programme (grant SEV-2015-0554),
  and by MINECO grants MTM2013-43963-P, MTM2016-81048-P. The initial
  work of the second author was provided by a FPI-UAM
  grant. Subsequent support of the second author was provided by IMPA
  and CAPES through their postdoctorate of excellence fellowships at
  UFRJ}

\begin{document}

\begin{abstract}
In this paper, we develop a differential-graded symplectic
(Batalin--Vilkovisky) version of the framework of Crawley-Boevey,
Etingof and Ginzburg on noncommutative differential geometry based on
double derivations to construct non-commutative analogues of the
Courant algebroids introduced by Liu, Weinstein and Xu. Adapting
geometric constructions of \v Severa and Roytenberg for
(commutative) graded symplectic supermanifolds, we express the BRST
charge, given in our framework by a `homological double derivation',
in terms of Van den Bergh's double Poisson algebras for graded
bi-symplectic non-commutative 2-forms of weight 1, and in terms of our
non-commutative Courant algebroids for graded bi-symplectic
non-commutative 2-forms of weight 2 (here, the grading, or ghost
degree, is called weight). We then apply our formalism to obtain
examples of exact non-commutative Courant algebroids, using
appropriate graded quivers equipped with bi-symplectic forms of weight
2, with a possible twist by a closed Karoubi--de Rham non-commutative
differential 3-form.
\end{abstract}

\maketitle

\section{Introduction}

In this paper, we propose a notion of non-commutative Courant
algebroid that satisfies the Kontsevich--Rosenberg principle, whereby
a structure on an associative algebra has geometric meaning if it
induces standard geometric structures on its representation
spaces~\cite{KR00}. Replacing vector fields on manifolds by
Crawley-Boevey's double derivations on associative
algebras~\cite{CB01}, this principle has been successfully applied by
Crawley-Boevey, Etingof and Ginzburg~\cite{CBEG07} to symplectic
structures and by Van den Bergh to Poisson
structures~\cite{VdB08,VdB08a}.

Courant algebroids, introduced in differential geometry by Liu,
Weinstein and Xu~\cite{LWX97}, generalize the notion of the Drinfeld
double to Lie bialgebroids. They axiomatize the properties of the
Courant--Dorfman bracket, introduced by Courant and
Weinstein~\cite{CW88, Cou90}, and Dorfman~\cite{Dor87}, to provide a
geometric setting for Dirac's theory of constrained mechanical
systems~\cite{Dir64}.

Our approach is based on a well-known correspondence (in commutative
geometry) between Courant algebroids and a suitable class of
differential graded symplectic manifolds. More precisely, symplectic
$\mathbb{N}$Q-manifolds are non-negatively graded manifolds (the
grading is called weight), endowed with a graded symplectic structure
and a symplectic homological vector field $Q$ of weight 1. They encode
higher Lie algebroid structures in the Batalin--Vilkovisky formalism
in physics, where the weight keeps track of the ghost
number. Following ideas and results of \v{S}evera~\cite{Sev00},
Roytenberg~\cite{Roy00} proved that symplectic $\mathbb{N}$Q-manifolds
of weights 1 and 2 are in 1-1 correspondence with Poisson manifolds
and Courant algebroids, respectively. Our method to construct
non-commutative Courant algebroids is to adapt this result to a graded
version of the formalism of Crawley-Boevey, Etingof and Ginzburg.

After setting out our notation (Section~\ref{sec:notation}), and
reviewing several constructions involving graded quivers,
non-commutative differential forms and double derivations
(Section~\ref{sub:elements}), in Section~\ref{sec:N-algebra-assoc} we
start generalizing to graded associative algebras the theories of
bi-symplectic forms and double Poisson brackets of
Crawley-Boevey--Etingof--Ginzburg and Van den Bergh, respectively. In
this framework, we obtain suitable Darboux theorems for graded
bi-symplectic forms, and prove a 1-1 correspondence between
appropriate bi-symplectic $\mathbb{N}$Q-algebras of weight 1 and Van
den Bergh's double Poisson algebras (Section~\ref{sec:Darboux}).  We
then use suitable non-commutative Lie and Atiyah algebroids to
describe bi-symplectic $\mathbb{N}$-graded algebras of weight 2 whose
underlying graded algebras are graded-quiver path algebras, in terms
Van den Bergh's pairings on projective bimodules
(Section~\ref{sec:bisympl-weight-2}).  To complete the data that
determines Courant algebroids, in Section~\ref{sub:DoubleCourant}, we
define bi-symplectic $\mathbb{N}$Q-algebras and use non-commutative
derived brackets to calculate the algebraic structure that corresponds
to symplectic $\mathbb{N}$Q-algebras of this type. By the analogy with
Roytenberg's correspondence for commutative algebras~\cite{Roy09},
this structure can be regarded as a double Courant--Dorfman algebra,
although we call them simply double Courant algebroids. Finally, we
consider examples of non-commutative Courant algebroids obtained by
deforming standard non-commutative Courant algebroids associated to
graded quivers (Section~\ref{sub:standard}).

\subsection*{Acknowledgements}

The authors wish to thank Henrique Bursztyn, Alejandro Cabrera, Mario
Garcia-Fernandez, Reimundo Heluani, Alastair King, and Marco Zambon,
for useful discussions.

\section{Notation and Conventions}
\label{sec:notation}

Throughout the paper, unless otherwise stated, all associative
algebras will be unital and finitely generated over a fixed base field
$\kk$ of characteristic 0. The unadorned symbols
$\otimes=\otimes_\kk$, $\Hom=\Hom_\kk$, will denote the tensor product
and the space of linear homomorphisms over the base field. The
opposite algebra and the enveloping algebra of an associative algebra
$B$ will be denoted $B^{\op}$ and $B^{\e}\defeq B\otimes B^{\op}$,
respectively. Given an associative algebra $R$, an $R$-algebra will
mean an associative algebra $B$ together with a unit preserving
algebra morphism $R\to B$ (note that the image of $R$ may not be in
the centre of $B$), and by a morphism $B_1\to B_2$ of $R$-algebras we
mean an algebra morphism such that $R\to B_2$ is the composite of the
given morphisms $R\to B_1$ and $B_1\to B_2$.

A graded algebra, and a graded $A$-module, will mean an $\NN$-graded
associative algebra $A$, and a $\ZZ$-graded left $A$-module $V$, with
degree decompositions
\[
A=\bigoplus_{d\in\NN}A_d, \quad 
V=\bigoplus_{d\in\mathbb{Z}}V_d, 
\]
where $\NN\subset\ZZ$ are the set of non-negative integers and the set
of integers, respectively. An element $v\in V_d$ is called homogeneous
of degree $\lvert v\rvert=d$. Depending on the context, the degree
will be called \emph{weight}, when it plays the role of the `ghost
degree' in the BRST or Batalin--Vilkovisky quantization in physics
(see, e.g.,~\secref{sub:grQ-pathAlg}). For any $N\in\ZZ$, the
$N$-degree of a homogeneous $v\in V$ is $\lvert v\rvert_N\defeq\lvert
v\rvert+N$ (this notation will be used in~\secref{sub:DDP}). The
graded $A$-module $V[n]$, with degree shifted by $n$, is
$V[n]=\bigoplus_{d\in\mathbb{Z}}V[n]_d$ with $V[n]_d=V_{d+n}$. Given
another graded module $V'$, a graded linear map and a graded
$A$-module homomorphism $V\to V'$ are respectively a linear map and an
$A$-module homomorphism, carrying $V_d$ to $V'_d$, for all
$d\in\ZZ$. Ungraded modules are viewed as graded modules concentrated
in degree 0.

We will use the Koszul sign rule to generalize standard constructions
for algebras and modules to graded algebras and graded modules, such
as tensor products, or the opposite $A^\op$ and the enveloping algebra
$A^\e$ of a graded algebra $A$, and to identify graded left
$A^\e$-modules with graded $A$-bimodules. For instance, the tensor
product of two graded algebras $A$ and $B$ is the graded algebra with
underlying vector space $A\otimes B$, and multiplication
\[
(a_1\otimes b_1)(a_2\otimes b_2)=(-1)^{\lvert b_1\rvert\lvert a_2\rvert}a_1a_2\otimes b_1b_2,
\]
for homogeneous $a_1,a_2\in A$ and $b_1,b_2\in B$.
Given a bigraded module 
\[
V=\bigoplus_{d,p\in\ZZ}V_d^p, 
\]
the elements $v$ of $V_d^p$ are called homogeneous of bigrading
$(\lvert v\rvert,\lVert v\rVert)\defeq(d,p)$. In this case, the
bigraded Koszul sign rule applied to two homogeneous elements $u,v\in
V$ yields a sign $(-1)^{(\lvert u\rvert,\lVert u\rVert)\cdot(\lvert
  v\rvert,\lVert v\rVert)}$, where
\begin{equation}\label{eq:bigraded-Koszul}
(\lvert u\rvert,\lVert u\rVert)\cdot(\lvert v\rvert,\lVert v\rVert)
\defeq \lvert v\rvert\lvert v\rvert+\lVert u\rVert\lVert v\rVert.
\end{equation}

Let $R$ be an associative algebra, and $A$ a graded $R$-algebra, that
is, a graded algebra equipped with an algebra homomorphism $R\to A$
with image in $A_0$. The $A$-bimodule $A\otimes A$ has two graded
$A$-bimodule structures (see \cite{BCER12}), called the \emph{outer}
graded bimodule structure $(A\otimes A)_{\out}$ and the \emph{inner}
graded bimodule structure $(A\otimes A)_{\inn}$, corresponding to the
left graded $A^\e$-module structure ${}_{A^\e}A^\e$ and right graded
$A^\e$-module structure ${}_{({A^\e)}^{\op}}A^\e=(A^\e)_{A^\e}$,
respectively. In other words, for all homogeneous $a,b,u,v\in A$,
\begin{align*}
 a(u\otimes v)b=au\otimes vb\quad  \text{ in } (A\otimes A)_\out,
\\
a*(u\otimes v)*b=(-1)^{\lvert a\rvert \lvert u\rvert+\lvert a\rvert\lvert b\rvert+\lvert b\rvert\lvert v\rvert}ub\otimes av \quad \text{ in } (A\otimes A)_\inn,
\end{align*}
The \emph{dual} of a graded $A$-bimodule $V$ is the graded
$A$-bimodule
\begin{equation}\label{eq:dual-bimodule}
V^\vee\defeq\bigoplus_{d\in\ZZ}V_d^\vee,\quad
\text{with }
V_d^\vee\defeq\Hom_{A^\e}(V_d,(A\otimes A)_\out),
\end{equation}
where the graded $A$-bimodule structure on $V_d^\vee$ is induced by
the one on $(A\otimes A)_{\inn}$.

The above inner and outer bimodule structures are special cases of the
following general construction, that for simplicity we will describe
only for ungraded algebras and modules. Let $B$ and $V$ be an
(ungraded) associative algebra and an (ungraded) $B$-bimodule,
respectively. Then the $n$-th tensor power $V^{\otimes n}$ has many
$B$-bimodule structures (cf.~\cite[pp. 5718,
5732]{VdB08}). Following~\cite{DSKV15}, the $d$-th left and right
$B$-module structures of $V^{\otimes n}$ are 
\begin{equation}
\begin{aligned}
b*_i(v_1\otimes\cdots \otimes v_n)&=v_1\otimes\cdots\otimes v_i\otimes b v_{i+1}\otimes\cdots\otimes v_n,
\\
(v_1\otimes\cdots\otimes v_n)*_i b&=v_1\otimes\cdots\otimes v_{n-i}b\otimes\cdots\otimes v_n,
\end{aligned}
\label{square-op}
\end{equation}
for all $i=0,...,n-1$, and $b\in B$, $v_1,\ldots,v_n\in V$, where the
index denotes the number of `jumps'. Then the outer bimodule structure
and the inner bimodule structure are given by $a_1*_0(a\otimes b) *_0
b_1$ and $a_1*_1(a\otimes b)*_1 b_1$, respectively. We use a similar
notation for the tensor product of an element $u$ of $V$ and an
element $v_1\otimes\cdots\otimes v_n$ of $V^{\otimes n}$, namely, 
\begin{equation}
\begin{aligned}
u\otimes_i(v_1\otimes\cdots\otimes v_n)&=v_1\otimes\cdots\otimes
v_i\otimes u \otimes v_{i+1}\otimes v_n
\in V^{\otimes n+1},
\\
(v_1\otimes\cdots\otimes v_n)\otimes_i u&=v_1\otimes\cdots\otimes
v_{n-i}\otimes u\otimes\cdots\otimes v_n
\in V^{\otimes n+1}.
\end{aligned}
\label{jumping-notation}
\end{equation}

\section{Basics on graded non-commutative algebraic geometry and quivers}
\label{sub:elements}

\subsection{Background on graded quivers}

\subsubsection{Quivers}
\label{sub:basics quivers}

To fix notation, here we recall a few relevant definitions from the
theory of quivers (see, e.g.,~\cite{ARS95, ASS06} for introductions to
this topic).
A \emph{quiver} $\QQ$ consists of a set $\Q0$ of vertices, a set $\Q1$
of arrows, and two maps $t,h\colon\Q1\to\Q0$ assigning to each arrow
$a\in\Q1$, its \emph{tail} and its \emph{head}. We write $a\colon i\to
j$ to indicate that an arrow $a\in\Q1$ has tail $i=t(a)$ and head
$j=h(a)$. Given an integer $\ell\geq 1$, a non-trivial path of length
$\ell$ in $\QQ$ is an ordered sequence of arrows $p=a_\ell\cdots a_1$,
such that $h(a_j)=t(a_{j+1})$ for $1\leq j<\ell$. This path $p$ has
tail $t(p)=t(a_1)$, head $h(p)=h(a_\ell)$, and is represented
pictorially as follows.
\begin{equation}\label{eq:path}
p\colon 
\bullet\stackrel{a_\ell}{\longleftarrow}\bullet\longleftarrow\cdots\longleftarrow\bullet\stackrel{a_1}{\longleftarrow}\bullet
\end{equation}
For each vertex $i\in\Q0$, $e_i$ is the \emph{trivial path} in $\QQ$,
with tail and head $i$, and length $0$. A \emph{path} in $\QQ$ is
either a trivial path or a non-trivial path in $\QQ$. The path algebra
$\kk\QQ$ is the associative algebra with underlying vector space
\[
\kk\QQ=\bigoplus_{\text{paths $p$}}\kk p,
\]
that is, $\kk\QQ$ has a basis consisting of all the paths in $\QQ$, with
the product $pq$ of two non-trivial paths $p$ and $q$ given by the
obvious path concatenation if $t(p)=h(q)$, $pq=0$ otherwise,
$pe_{t(p)}=e_{h(p)}p=p$, $pe_i=e_jp=0$, for non-trivial paths $p$ and
$i,j\in\Q0$ such that $i\neq t(p)$, $j\neq h(p)$, and $e_ie_i=e_i, e_ie_j=0$
for all $i,j\in\Q0$ if $i\neq j$. We will always assume that a
quiver $\QQ$ is finite, i.e. its vertex and arrow sets are finite, so
$\kk\QQ$ has a unit
 \begin{equation}\label{eq:path-algebra-unit}
1=\sum_{i\in\Q0}e_i.
\end{equation}
Define vector spaces
\[
R_\QQ=\bigoplus_{i\in \Q0}\kk e_i,\quad V_\QQ=\bigoplus_{a\in \Q1}\kk a.
\]
Then $R_\QQ\subset \kk\QQ$ is a semisimple commutative (associative)
algebra, 
because it is the subalgebra spanned by the trivial paths, which are a
complete set of orthogonal idempotents of $\kk\QQ$. 
Furthermore, as $V_\QQ$ is a vector space with basis consisting of the
arrows, it is an $R_\QQ$-bimodule with multiplication $e_jae_i=a$ if
$a\colon i\to j$ and $e_iae_j=0$ otherwise, and the path algebra is
the tensor algebra of the bimodule $V_\QQ$ over $R\defeq R_\QQ$, that
is 
\begin{equation}\label{eq:grPathAlg.1}
\kk\QQ=\T_{R}V_\QQ,
\end{equation}
where a path $p=a_\ell\cdots a_1\in \kk\QQ$ is identified with a tensor
product $a_\ell\otimes\cdots\otimes a_1\in \T_{R}V_\QQ$.

Given a quiver $\QQ$, the \emph{double quiver} of $\QQ$ is the quiver
$\overline{\QQ}$ obtained from $\QQ$ by adjoining a reverse arrow
$a^*\colon j\to i$ for each arrow $a\colon i\to j$ in $\QQ$. Following
\cite[\S 8.1]{CBEG07}, for convenience we introduce the function
\begin{equation}
\varepsilon\colon\overline{\QQ}\lto \{\pm 1\}\colon\quad a\longmapsto \varepsilon(a)=\begin{cases}
1 &\mbox{if } a\in \Q1,\\
-1 & \mbox{if }  a\in\QQ^*_1\defeq \dQ1\setminus\Q1.
\end{cases}
\label{varepsilon}
\end{equation}

\subsubsection{Graded quivers and graded path algebras}
\label{sub:grQ-pathAlg}

The following construction is partially inspired by similar ones in
physics~\cite{Laz05} and for Calabi--Yau algebras
(cf.~\cite{Gin07},~\cite[\S 10.3]{VdB15}).
%

\begin{definition}\label{def:graded-quiver}
  A \emph{graded quiver}
      is a quiver $\PP$ together with a map
      \[
      \lvert -\rvert \colon \P1\lto\NN\colon \quad a\longmapsto\lvert a\rvert
      \] that to each arrow $a$ assigns its \emph{weight} $\lvert
  a\rvert$. The \emph{weight} of the graded quiver is
\[
\lvert\PP\rvert\defeq\max_{a\in\PP_1}\,\lvert a\rvert.
\]
Given a graded quiver $\PP$ and an integer $N\geq \lvert\PP\rvert$, the
\emph{weight $N$ double graded quiver} of $\PP$ is the graded quiver
$\dPP$ obtained from $\PP$ by adjoining a reverse arrow $a^*\colon
j\to i$ for each arrow $a\colon i\to j$ in $\PP$, with weight $\lvert
a^*\rvert =N-\lvert a\rvert$ (note that $\lvert\dPP\rvert=N$ if and
only if $\PP$ has at least one arrow of weight $0$).
\end{definition}


The weight function $\lvert -\rvert\colon \P1\to\NN$ induces a
structure of graded associative algebra, called the \emph{graded path
  algebra} of $P$, on the path algebra $\kk\PP$ of the underlying
quiver of $\PP$, where a trivial path $e_i$ has weight $\lvert
e_i\rvert=0$, and a non-trivial path $p=a_\ell\cdots a_1$ has weight
\[
\lvert p\rvert=\lvert a_1\rvert+\cdots+\lvert a_\ell\rvert.
\]
Let $R\defeq R_\PP$ the algebra with basis the trivial paths in a
graded quiver $\PP$, and
\begin{equation}\label{eq:grPathAlg.3}
V_\PP=\bigoplus_{a\in\P1}\kk a
\end{equation}
the graded $R$-bimodule with basis consisting of the arrows in $\P1$,
where $a\in\P1$ has weight $\lvert a\rvert$, and
multiplications $e_jae_i=a$ if $i=t(a), j=h(a)$, and $e_iae_j=0$
otherwise, for all $a\in\P1$.
As in~\eqref{eq:grPathAlg.1}, the graded path algebra $k\PP$ is the
graded tensor algebra
\begin{equation}\label{eq:grPathAlg.2}
\kk\PP=\T_RV_\PP,
\end{equation}
where a path $p=a_\ell\cdots a_1\in k\PP$ is identified with a tensor
product $a_\ell\otimes\cdots\otimes a_1\in \T_RV_\PP$.

The graded path algebra $k\PP$ can be expressed as a graded tensor
algebra in another way, using the following two subquivers of
$\PP$. The \emph{weight 0 subquiver} of $\PP$ is the (ungraded) quiver
$\QQ$ with vertex set $\Q0=\P0$, arrow set $\Q1=\{a\in\P1\mid\lvert
a\rvert=0\}$, and tail and head maps $t,h\colon\Q1\to\Q0$ obtained
restricting the tail and head maps of $\PP$. The \emph{higher-weight
  subquiver} of $\PP$ is the graded quiver $\PP^+$ with vertex set
$\PP_0^+=\P0$, arrow set $\Pp=\{a\in\P1\mid\lvert a\rvert>0\}$, tail
and head maps $t,h\colon\Pp\to\PP_0^+$ obtained restricting the tail
and head maps of $\PP$, and weight function $\Pp\to\NN$ obtained
restricting the weight function of $\PP$. Later it will also be useful
to consider the graded subquivers $\P{(w)}\subset\PP$ with vertex set
$\P0$ and arrow set $\P{(w),1}$ consisting of all the arrows $a\in\Pp$
with weight $w$, for $0\leq w\leq\lvert\PP\rvert$.

In Lemma~\ref{lem:grPathAlg-TensorAlg}, $BaB\subset A$ denotes the
$B$-sub-bimodule of $_BA_B$ generated by $a\in\ A$.

\begin{lemma}
\label{lem:grPathAlg-TensorAlg}
Let $B=\kk\QQ$ be the path algebra of $\QQ$. Define the graded $B$-bimodule
\begin{equation}\label{eq:grPathAlg.MP+}
M_\PP\defeq\bigoplus_{a\in\Pp}B aB.
\end{equation}
Then $M_\PP$ is a finitely generated projective $B$-bimodule and the
graded path algebra $A=\kk\PP$ of $\PP$ is canonically isomorphic to
the graded tensor algebra of $M_\PP$ over $B$, that is, 
\begin{equation}\label{eq:grPathAlg-TensorAlg}
A=\T_BM_\PP.
\end{equation}
\end{lemma}

\begin{proof}
As $M_\PP$ has a basis consisting of paths in $\PP$ of weight 1,
its tensor algebra $\T_BM_\PP$ has a basis consisting of paths in
$\PP$ of arbitrary weight, that is, $A=\T_BM_\PP$ as graded vector
spaces, and hence as graded algebras, with concatenation of paths on
$A$ identified with multiplication of paths in $\T_BM_\PP$
(see~\cite[Lemma 3.3.7]{Fer16} for further details).
\end{proof}

It will be useful to decompose 
\[
M_\PP=\bigoplus_{w=1}^{\lvert\PP\rvert} M_{\P{(w)}},
\text{ with }
M_{\PP_{(w)}}=\bigoplus_{a\in\P{(w),1}}BaB,
\]
where $M_{\PP_{(w)}}$ is a finitely generated projective $B$-bimodule
of weight $w$, because so are the $B$-bimodules $BaB$, and hence
$E_w\defeq M_{\PP_{(w)}}[w]$ is a finitely generated projective
$B$-bimodule (concentrated in weight $0$).


\subsection{Graded non-commutative differential forms and double derivations}
\label{sub:ncdiffforms}

Let $R$ be an associative algebra and $A$ a graded $R$-algebra.  Given
a graded $A$-bimodule $M$, a \emph{derivation of weight $d$} of $A$
into $M$ is an additive map $\theta\colon\, A\to M$, such that
$\theta(A_i)\subset M_{i+d}$, satisfying the graded Leibniz rule
$\theta(ab)=(\theta a)b+(-1)^{d\lvert a\rvert}a(\theta b)$ for all
$a,b\in A$. It is called an \emph{$R$-linear graded derivation} if,
furthermore, $\theta(R)=0$, i.e., it is a graded $R$-bimodule morphism
of weight $d$. The graded vector space of graded $R$-derivations is
\[
\Der_R(A,M)=\bigoplus_{d\in\ZZ}\Der^d_R(A,M),
\]
where $\Der^d_R(A,M)$ is the vector space of $R$-linear graded
derivations of weight $d$.

\subsubsection{Graded non-commutative differential forms}

\begin{lemma}[cf. {\cite[\S 2]{Qui89}}]
\label{lem:NC-diff-forms}
There exists a unique pair $(\diff{A},\du\,)$ (up to isomorphism),
where $\diff{A}$ is a graded $A$-bimodule and $\du\,\colon
A\to\diff{A}$ is an $R$-linear graded derivation, satisfying the
following universal property: for all pairs $(M,\theta)$ consisting of
a graded $A$-bimodule $M$ and an $R$-linear graded derivation
$\theta\colon A\to M$, there exists a unique graded $A$-bimodule
morphism $i_\theta\,\colon\diff{A}\to M$ such that
$\theta=i_\theta\circ\du\,$.
\label{Univ-Property-Kaehler-diff}
\end{lemma}

The elements of $\diff{A}$ are called relative noncommutative
differential 1-forms of $A$ over $R$. Concretely, we can construct
$\Omega^1_R A$ as the kernel of the multiplication $A\otimes A\to A$,
and
\begin{equation}
\du\,\colon A\lto \Omega^1_R A\colon \quad a\longmapsto   \du a=1\otimes a-a\otimes 1.
\label{univ-derivation}
\end{equation}

\begin{lemma}[cf. {\cite[Proposition 2.6]{CQ95}}]
\label{lema 2.6}
Let $A=\T_{B}M$ be the graded tensor algebra of a graded bimodule $M$
over an  associative $R$-algebra $B$. Then there is a canonical
isomorphism 
\[
A\otimes_B M\otimes_B  A\lra{\cong} \Omega^{1}_{B} A
\colon a_1\otimes m\otimes a_2\longmapsto a_1(\du m) a_2,
\]
where $\du\,\colon A\to\Omega^1_R A$ is the universal graded derivation.
\end{lemma}

\begin{example}
If $A$ is the graded path algebra of a graded quiver $\PP$, then
\[
\diff A=\bigoplus_{a\in\P1}(Ae_{h(a)})\du a (e_{t(a)}A).
\]
\end{example}

The \emph{relative graded non-commutative differential
  forms} of $A$ over $R$ are the elements of 
\begin{equation}
\Omega^\bullet_R A:=T_A\Omega^1_R A, 
\label{NC-diff-forms}
\end{equation}
i.e., the tensor algebra of the graded $A$-bimodule $\Omega^1_R
A$. This is a bigraded algebra, with bigrading denoted
$(\lvert-\rvert,\lVert-\rVert)$, where the \emph{weight}
$\lvert-\rvert$ is induced by the grading of $A$ (also called weight),
and the \emph{degree} $\lVert-\rVert$ is the `form degree', i.e., the
elements of the $n$-th tensor power $(\Omega^1_R A)^{\otimes_A n}$
have degree $n$.

As for ungraded associative algebras, $(\Omega^\bullet_R A,\du)$ has
trivial cohomology (cf., e.g.,~\cite[\S 11.4]{Gin05}). To obtain
interesting cohomology spaces, one defines the \emph{non-commutative
  Karoubi--de Rham complex} of $A$ (relative over $R$) as the
bigraded vector space
\begin{equation}
\DR{\bullet} A=\Omega^\bullet_R A/[\Omega^\bullet_R A,\Omega^\bullet_R A],
\label{Karoubi-de-Rham}
\end{equation}
where the bigraded commutator $[-,-]$ is given by the bigraded Koszul
sign rule, i.e.,
\[
[\alpha,\beta]\defeq\alpha\beta-(-1)^{(\lvert\alpha\rvert,\lVert\alpha\rVert)\cdot(\lvert\beta\rvert,\lVert\beta\rVert)}\beta\alpha,
\]
with the sign given by~\eqref{eq:bigraded-Koszul}. Then the
differential $\du\,\colon\ \Omega^\bullet_R A\to \Omega^{\bullet+1}_R
A$ descends to another differential $\du\,\colon \DR{\bullet} A\to
\DR{\bullet+1} A$, and so $\DR{\bullet} A$ is a differential bigraded
vector space.

\subsubsection{Graded double derivations}
\label{subsub:graded-double-Der}

By Lemma~\ref{lem:NC-diff-forms}, there is a canonical isomorphism
\begin{equation}\label{sub:ncdiffforms.corep}
\Der_R(A,M)\lra{\cong}\Hom_{A^\e}(\diff{A},M)\colon\quad \theta\longmapsto i_\theta
\end{equation}
of graded $A$-bimodules, such that $\theta=i_\theta\circ\du$. In
particular, when $M=(A\otimes A)_\out$,
\begin{equation}\label{eq:double-dual-1}
\D_R{A}\lra{\cong}(\diff{A})^\vee, \quad
\Theta\longmapsto i_\Theta,
\end{equation}
where the graded $A$-bimodule of $R$-linear \emph{graded double
  derivations} on $A$ is 
\begin{equation}
\D_R{A}\defeq\Der_R(A,(A\otimes A)_\out)
\label{def-doble-deriv},
\end{equation}
and the graded $A$-bimodule structures on both $\D_R{A}$ and
$(\diff{A})^\vee$ come from the inner graded $A$-bimodule structure
$(A\otimes A)_\inn$ (see~\ref{eq:dual-bimodule}). If we want to
consider the outer $A$-bimodule structure instead, we can compose with
the graded flip isomorphism
\begin{equation}\label{eq:flip}
\sigma_{(12)}\colon (A\otimes A)_\out \lto (A\otimes A)_\inn\colon a'\otimes
a''\longmapsto (-1)^{\lvert a_1\rvert\lvert a_2\rvert}a''\otimes a',
\end{equation}
obtaining another canonical isomorphism (cf.~\cite[\S 5.3]{BCER12})
\begin{equation}
\D_R A \lra{\cong}\Hom_{A^\e}(\diff A,(A\otimes A)_\inn)
\colon\Theta \longmapsto \Theta^\vee=\sigma_{(12)}\circ i_\Theta,
\label{double universal property}
\end{equation}
where 
\[
\Theta^\vee \colon\diff A\lto (A\otimes A)_\inn
\colon
\alpha\longmapsto 
(-1)^{\lvert i'_\Theta(\alpha)\rvert\lvert i''_\Theta(\alpha)\rvert}i''_\Theta\alpha\otimes i'_\Theta\alpha.
\]
Following~\cite{CBEG07}, in this paper we will systematically use
symbolic Sweedler's notation, writing an element $x$ of $A\otimes A$
as $x'\otimes x''$, omitting the summation symbols. In particular, we
write $\Theta\colon A\to A\otimes A\colon a\mapsto
\Theta'(a)\otimes\Theta''(a)$ and $i_\Theta\colon\diff A\to A\otimes
A\colon \alpha\mapsto i_\Theta\alpha=i'_\Theta\alpha\otimes
i''_\Theta\alpha$.

\begin{example}
Consider the graded path algebra $\kk\PP$ of a graded quiver
$\PP$. This is a graded $R$-algebra, where $R=R_\PP$
(see~\secref{sub:grQ-pathAlg}). As in the ungraded case (see~\cite[\S
6]{VdB08}), the $\kk\PP$-bimodule of $R$-linear double derivations
$\D_R(\kk\PP)$ is generated by the set of double derivations
$\{\partial/\partial a\}_{a\in P_1}$, which on each arrow $b\in P_1$
act by the formula
\begin{equation}
\frac{\partial b}{\partial a} =
\begin{cases}
e_{h(a)}\otimes e_{t(a)} &\mbox{if } a=b, \\
0 & \text{otherwise}.
\end{cases}
\end{equation}
\label{derivaciones-quivers}
Note that by convention, we compose arrows from right to left
(see~\eqref{eq:path}), whereas Van den Bergh composes arrows from left
to right (see, e.g.,~\cite[Proposition 6.2.2]{VdB08}).
\end{example}

\subsubsection{Smooth graded algebras}

Let $A$ be a graded $R$-algebra. A finitely generated graded
$A$-bimodule $M$ is \emph{projective} if it is a projective object of
the abelian category $\mod(A)$ of finitely generated graded
$A$-bimodules, i.e. the functor $\Hom_{A^\e}^0(M,-)$ is exact on
$\mod(A)$.

\begin{definition}
The graded $R$-algebra $A$ is called \emph{smooth} over $R$ if it
is finitely generated over $R$ and the graded $A^\e$-module $\diff{A}$ is projective.
\end{definition}

Note that in the above definition, if the algebra $A$ is finitely
generated over $R$, then $\diff{A}$ is a finitely generated
$A^\e$-module.

\begin{lemma}[cf. {\cite[Proposition 5.3(3)]{CQ95}}]
\label{tensor smooth}
If an associative (ungraded) algebra $B$ is smooth over $R$ and $M$ is
a finitely generated and projective graded $B$-bimodule, then the
graded tensor algebra $A=\T_B M$ of $M$ over $B$ is also smooth over
$R$.
\end{lemma}

By the above lemma, graded path algebras of graded quivers are
prototypical examples of smooth graded algebras.

\subsubsection{The morphism $\mathtt{bidual}$}

Given a graded $R$-algebra $A$, 
the evaluation map gives a canonical $A$-bimodule morphism from any
graded $A$-bimodule $M$ into its double dual,
\[
\mathtt{bidual}_M\colon M\lto M^{\vee\vee}, 
\text{ where $M^{\vee\vee}\defeq(M^\vee)^\vee$}
\]
(see~\secref{eq:dual-bimodule}). This is an isomorphism when $M$ is a
finitely generated projective graded $A^\e$-module (cf.~\cite[\S
5.3]{CBEG07}).
In the special case $M=\diff{A}$, $\D_R{A}=(\diff{A})^\vee$
(see~\eqref{eq:double-dual-1}), so if $A$ is smooth over $R$, then
both $\diff{A}$ and $\D_R{A}$ are finitely generated and projective,
and the above morphism becomes a graded $A$-bimodule isomorphism
\begin{equation}
\label{bidual}
\mathtt{bidual}_{\diff{A}}\colon\diff{A}\lra{\cong}
(\D_R{A})^\vee=(\diff{A})^{\vee\vee}
\colon\quad \alpha\longmapsto i(\alpha)=\alpha^\vee,
\end{equation}
where
\[
i(\alpha)\colon\D_R{A}\lto(A\otimes A)_\out\colon \Theta\longmapsto i_\Theta\alpha.
\]

\subsection{Non-commutative differential calculus}

Using symbolic Sweedler's notation as
in~\secref{subsub:graded-double-Der}, any $\Theta\in\D_R A$ determines
a \emph{contraction operator}
\begin{equation}
\label{contraction-operator-double-deriv}
i_\Theta\colon\Omega^1_R A\lto A\otimes A
\colon\alpha\longmapsto i_\Theta\alpha=i'_\Theta\alpha\otimes i''_\Theta\alpha,
\end{equation}
that is determined by its values on generators, namely 
\[
i_\Theta(a)=0, \quad i_\Theta(\du b)=\Theta(b)=\Theta'(b)\otimes\Theta''(b),
\]
for all $a,b\in A$ (so $\du b\in\Omega^1_{R}A$). Since
$\Omega^{\bullet}_{R}A=\T_A\diff A$ in \eqref{NC-diff-forms} is the
free algebra of the graded $A$-bimodule $\diff A$, the graded
$A$-bimodule morphism $i_\Theta$ admits a unique extension to a
graded double derivation of bidegree $(\lvert\Theta\rvert,-1)$ on
$\Omega^{\bullet}_{R}A$,
\begin{equation}
i_{\Theta} \colon \Omega^{\bullet}_{R}A\lto\bigoplus
(\Omega^{i}_{R} A \otimes \Omega^{j}_{R} A)
\subset \Omega^{\bullet}_{R}A\otimes\Omega^{\bullet}_{R}A,
\label{extension-contraccion}
\end{equation}
where the direct sum is over pairs $(i,j)$ with $i+j=\bullet-1$, and
$\Omega^{\bullet}_{R}A\otimes\Omega^{\bullet}_{R}A$ is regarded a
graded $\Omega^{\bullet}_{R}A$-bimodule with respect to the outer
graded bimodule structure. Explicitly,
\begin{equation}
i_\Theta(\alpha_0\alpha_2\cdots\alpha_n)
=\sum_{k=0}^n
(-1)^{k}(\alpha_1\cdots\alpha_{k-1}(i^\prime_\Theta\alpha_k))\otimes((i^{\prime\prime}_\Theta\alpha_k)\alpha_{k+1}\cdots\alpha_n).
\label{formula-larga-contraccion}
\end{equation}
for all $\alpha_0,\ldots,\alpha_n\in\Omega^1_R A$ (see
\cite[(2.6.2)]{CBEG07}). Sometimes we will also need to view the
contraction operator $i_\Theta$ as a map $\Omega^{\bullet}_{R}A\to
(T_A(\Omega^\bullet_{R}A))^{\otimes 2}$, and then extend it further to
a graded double derivation of the tensor algebra
$T_A(\Omega^\bullet_{R}A)$.

The most interesting properties of the contraction operator are the
following.

 \begin{lemma}[{\cite[Lemma 2.6.3]{CBEG07}}]
Let $\Theta,\Delta\in\D_R A$, and $\alpha,\beta\in\Omega^\bullet_R A$. Then we have
\begin{enumerate}
\item [\textup{(i)}]
$i_\Theta(\alpha\beta)=i_\Theta(\alpha)\beta
+(-1)^{(\lvert\alpha\rvert,\lVert\alpha\rVert)\cdot(\lvert\Theta\rvert,\lVert\Theta\rVert)}
\alpha i_\Theta(\beta)$;
\item [\textup{(ii)}]
$i_\Theta\circ i_\Delta+i_\Delta\circ i_\Theta=0$.
\end{enumerate}
\label{i-leibniz}
 \end{lemma}

 Given $\Theta\in\D_RA$, the corresponding \emph{Lie derivative} is
 the graded double derivation
\[
L_\Theta\colon \Omega^\bullet_RA\lto
\bigoplus (\Omega^{i}_{R} A \otimes \Omega^{j}_{R} A) \subset \Omega^{\bullet}_{R}A\otimes\Omega^{\bullet}_{R}A,
\]
of bidegree $(\lvert\Theta\rvert,0)$, where the sum is over pairs
$(i,j)$ with $i+j=\bullet$, and
$\Omega^{\bullet}_{R}A\otimes\Omega^{\bullet}_{R}A$ is regarded a
graded $\Omega^{\bullet}_{R}A$-bimodule with respect to the outer
graded bimodule structure, that is is determined by its values on
generators by the following formulae for all $a,b\in A$:
\[
L_\Theta(a)=\Theta(a), \quad L_\Theta(\du b)=\du\Theta(b).
\]
By a simple calculation on generators, one obtains a Cartan formula
(see~\cite[(2.7.2)]{CBEG07}):
\begin{equation}
L_\Theta=\du\;\circ\; i_\Theta+i_\Theta\circ \du .
\label{Cartan-NC}
\end{equation}

Given a graded $R$-algebra $C$, we define a linear map
\[
C\otimes C\lto C\colon c=c_1\otimes c_2\longmapsto {}^\circ c\defeq(-1)^{\lvert c_1\rvert\lvert c_2\rvert} c_2 c_1.
\]
Similarly, given a linear map $\phi\colon C\lto C^{\otimes 2}$, we
define ${}^\circ \phi\,\colon \,C\lto C\colon$ $ c\longmapsto {}^\circ
(\phi(c))$. Applying this construction to $C=\Omega^{\bullet}_{R}A$,
we define for all $\Theta\in\D_RA$, the corresponding \emph{reduced
  contraction operator} and \emph{reduced Lie derivative},
\begin{equation}
\label{circulo derecha}
\iota_\Theta\, \colon \Omega^\bullet_R A \lto \Omega^\bullet_R A
\colon\alpha \longmapsto {}^\circ(i_\Theta\alpha)= 
(-1)^{(\lvert i'_\Theta\alpha\rvert,\lVert i'_\Theta\alpha\rVert)\cdot(\lvert i''_\Theta\alpha\rvert,\lVert i''_\Theta\alpha\rVert)}
i^{\prime\prime}_\Theta (\alpha) i^\prime_\Theta(\alpha),
\end{equation}
\begin{equation}
\label{reduced-Lie}
\mathcal{L}_{\Theta}\,\colon
\Omega^{\bullet}_{R}A\lto\Omega^{\bullet}_{R}A
\colon\alpha\longmapsto {}^\circ(L_\Theta\alpha),
\end{equation}
respectively. Explicitly, for all
$\alpha_0,\alpha_1,\ldots,
\alpha_n\in\Omega^1_R A$,
\begin{equation}
\iota_\Theta(\alpha_0\cdots\alpha_n)=\sum^n_{k=0}(-1)^{k(n-k)}
(i^{\prime\prime}_\Theta\alpha_k)\alpha_{k+1}\cdots\alpha_n\alpha_0\cdots\alpha_{k-1}(i^\prime_\Theta\alpha_k).
\label{reduced-contraction-general}
\end{equation}
Applying ${}^\circ(-)$ to~\eqref{Cartan-NC}, we obtain the
\emph{reduced Cartan identity} (cf.~\cite[Lemma 2.8.8]{CBEG07}):
\begin{equation}
\mathcal{L}_\Theta=\du\;\circ\;\iota_\Theta+\iota_\Theta\circ \du.
\label{reduced-Cartan}
\end{equation}

\section{Bi-symplectic tensor $\mathbb{N}$-algebras and doubled graded quivers}
\label{sec:N-algebra-assoc}

\subsection{Associative and tensor $\mathbb{N}$-algebras}

\begin{definition} 
Let $R$ be an associative algebra.
\begin{enumerate}
\item [\textup{(i)}]
An \emph{associative} \emph{$\mathbb{N}$-algebra} over $R$ (shorthand for
`non-negatively graded algebra') is a $\ZZ$-graded
associative $R$-algebra $A$ such that $A^i=0$ for all $i<0$. We say
$a\in A$ is \emph{homogeneous of weight} $\lvert a\rvert=i$ if $a\in
A^i$.
\item [\textup{(ii)}]
A \emph{tensor $\mathbb{N}$-algebra} over $R$ is an associative
$\mathbb{N}$-algebra $A$ over $R$ which can be written as a tensor
algebra $A=\T_{R}V$, for a positively graded $R$-bimodule $V$, so
$V=\bigoplus_{i\in\ZZ}V^i$, where $V^i=0$ for $i\leq 0$.
\item [\textup{(iii)}]
We say $v\in V$ is \emph{homogeneous of weight} $\lvert v\rvert=i$ if
$v\in V^i$.
\item [\textup{(iv)}]
The \emph{weight} of an associative $\mathbb{N}$-algebra $A$ is
$\lvert A\rvert\defeq{\displaystyle\min_{S\in\mathcal{G}}\max_{a\in S}\;\lvert a\rvert}$, where the
elements of $\mathcal{G}$ are the finite sets of homogeneous generators
of $A$.
\end{enumerate}
\label{def-tensor-N-algebra}
\end{definition}

\subsection{Double Poisson brackets on graded algebras}
\label{sub:DDP}
\subsubsection{Double Poisson brackets on graded algebras}

Commutative Poisson algebras appear in several geometric and algebraic
contexts. As a direct non-commutative generalization, one might
consider associative algebras that are at the same time Lie algebras
under a `Poisson bracket' $\{-,-\}$ satisfying the Leibniz rules
$\{ab,c\}=a\{b,c\}+\{a,c\}b$, $\{a,bc\}=b\{a,c\}+\{a,b\}c$. However
such Poisson brackets are the commutator brackets up to a scalar
multiple, provided the algebra is prime and not
commutative~\cite[Theorem 1.2]{FL98}. A way to resolve this apparent
lack of noncommutative Poisson algebras is provided by double Poisson
structures, introduced by Van den Bergh in \cite[\S 2.2]{VdB08}. In
this subsection, we will extend his definitions to graded associative
algebras (cf.~\cite{BCER12}).

Let $A$ be an associative $\mathbb{N}$-algebra over $R$, and $N\in\mathbb{Z}$. A \emph{double bracket of weight $N$} on $A$ is an $R$-bilinear map
\[
\lr{-,-}\colon A^p\otimes A^q\lto\bigoplus_{i+j=p+q+N} A^i\otimes A^j\subset A\otimes A,
\]
for any integers $p$ and $q$, which is a double $R$-derivation of weight $N$ (for the outer graded $A$-bimodule structure on $A\otimes A$) in its second argument, that is, 
\begin{equation}
\label{graded-Leibniz-Poisson}
\lr{a,bc}=\lr{a,b}c+(-1)^{\lvert a\rvert_N \lvert b\rvert}b\lr{a,c},
\end{equation}
for all homogeneous $a,b,c\in A$, and is $N$-graded skew-symmetric,
that is,
\begin{equation}
\lr{a,b}=-(-1)^{\lvert a\rvert_N\lvert b\rvert_N}\lr{b,a}^\circ,
\label{skew-symmetry-double}
\end{equation}
where $(u\otimes v)^\circ=(-1)^{\lvert u\rvert\lvert v\rvert}v\otimes
u$. Let $S_n$ be the group of permutations of $n$ elements. For
$a,b_1,...,b_n$ in $A$, and a permutation $s\in S_n$, we define
\[
\lr{a,b}_L\defeq\lr{a,b_1}\otimes b_2\otimes\cdots\otimes b_n,
\]
\begin{equation}\label{eq:permutation-tau}
\sigma_s(b)\defeq(-1)^t b_{s^{-1}(1)}\otimes\cdots\otimes b_{s^{-1}(n)},
\end{equation}
where $b=b_1\otimes\cdots\otimes b_n\in A^{\otimes n}$, and 
\[
t=\sum_{\substack{i<j\\s^{-1}(i)>s^{-1}(j)}}\lvert b_{s^{-1}(i)}\rvert\lvert b_{s^{-1}(j)}\rvert.
\]

A \emph{double Poisson bracket of weight $N$} on $A$ is a double
bracket $\lr{-,-}$ of weight $N$ on $A$ that satisfies the
\emph{graded double Jacobi identity}: for all homogeneous $a,b,c\in
A$,
\begin{equation}
\begin{aligned}
0=&\lr{a,\lr{b,c}}_L+(-1)^{\lvert a\rvert_N(\lvert b\rvert+\lvert
  c\rvert)}\sigma_{(123)}\lr{b,\lr{c,a}}_L\\
&+(-1)^{\lvert c\rvert_N(\lvert a\rvert+\lvert b\rvert)}\sigma_{(132)}\lr{c,\lr{a,b}}_L.
\end{aligned}
\label{graded-Jacobi-double}
\end{equation}
A \emph{double Poisson algebra of weight} $N$ is a pair $(A,\lr{-,-})$
consisting of a graded algebra and a double Poisson bracket of weight
$N$. Following \cite[\S 2.7]{VdB08}, double Poisson algebras of weight
-1 will be called \emph{double Gerstenhaber algebras}. Next, given a
double bracket $\lr{-,-}$, the \emph{bracket associated to} $\lr{-,-}$
is
\begin{equation}
\label{corchete asociado}
\{-,-\}\colon A\otimes A\lto A\colon\quad  (a,b)\longmapsto\{a,b\}\defeq m\circ\lr{a,b}=\lr{a,b}^\prime\lr{a,b}^{\prime\prime},
\end{equation}
where $m$ is the multiplication map. It is clear that $\{-,-\}$ is a
derivation in its second argument. Furthermore, it follows from
\eqref{skew-symmetry-double} that
\begin{equation}
\{a,b\}=-(-1)^{\lvert a\rvert_n\lvert b\rvert_n}\{b,a\}\,\,\text{mod}\,[A,A],
\label{skew-sym-assoc}
\end{equation}
A \emph{left Loday algebra} is a vector space $V$ equipped with a
bilinear operation $[-,-]$ such that the following Jacobi identity is
satisfied: $[a,[b,c]]=[[a,b],c]+[b,[a,c]]$, for all $a,b,c\in V$.

\begin{lemma}
\begin{enumerate}
\item [\textup{(i)}]
Let $A$ be a double Poisson algebra. Then
\begin{equation}
\{a,\lr{b,c}\}-\lr{\{a,b\},c}-\lr{b,\{a,c\}}=0,.
\label{mix-corchetes}
\end{equation}
\item [\textup{(ii)}]
Let $A$ be an associative $\mathbb{N}$-algebra endowed with a double Poisson bracket of weight $N$. Then the following identity holds:
\begin{equation}
(-1)^{\lvert a\rvert_N\lvert c\rvert_N}\{\{a,b\},c\}+(-1)^{\lvert b\rvert_N\lvert a\rvert_N}\{b,\{a,c\}\}+(-1)^{\lvert c\rvert_N\lvert b\rvert_N}\{a,\{b,c\}\}=0,
\label{LodayLoday}
\end{equation}
\end{enumerate}
\noindent where, in \eqref{mix-corchetes}, $\{a,-\}$ acts on tensors by $\{a,u\otimes v\}=\{a,u\}\otimes v+u\otimes\{a,v\}$, for all $a,u,v\in A$. In fact, $(A,\{-,-\})$ is a left Loday graded algebra.
\end{lemma}

\begin{proof}
  Immediate from~\cite[Proposition 2.4.2]{VdB08}. In fact,
  \eqref{LodayLoday} is \cite[Corollary 2.4.4]{VdB08}.
\end{proof}

Consider the bigraded algebra $\T_{A}\D_{R}A$ of $R$-linear
poly-vector fields on a finitely generated graded algebra $A$~\cite[\S
3]{VdB08}, with degree $d$ component
$(\T_{A}\D_{R}A)_d=(\D_{R}A)^{\otimes_A d}$. Then 
Van den Bergh~\cite[Proposition 4.1.1]{VdB08} constructs a map 
\begin{equation}
\label{mu}
\mu\colon P\longmapsto \lr{-,-}_P,
\end{equation}
from $(\T_{A}\D_{R}A)_{ 2}$ into the space of $R$-bilinear double
brackets on $A$, given by 
\begin{equation}
\lr{a,b}_P=(\Theta^\prime(a)*\Delta*\Theta^{\prime\prime}(a))(b)-(\Delta^\prime(a)*\Theta*\Delta^{\prime\prime}(a))(b),
\label{subQ}
\end{equation}
for all $P=\Theta\Delta$, with $\Theta,\Delta\in\D_R A$, and $a,b\in
A$. Furthermore, the map $\mu$ is isomorphism, provided $A$ is smooth
over $R$~\cite[Proposition 4.1.2]{VdB08}.




\begin{definition}[{\cite[Definition 4.4.1]{VdB08}}]
\label{DDP-def}
We say that $A$ is a \emph{differential double Poisson algebra} (a \emph{DDP} for short) over $R$ if it is equipped with an element $P\in (\T_A\D_RA)_2$ (a \emph{differential double Poisson bracket}) such that
\begin{equation}\label{DDP}
\{P,P\}=0\; \text{mod}\;[\T_A\D_RA,\T_A\D_RA].
\end{equation}
\end{definition}

Note that if $A$ is smooth over $R$, then the notions of differential
double Poisson algebra and double Poisson algebra coincide, because
$\mu$ in \eqref{mu} is an isomorphism in this case.

\begin{example}[{\cite[Theorem 6.3.1]{VdB08}}]
Let $A=k\dPP$ be the graded path algebra of a double graded quiver
$\dPP$. Then $A$ has the following differential double Poisson
bracket: 
\begin{equation}
P=\sum_{a\in\PP}  \frac{\partial}{\partial a}\frac{\partial}{\partial a^*}.
\label{P}
\end{equation}
\label{DDP-doubled-graded-quivers}
\end{example}

\subsubsection{The double Schouten--Nijenhuis bracket}
\label{sub:double-Schouten-Nijenhuis}

Suppose $A$ is a finitely generated graded $R$-algebra. Given 
homogeneous $\Theta,\Delta\in\D_R A$, a graded version of
\cite[Proposition 3.2.1]{VdB08} provides two graded derivations $A\to
A^{\otimes 3}$, for the graded outer structure on $A^{\otimes 3}$,
given by 
\begin{equation}
\begin{aligned}
\lr{\Theta,\Delta}^{\sim}_{l}&=(\Theta\otimes 1)\Delta-(-1)^{\lvert \Delta\rvert\lvert \Theta\rvert}(1\otimes \Delta)\Theta,
\\
\lr{\Theta,\Delta}^{\sim}_{r}&=(1\otimes \Theta)\Delta-(-1)^{\lvert \Delta\rvert\lvert \Theta\rvert}( \Delta\otimes 1)\Theta= -\lr{\Delta, \Theta}^{\sim}_{l}.
\end{aligned}
\label{tilde-brackets}
\end{equation}
Now, the graded $A$-bimodule isomorphisms
\[
\tau_{(12)}\colon A\otimes(A\otimes A)_\out\lra{\cong}A^{\otimes 3}
\colon a_1\otimes(a_2\otimes a_3)\longmapsto (-1)^{\lvert a_1\rvert\lvert a_2\rvert}a_2\otimes a_1\otimes a_3,
\]
\[
\tau_{(23)}\colon(A\otimes A)_\out\otimes A\lra{\cong}A^{\otimes 3}
\colon (a_1\otimes a_2)\otimes a_3\longmapsto (-1)^{\lvert a_2\rvert\lvert a_3\rvert}a_1\otimes a_3\otimes a_2,
\]
induce isomorphisms
\begin{equation}\label{eq:tau_12}
\begin{aligned}
\tau_{(12)}\colon 
\Der_R(A,&A^{\otimes 3})
\cong\Hom_{A^\e}(\diff{A},A^{\otimes 3})
\\&\lra{\cong}
\Hom_{A^\e}(\diff{A},A\otimes (A\otimes A)_\out)
\cong A\otimes\D_R{A}
\end{aligned}
\end{equation}
\begin{equation}\label{eq:tau_23}
\begin{aligned}
\tau_{(23)}\colon 
\Der_R(A,&A^{\otimes 3})\cong\Hom_{A^\e}(\diff{A},A^{\otimes 3})
\\&\lra{\cong} 
{\cong}\Hom_A(\diff{A},(A\otimes A)_\out\otimes A)
\cong \D_R A\otimes A,
\end{aligned}
\end{equation}
because $\diff{A}$ is finitely generated. Hence we can convert the
above triple derivations into 
\begin{align}\label{eq:doublebracket-l}
&
\lr{\Theta,\Delta}_l\defeq\tau_{(23)}\circ\lr{\Theta,\Delta}^\sim_l
\in\D_R A\otimes A,
\\&
\label{eq:doublebracket-r}
\lr{\Theta,\Delta}_r\defeq\tau_{(12)}\circ\lr{\Theta,\Delta}^\sim_r
\in A\otimes \D_R A, 
\end{align}
and hence make decompositions 
\begin{align*}
\lr{\Theta,\Delta}_l=\lr{\Theta,\Delta}^\prime_l\otimes \lr{\Theta,\Delta}^{\prime\prime}_l,
\\
\lr{\Theta,\Delta}_r=\lr{\Theta,\Delta}^\prime_r\otimes \lr{\Theta,\Delta}^{\prime\prime}_r,
\end{align*}
with $\lr{\Theta,\Delta}^{\prime\prime}_l,
\lr{\Theta,\Delta}^\prime_r\in A$,
$\lr{\Theta,\Delta}^{\prime\prime}_r,
\lr{\Theta,\Delta}^\prime_l\in\D_R A$. Using these constructions,
given homogeneous $a,b\in A$ and $\Theta, \Delta\in\D_R A$, we define
\begin{equation}\label{Gerstenhaber}
\begin{aligned}
\lr{a,b} &=0,
\\
\lr{\Theta,a}&=\Theta(a),
\\
\lr{\Theta,\Delta}&=\lr{\Theta,\Delta}_{l}+\lr{\Theta,\Delta}_{r},
\end{aligned}
\end{equation}
with the right-hand sides in~\eqref{Gerstenhaber} viewed as elements
of $(\T_A\D_RA)^{\otimes 2}$. Now, the \emph{graded double
  Schouten--Nijenhuis bracket} is the unique extension
\[
\lr{-,-}\colon (\T_A\D_R A)^{\otimes 2}\to (\T_A\D_R A)^{\otimes 2}
\]
of~\eqref{Gerstenhaber} of weight -1 to the tensor algebra $\T_A\D_R
A$ satisfying the graded Leibniz rule
\[
\lr{\Delta,\Theta\Phi}=(-1)^{(\lvert\Delta\rvert-1)\lvert\Theta\rvert}\Theta\lr{\Delta,\Phi}+\lr{\Delta,\Theta}\Phi,
\]
for homogeneous $\Delta,\Theta,\Phi\in\T_A\D_RA$.

\begin{example}[{\cite[Proposition 6.2.1]{VdB08}}]
The double Schouten--Nijenhuis bracket has very simple formulae for 
quiver path algebras. Let $A=\kk Q$ and $a,b\in Q_1$. Then
\begin{equation*}
\begin{aligned}
\lr{a,b}&=0
\\
\lr{\frac{\partial}{\partial a},b}&= \begin{cases} e_{h(a)}\otimes e_{t(a)} &\text{if } a=b,\\
0 & \text{otherwise},\end{cases} \label{Schouten quivers}
\\
\lr{\frac{\partial}{\partial a},\frac{\partial}{\partial b}}&=0.
\end{aligned}
\end{equation*}
\label{Schouten-Nijenhuis-quivers}
\end{example}

One of the most remarkable results in \cite{VdB08} is the following.

\begin{proposition}[{\cite[Theorem 3.2.2]{VdB08}}]
The tensor algebra $\T_A\D_R(A)$ together with the double
graded Schouten-Nijenhuis bracket $\lr{-,-}\colon
(\T_A\D_R(A))^{\otimes 2}\to (\T_A\D_R(A))^{\otimes 2}$ is a double
Gerstenhaber algebra.
\label{Gerstenhaber-str}
\end{proposition}

\subsection{Bi-symplectic tensor $\mathbb{N}$-algebras}

\subsubsection{Bi-symplectic tensor $\mathbb{N}$-algebras}

Let $B$ be an associative algebra and $A$ a tensor
$\mathbb{N}$-algebra over $B$.

\begin{definition}
\label{bi-sympldef}
An element $\omega\in \DR{2}(A)$ of weight $N$ which is closed for the
universal derivation $\du\,$ is a \emph{bi-symplectic form of weight
  $N$} if
the following map of graded $A$-bimodules is an isomorphism:
\[
\iota(\omega)\colon\D_R A\lra{\cong} \Omega^{1}_R(A)[-N]\colon \quad
\Theta \longmapsto \iota_{\Theta}\omega.
\]
\end{definition}

A tensor $\mathbb{N}$-algebra $(A,\omega)$ equipped with a
bi-symplectic form of weight $N$ is called a \emph{bi-symplectic
  tensor $\mathbb{N}$-algebra of weight} $N$ over $B$ if the
underlying tensor $\mathbb{N}$-algebra can be written as $A=\T_BM$,
for an $\NN$-graded $B$-bimodule $M=\bigoplus_{i\in\mathbb{N}}M^i$,
such that $M^i=0$ for $i > N$, and the underlying ungraded
$B$-bimodule of $M^i$ is finitely generated and projective, for all
$0\leq i\leq N$.

\subsubsection{Double Hamiltonian derivations}
\label{sub:double-Ham}

Following \cite{CBEG07,VdB15}, if $(A,\omega)$ is a bi-symplectic
tensor $\mathbb{N}$-algebra, we define the \emph{Hamiltonian double
  derivation} $H_a\in\D_R A$ corresponding to $a\in A$ via
 \begin{equation}
 \iota_{H_a}\omega=\du a,
 \label{hamiltonian-double-deriv}
 \end{equation}
 and write
 \begin{equation}
 \lr{a,b}_\omega=H_a(b)\in A\otimes A;
 \label{10.333}
 \end{equation}
 since $H_a(b)=i_{H_a}(\du b)$, we may write this expression as
 \begin{equation}
 \lr{a,b}_{\omega}=i_{H_a}\iota_{H_b}\omega,
 \label{determinacion del corchete doble}
 \end{equation}

\begin{lemma}
If $(A,\omega)$ is a bi-symplectic associative $\mathbb{N}$-algebra of weight $N$ over $R$, then $\lr{-,-}_{\omega}$ is a double Poisson bracket of weight $-N$ on $A$.
\label{bi-symp-double-Poisson}
\end{lemma}

\begin{proof}
This is a graded version of \cite[Proposition A.3.3]{VdB08}.
To determine the weight of $\lr{-,-}_{\omega}$, observe that by
\eqref{hamiltonian-double-deriv}, $\lvert
H_a\rvert+\lvert\omega\rvert=\lvert a\rvert$ and by \eqref{10.333},
$\lvert \lr{a,b}_\omega\rvert=\lvert a\rvert +\lvert b\rvert -\lvert
\omega\rvert$, so $\lvert\lr{-,-}_{\omega}\rvert=-N$.
\end{proof}

As in~\cite[\S 2.7]{CBEG07}, the grading of $A$ determines the
\emph{Euler derivation} $\Eu\colon A\lto A$, defined by
$\Eu\vert_{A_j}=j\cdot\Id$ for $j\in\NN$. The action of the
corresponding Lie derivative operator
 \begin{equation}
 \LL_\Eu\colon \DR{\bullet}(A)\to \DR{\bullet}(A)
 \label{Euler-derivative}
 \end{equation}
 has nonnegative integral eigenvalues. As usual all canonical objects
 (differential forms, double derivations, etc.) acquire weights by
 means of this operator, which will be denoted by $\lvert-\rvert$ and
 called \emph{weight} (e.g. a double derivation $\Theta\in \D_R A$ has
 weight $\lvert\Theta\rvert$).  Furthermore, if $\omega$ is a
 bi-symplectic form of weight $k$ on a graded $R$-algebra A, we say
 that a homogeneous double derivation $\Theta\in\D_R A$ is
 \emph{bi-ymplectic} if $\mathcal{L}_\Theta\omega=0$ where
 $\mathcal{L}_\Theta$ is the the reduced Lie
 derivative~\eqref{reduced-Lie}.

As in the commutative case, bi-symplectic forms of weight $k$ impose
strong constraints on the associative $\mathbb{N}$-algebra $A$.
\begin{lemma}
Let $\omega$ be a bi-symplectic form of weight $j\neq 0$ on an associative $\mathbb{N}$-algebra $A$ over $R$. Then
\begin{enumerate}
\item [\textup{(i)}]
$\omega$ is exact.
\item [\textup{(ii)}] if $\Theta$ is a bi-symplectic double derivation
  of weight $l$, and $j+l\neq 0$, then $\Theta$ is a Hamiltonian
  double derivation.
\end{enumerate}
\label{hamiltonian7}
\end{lemma}

\begin{proof}
For (i), note that $\LL_\Eu\omega=j\omega$, as $\omega$ has weight
$j$, so the Cartan identity implies $j\omega=\LL_\Eu\omega=\du
\iii_\Eu\omega$, because $\omega$ is closed, where $\iii_\Eu\colon
\DR{\bullet}(A)\to \DR{\bullet}(A)$. For (ii), we apply
\eqref{reduced-Cartan} to a bi-symplectic
double derivation $\Theta$, obtaining
\begin{equation}
\label{un comienzo}
0=\mathcal{L}_\Theta\omega=\du\iota_\Theta\omega,
\end{equation}
so defining $H\defeq\iii_\Eu\iota_\Theta\omega$, we conclude that
\[
\du H=\du\;(\iii_\Eu\iota_\Theta\omega)=\LL_\Eu(\iota_\Theta\omega)=\lvert\iota_\Theta\omega\rvert\iota_\Theta\omega=(l+j)\iota_\Theta\omega,
\]
where the second identity follows from~\eqref{un comienzo}. 
\end{proof}

The following result describes how the Hamiltonian double derivations
$H_a$ exchange double Poisson brackets and double Schouten--Nijenhuis
brackets.

\begin{lemma}[{\cite[Proposition 3.5.1]{VdB08}}]
The following are equivalent:
\begin{enumerate}
\item [\textup{(i)}]
$\lr{-,-}$ is a double Poisson bracket on $A$.
\item [\textup{(ii)}]
$\lr{H_{a},H_{b}}=H_{\lr{a,b}}$, for all $a,b\in A$. Here, $H_x\defeq H_{x'}\otimes x''+x'\otimes H_{x''}$ for all $x=x'\otimes x''\in A\otimes A$.
\end{enumerate}
\label{intercambio}
\end{lemma}

 \subsection{The canonical bi-symplectic form for a doubled graded quiver}
 \label{sub:bi-symplectic form for quivers}

\subsubsection{Casimir elements}
\label{sub:Casimir}

To avoid cumbersome signs, in this subsection we shall deal with a
finitely generated projective graded $(A^\e)^\op$-module $F$. Its
\emph{Casimir element} $\cas_F$ is defined as the pre-image of the
identity under the canonical isomorphism
\[
F \otimes_{(A^\e)^\op} F^\vee\lto \End_{(A^\e)^\op} F.
\]
Note that $F^\vee=\Hom_{A^\e}(F,A^{\e}_{A^{\e}})$ is equipped with the graded $A$-bimodule structure induced by the outer $A$-bimodule structure on $A\otimes A$. In the following result, we determine the Casimir element for a graded quiver:
\begin{lemma}
Let $\PP$ be a graded quiver, with graded path algebra $\kk\PP=T_R
V_{\PP}$. For a homogeneous $b\in\P1$, we define $\widetilde{a}\in
V^\vee_{\PP}$ by
 \begin{equation}
\widetilde{a}(b)=\begin{cases}
e_{h(a)}\otimes e_{t(a)} &\text{if } a=b,\\
0 & \text{otherwise}.
\end{cases}.
\label{elementos-tilda}
\end{equation}
Then $\cas_{V_\PP}=\sum_{a\in P_1}\widetilde{a}\otimes a$ is the element Casimir for the $(R^\e)^\op$-module $V_P$.
\label{Casimir}
\end{lemma}

\begin{proof}
Observe that $V_P$ is an $(R^\e)^\op$-module. Since, by convention, we
compose arrows from right to left, we can check that
$\texttt{eval}\left(\sum_{a\in P_1}a\otimes \widetilde{a}\right)(b)=b$
for all homogeneous $b\in \PP_1$;
\[
\texttt{eval}\left(\sum_{a\in P_1}a\otimes \widetilde{a}\right)(b)=\sum_{a\in P_1}a*\widetilde{a}(b)=e_{h(b)}\;b\;e_{t(b)}=b.\qedhere
\]
\end{proof}

Let $\PP$ be a graded quiver, with graded path algebra $\kk\PP$ as
above. Then the inverse of the canonical map $\texttt{eval}\colon
V^\vee_P\otimes_{R^{\e}} A^\e\lto\Hom_{A^\e}(V_P,{}_{A^\e}A^\e)$
 is given by
\begin{equation}
\begin{aligned}
\kappa\colon \Hom_{A^\e}(V_P,{}_{A^\e}A^\e)&\lto V^\vee_P\otimes_{R^{\e}}A^\e
\\
g&\longmapsto\sum_{a\in P_1}\widetilde{a}\otimes g(a)=\sum_{a\in P_1}(-1)^\square g^{\prime\prime}(a)\otimes \widetilde{a}\otimes g^{\prime}(a),
\label{kappa}
\end{aligned}
\end{equation}
where now $V_\PP$ is viewed as an $R^\e$-module, we use the
isomorphism $V^\vee_P\otimes_{R^\e}A^\e\cong A\otimes_R
V^\vee_P\otimes_R A$, $g(a)=g^\prime(a)\otimes
(g^{\prime\prime})^\op(a)\in A^\e$ for $a\in P_1$, and
$(-1)^\square=(-1)^{\lvert g^{\prime\prime}(a) \rvert(\lvert
  g^{\prime}(a)\rvert + N- \lvert a\rvert )}$.

\subsubsection{Duals and biduals}

Let $\dPP$ be the weight $N$ double graded quiver of a graded quiver
$\PP$, $R=R_\dPP$, and $A=\kk\dPP$ its graded path algebra.
Since $R$ is a finite-dimensional semisimple algebra over $\kk$
(see~\secref{sub:basics quivers}), it is well known that there are
four sensitive ways of defining the dual of an $R$-bimodule, but that
all of them can be identified by fixing a \emph{trace} on $R$, that
is, a $\kk$-linear map $\Tr\colon R\to\kk$ such that the bilinear form
$R\otimes R\to\kk\colon (a,b)\mapsto \Tr(ab)$ is symmetric and
non-degenerate.
More precisely, let $V$ be an $R$-bimodule, $V^*\defeq\Hom(V,\kk)$ and
$ V^\vee\defeq\Hom(V,R\otimes R)$. Then we can use $\Tr\colon R\to
\kk$ to construct an isomorphism $B\colon V^*\to V^\vee$, by the
following formula, for all $\psi\in V^*, v\in V$:
\begin{equation}
\psi(v)=\Tr((B(\psi)^\prime)(v))\Tr((B(\psi)^{\prime\prime})(v)),
\label{isom-identif-usando-trazas}
\end{equation}

Consider the graded $R$-bimodule $V_\dPP$ as a vector space, and the
space of linear forms $V^*_\dPP\defeq\Hom(V_\dPP,\kk)$.
Then $V_\dPP$ has a basis $\{a\}_{a\in\dPP_1}$ consisting of all the
arrows of $\dPP$. Let $\{\widehat{a}\}_{a\in\dPP_1} \subset V^*_\dPP$
be its dual basis. Given $b\in A$, we have
\begin{equation}
\widehat{a}( b)=\delta_{ab} =\begin{cases}
1 &\mbox{if } a=b \\
0 & \mbox{otherwise}
\end{cases} =\Tr(\delta_{ab}e_{h(a)})\Tr(e_{t(a)})=\Tr((\widetilde{a})^\prime(b))\Tr((\widetilde{a})^{\prime\prime}(b)),
\label{segunda-parte-de-identific-isom-trazas}
\end{equation}
with $\{\widetilde{a}\}_{a\in \dPP_1}$ as in Lemma
\ref{Casimir}. Furthermore, \eqref{isom-identif-usando-trazas} applied
to $\{\widehat{a}\}_{a\in \PP_1}$ gives
\begin{equation}
\widehat{a}( b)=\Tr\left(B(\hat{a})^\prime (b)\right)\Tr\left(B(\hat{a})^{\prime\prime} (b)\right).
\label{otrat-otra-parte-del-isom-trazas}
\end{equation}
Comparing \eqref{segunda-parte-de-identific-isom-trazas} and
\eqref{otrat-otra-parte-del-isom-trazas}, it follows that
$B(\widehat{a})=\widetilde{a}$, that is, 
\begin{equation}
B^{-1}(\widetilde{a})=\widehat{a}.
\label{identidad-B-1}
\end{equation}
Using \eqref{varepsilon}, we define
\[
\langle-,-\rangle\colon V_\dPP\times V_\dPP\lto\kk
\colon (a,b)\longmapsto \langle a,b\rangle 
=\begin{cases}
\varepsilon(a) &\text{if } a=b^* \\ 0 & \text{otherwise}
\end{cases}
= \begin{cases}
1 &\text{if } a=b^*\in \PP_1, \\
-1 & \text{if } a=b^*\in P^*_1, \\
0 & \text{otherwise}.
\end{cases}
\]

It is not difficult to see that $(V_\dPP,\langle-,-\rangle)$ is a
graded symplectic vector space of weight $N$. Moreover, the graded
symplectic form $\langle-,-\rangle$ determines an isomorphism of
graded $R$-bimodules, for the canonical graded $R$-bimodule structure
on $V_\dPP$ and the induced one on its dual. As in \cite{CBEG07}, we
define an isomorphism, where $V_\dPP$ is regarded as a vector space:
\begin{equation}
\#\colon V^*_\dPP\lra{\cong}V_\dPP[-N]\colon\quad  \widehat{a}\longmapsto \varepsilon (a)a^*
\label{sostenido def}
\end{equation}


\subsection{The canonical bi-symplectic form for a double graded  quiver}
\begin{proposition}
\label{bi-symplectic-form-double-quivers}
Let $R=R_\dPP$, $A=\kk\dPP$, and 
\begin{equation}
\label{bi-symplectic-gaded-canonical}
\omega\defeq \sum_{a\in \PP_1}\du a\du a^*\in \DR{2}A.
\end{equation}
Then $\omega$ is bi-symplectic of weight $N$.
\end{proposition}

\begin{proof}
This result is a routine graded generalization of the ungraded
statement~\cite[Proposition 8.1.1\textup{(ii)}]{CBEG07}. We omit the
proof.
\end{proof}

\section{Restriction theorems of graded bi-symplectic forms}
\label{sec:Darboux}

In this section, we prove two technical results of graded
bi-symplectic forms, roughly speaking corresponding to graded
non-commutative versions, in weights 1 and 2, of the Darboux Theorem
in symplectic geometry. Furthermore, we will describe in
\secref{sub:cotangent} a noncommutative analogue of the cotangent
exact sequence relating relative and absolute differential forms,
focusing on the case of bi-symplectic tensor $\mathbb{N}$-algebras.


\subsection{The cotangent exact sequence}
\label{sub:cotangent}

Let $R$ be a smooth semisimple associative $\kk$-algebra, $B$ a smooth
graded $R$-algebra, $A=\T_B M$ the tensor algebra of a graded
$B$-bimodule $M$, and $\omega\in \DR{2}(A)_N$ a bi-symplectic form of
weight $N$ on $A$, where $N\in\NN$.

The cotangent exact sequence for an arbitrary graded associative
$B$-algebra is as follows.
\begin{lemma}
\begin{enumerate}
\item [\textup{(i)}]
There is a canonical exact sequence of graded $A$-bimodules
\[
0\to \Tor^{B}_1(A,A)\lto A\otimes_B\Omega_{R}^{1} B \otimes_B A\lto
\Omega^{1}_{R} A\lto \Omega^1_B A  \to 0.
\]
\item [\textup{(ii)}]
Suppose $A=\T_{B} M$, where $M$ is a graded $B$-bimodule which is flat
as either left or right graded $B$-module. Then there is an exact
sequence of graded $A$-bimodules
\[
0\to A\otimes_B\Omega_{R}^{1} B\otimes_B A\lto \Omega^{1}_{R} A\lto A\otimes_B M\otimes_B A\to 0.
\]
\end{enumerate}
\end{lemma}

\begin{proof}
This is a consequence of \cite[Proposition 2.6 and Corollary
2.10]{CQ95}, because the maps involved preserve weights.
\end{proof}


We will now use an explicit description of the space of noncommutative
relative differential forms on $A$ over $R$, following \cite[\S
5.2]{CBEG07}. Define the graded $A$-bimodule
\begin{equation}
\label{omega-tilde}
\widetilde{\Omega}\defeq(A\otimes_B\Omega^{1}_{R} B\otimes_{B}A)\bigoplus(A\otimes_R M\otimes_R A).
\end{equation}
Abusing the notation, for any $a^\prime, a^{\prime\prime}\in A$, $m\in
M$, $\beta\in \Omega^{1}_{R} B$, we write
\begin{equation}
\label{abuso de notacion}
\begin{aligned}
a^{\prime}\cdot\widetilde{m}\cdot a^{\prime\prime}:&=0\oplus(a^{\prime}\otimes m\otimes a^{\prime\prime})\in A\otimes_R M\otimes_R A\subset \widetilde{\Omega},\nonumber
\\
a^{\prime}\cdot\widetilde{\beta}\cdot a^{\prime\prime}:&=(a^{\prime}\otimes \beta\otimes a^{\prime\prime})\oplus 0\in A\otimes_B\Omega^{1}_{R} B\otimes_B A\subset \widetilde{\Omega}.
\end{aligned}
\end{equation}
Let $Q\subset\widetilde{\Omega}$ be the graded $A$-subbimodule
generated by the Leibniz rule in $\widetilde{\Omega}$, that is,
\begin{equation}
Q=\langle\!\langle\widetilde{b^{\prime}mb^{\prime\prime}}-\widetilde{db^{\prime}}\cdot(mb^{\prime\prime})-b^{\prime}\cdot \widetilde{m}\cdot b^{\prime\prime}-(b^{\prime}m)\cdot \widetilde{db^{\prime\prime}}\rangle\!\rangle_{b',b''\in B, m\in M},
\label{Q}
\end{equation}
where $\langle\!\langle-\rangle\!\rangle$ denotes the graded
$A$-subbimodule generated by the set $(-)$. 

The graded algebra structure of $A=\T_B M$ induces a graded
$A$-bimodule structure on $\widetilde{\Omega}$. Then
$Q\subset\widetilde{\Omega}$ is a graded $A$-subbimodule, because it
is generated by homogeneous elements, so the quotient
$\widetilde{\Omega}/Q$ is a graded $A$-bimodule. The following result
follows from \cite[Lemma 5.2.3]{CBEG07}, simply because weights are
preserved.

\begin{proposition}
Let $B$ be a smooth graded $R$-algebra, $M$ a finitely generated projective graded $B$-bimodule, and $A=\T_B M$. Then
\begin{enumerate}
\item [\textup{(i)}]
There exists a graded $A$-bimodule isomorphism
\[
f\colon\Omega^{1}_R A\lra{\cong} \tilde{\Omega}/Q .
\]
\item [\textup{(ii)}]
The embedding of the first direct summand in $\widetilde{\Omega}$ (respectively, the projection onto the second direct summand in $\widetilde{\Omega}$), induces, via the isomorphism in \textup{(i)}, a canonical extension of graded $A$-bimodules
\begin{equation}
\label{cotangent}
0\to A\otimes_B\Omega^{1}_R B\otimes_B A\lra{\varepsilon} \Omega^1_R A
\lra{\nu} A\otimes_B M\otimes_B A\lto 0
\end{equation}

\item [\textup{(iii)}]
The assignment $B\oplus M=\T^0_BM\oplus \T^1_B M\to\widetilde{\Omega}$, $b\oplus m\mapsto\widetilde{\du b}+\widetilde{m}$ extends uniquely to a graded derivation $\widetilde{\du}\;\colon A=\T_B M\to\widetilde{\Omega}/Q$; this graded derivation corresponds, via the isomorphism in \textup{(i)}, to the canonical universal graded derivation $\du\;\colon A\to\Omega^{1}_R A$. In other words, we have (see \eqref{abuso de notacion})
\begin{equation}
f(\widetilde{m})=\du m,\quad f(\widetilde{\du b})=\du b,
\label{f}
\end{equation}
for homogeneous $m\in M$ and $b\in B$, and the commutative diagram
\[
\xymatrix{
& A\ar[dl]_-{\widetilde{\du}}\ar[dr]^-{\du}&
\\
\tilde{\Omega}/Q\ar[rr]^-{f}_-{\cong}& & \Omega^1_R A
}
\]
\end{enumerate}
\label{5.2.3}
\end{proposition}


Applying the functor $\Hom_{A^\e}(-{}_{A^\e}A^{\e})$ to
\eqref{cotangent}, we obtain the ``tangent exact sequence".
\begin{lemma}
\begin{enumerate}
\item [\textup{(i)}]
Let $B$ be a smooth graded $R$-algebra, $M$ a finitely generated
projective graded $B$-bimodule and $A=\T_B M$. Then there is a short exact sequence
\begin{equation}
\label{dual cotangent}
0\to A\otimes_B M^{\vee}\otimes_B A \lra{\nu^\vee}
 \D_R A\lra{\varepsilon^\vee} A\otimes_B\D_R B\otimes_B A \lto 0.
\end{equation}
\item [\textup{(ii)}]
If, in addition, $A$ is endowed with a bi-symplectic form of weight $N$, then the following diagram, where the rows are short exact sequences, commutes.
\begin{equation}
\xymatrix{
0\ar[r] & A\otimes_{B} M^{\vee}\otimes_{B}A \ar[r]^-{\nu^\vee} \ar[d]& \D_{R} A\ar[r]^-{\varepsilon^\vee} \ar[d]^{\iota(\omega)}&  A\otimes_{B} \D_{R} B \otimes_{B}A \ar[r] \ar[d]& 0
\\
0\ar[r] &A\otimes_{B} \Omega^{1}_{R} B \otimes_{B}A \ar[r]^-{\varepsilon}  \ar[r] & \Omega^{1}_{R} A \ar[r]^-{\nu} & A\otimes_{B} M\otimes_{B}A\ar[r] & 0
}
\label{DIAGRAMA}
\end{equation}
\end{enumerate}
\label{DIAGRAMAGRAL}
\end{lemma}

\begin{proof}
This result is a graded version of \cite[Lemma 5.4.2]{CBEG07}.
\end{proof}

\subsection{Restriction Theorem in weight 0}
\label{sub: restriction 0}

\begin{theorem}
\label{weight 0}
Let $R$ be a semisimple finite-dimensional $\kk$-algebra, $B$ a smooth
associative $R$-algebra, and $E_1,\ldots,E_{N}$ finitely generated projective $B$-bimodules, where $N>0$. Define the tensor $\mathbb{N}$-algebra $A=\T_BM$ as the tensor $B$-algebra of the graded $B$-bimodule
\[
M\defeq M_1\oplus\cdots\oplus M_N,
\]
where $M_i\defeq E_i[-i]$, for $i=1,...,N$. Let $\omega\in\DR{2} (A)$ be a bi-symplectic form of weight $N$ over $A$. Then, the isomorphism $\iota(\omega)\colon \D_{R} A\lra{\cong}\Omega^{1}_R A[-N]$ induces another isomorphism
\[
\widetilde{\iota}(\omega)\colon A\otimes_B\D_{R} B\otimes_B A\lra{\cong} A\otimes_B M_{N}\otimes_B A,
\]
which, in weight zero, restricts to the following isomorphism:
\[
\widetilde{\iota}(\omega)_{(0)}\colon \D_{R} B \lra{\cong} E_{N}.
\]
\label{weight 0}
\end{theorem}

The technical proof of this result is given in Appendix A.

\subsection{Restriction Theorem in weight 1 for doubled graded quivers}
\label{sub: restriction-1}

For convenience, we fix the following. 

\begin{framework}
\label{framework-quivers}
Let $\dPP$ be a doubled graded quiver of weight 2, with graded path
algebra $A\defeq\kk\dPP$. Let $R=R_\dPP$ be the semisimple finite
dimensional algebra with basis the trivial paths in $P$, and let $B$
be the smooth path algebra of the weight 0 subquiver of $\dPP$.  Let
$\omega$ be the canonical bi-symplectic form $\omega\in\DR{2}(A)$ of
weight $2$ on $A$. Then $A=\T_BM$, where $M$ is the graded
$B$-bimodule $M\defeq E_1[-1]\oplus E_2[-2]$, for finitely generated
projective $B$-bimodules $E_1$ and $E_2$.
\end{framework}

\begin{theorem}
\label{Theorem in weight 1}
In the Framework \ref{framework-quivers}, the isomorphism
$\iota(\omega)\colon \D_R A\lto \Omega^1_R A[-2]$ restricts, in weight
1, to a $B$-bimodule isomorphism
\begin{equation}
\label{Caligraphic R}
(\iota(\omega))_1\colon E^\vee_1\lra{\cong} E_1\colon \quad\widetilde{a}\longmapsto\varepsilon (a)a^*,
\end{equation}
with inverse
\begin{equation}
\label{flat}
\flat\colon E_1\lra{\cong} E^\vee_1\colon \quad a\longmapsto \varepsilon(a)\widetilde{a^*}.
\end{equation}
\end{theorem}

\begin{proof}
By Lemma \ref{lem:grPathAlg-TensorAlg} and \eqref{eq:grPathAlg.2}, we
know that $A$ can be identified both with $\T_RV_\dPP$ and
$\T_BM_\dPP$.  We will use the following isomorphisms from the proof
of~\cite[Proposition 8.1.1]{CBEG07}, or its graded generalization,
namely, Theorem~\ref{bi-symplectic-form-double-quivers}:
\begin{equation}
\label{G}
G\colon A\otimes_R V_\dPP\otimes_R A\lra{\cong} \diff A
\colon\sum_{a\in\dPP_1}f_a\otimes a\otimes g_a\longmapsto  \sum_{a\in\dPP_1}f_a\du a\; g_a,
\end{equation}
\begin{equation}
\label{H}
H\colon \D_R A \lra{\cong} A\otimes_R V^*_\dPP\otimes_R A
\colon \Theta \longmapsto\sum_{a\in\overline{P}_1}(-1)^{\square}\Theta^{\prime\prime}(a)\otimes \widehat{a}\otimes\Theta^\prime(a),
\end{equation}
where $(-1)^{\square}$ is given by the Koszul sign rule. 
To shorten notation,
\[
V_1\defeq {(V_\dPP)}_1,
\quad 
V^\vee_1\defeq\Hom_{R^\e}(V_1,{}_{R^\e}R^\e), 
\quad 
M_w\defeq {(M_\dPP)}_w, \quad\text{ for $w>0$},
\]
where the subindexes mean weights. Using \eqref{sostenido
  def}, \eqref{G} and \eqref{H}, we consider the following commutative
diagram (some arrows will be constructed below).
\begin{equation}
\xymatrix{
A\otimes_B M^\vee_1\otimes_B A \ar@{^{(}->}[r]^J \ar[dr]^{h}_\cong &\D_R A\ar[r]^{\iota(\omega)}_{\cong}\ar[d]_H^\cong &\diff A \ar@{->>}[r]^-P & A\otimes_B M_1\otimes_ B A
\\
&A\otimes_R V^*_\dPP\otimes_R  A\ar[r]^{\Id\otimes \#\otimes\Id}_{\cong} &A\otimes_R V_\dPP\otimes_R A \ar@{->>}[r]^-{\text{proj}} \ar[u]^G_\cong& A\otimes_R V_1\otimes_R A\ar[u]^{g}_\cong
}
\label{diagrama peso 1}
\end{equation}

\begin{claim}
\label{claim-M-V}
If $a\in \PP_1$, we have $A\otimes_B M_1\otimes_B A\lra{\cong} A\otimes_R V_1\otimes_R A$, and consequently
\[
T\colon A\otimes_B M^\vee_1\otimes_B A\lra{\cong} A\otimes_R V^\vee_1\otimes_R A.
\]
\end{claim}
\begin{proof}[Proof of Claim~\ref{claim-M-V}]
This follows simply because
\begin{equation}
\begin{aligned}
 \label{isomorfismoVM}
A\otimes_B M_1\otimes_B A&=\bigoplus_{\lvert a\rvert=1}A\otimes_B BaB\otimes_B A
\cong\bigoplus_{\lvert a\rvert=1}AaA
\\&
\cong\bigoplus_{\lvert
  a\rvert=1}A\otimes_R\kk a\otimes_R  A
=A\otimes_R V_1\otimes_R A.
\qedhere
\end{aligned}
 \end{equation}
\end{proof}
Using Claim \ref{claim-M-V}, we can construct the left-hand triangle
in \eqref{diagrama peso 1}. Let $a\in\dPP_1$ with $\lvert
a\rvert=1$. Then consider $\widetilde{a}\in M^\vee_1$ (as defined in
\eqref{elementos-tilda}) and $1\otimes \widetilde{a}\otimes 1\in
A\otimes_B M^\vee_1\otimes_B A$. By Claim \ref{claim-M-V}, and the
natural injection, this element can be viewed in $A\otimes_R
V^\vee_{\dPP}\otimes_R A\cong V^\vee_\dPP\otimes_{R^\e}A^\e$. Now, to
define $J$, we use the sequence of isomorphisms
\begin{equation}
\begin{aligned}
V^\vee_\dPP\otimes_{R^\e}A^\e &\cong \Hom_{R^\e}(V_\dPP,\Hom_{A^\e}(A^\e,A^\e))
\\
&\cong \Hom_{R^\e}(A^\e\otimes_{R^\e}V_\dPP,{}_{A^\e}A^\e)
\\
&=\Hom_{A^\e}(\diff A, {}_{A^\e}A^\e)=\D_RA.
\label{more-isomorphisms}
\end{aligned}
\end{equation}
Note that $h=H\circ J$ is given by 
\begin{equation}
\xymatrix{
h\colon A\otimes_B M^\vee_1\otimes_B A\ar[r]^-J &\D_R A\ar[r]^-H & A\otimes_R V^*_\dPP\otimes_R A
\\
1\otimes\widetilde{a}\otimes 1  \ar@{|->}[r] & \frac{\partial}{\partial a}  \ar@{|->}[r] & e_{t(a)}\otimes\widetilde{a}\otimes e_{h(a)}.
}
\label{aplicacion-h-compos-J-H}
\end{equation}
Next, we will focus on the right-hand square of \eqref{diagrama peso
  1}, where $\text{proj}$ is the canonical projection, and construct
$g$ so that it is commutative. Let $qap$ be a generator of $A\otimes_R
V_\dPP\otimes_R A$, i.e., $a\in \P1$ is such that $\lvert a\rvert=1$,
and $p,q$ are paths in $\PP$, such that $h(p)=t(a), h(a)=t(q)$. Then
$\text{proj}\vert_{A\otimes_R V_1\otimes_R A}=\Id$, and
$P=\texttt{pr}\circ\nu\circ f$ by Lemma \ref{5.2.3}, where
$\texttt{pr}\colon M=\bigoplus_{w>0}M_w\to M_1$ is the natural
projection. Hence
\begin{equation}
\begin{aligned}
(\texttt{pr}\circ \nu\circ f\circ G)(q\otimes a\otimes p)&=(\texttt{pr}\circ \nu\circ f)(q\du a\;p)
\\
&=(\texttt{pr}\circ \nu)((0\oplus (q\otimes a\otimes p))\;\text{mod}\; Q)
\\
&=\texttt{pr}(q\otimes a\otimes p)=q\otimes a\otimes p,
\end{aligned}
 \label{h prima}
 \end{equation}
so the isomorphism $g$ restricts to the isomorphism $G$, because 
\begin{equation}
g\colon A\otimes_R V_1\otimes_R A\lra{\cong} A\otimes_B M_1\otimes_B A\colon \quad q\otimes a\otimes p\longmapsto q\otimes a\otimes p.
\label{G pequeno}
\end{equation}
Therefore, by \eqref{aplicacion-h-compos-J-H}, \eqref{sostenido def} and \eqref{G pequeno}, we have
\begin{equation}
\begin{aligned}
(g\circ p)\circ(\Id\otimes\sharp\otimes \Id)\circ h\colon A\otimes_B M^\vee_1\otimes_B A&\lto A\otimes_B M_1\otimes_B A
\\
 1\otimes\widetilde{a}\otimes 1&\longmapsto e_{t(a)}\otimes\varepsilon(a)a^*\otimes e_{h(a)}.
\end{aligned}
\label{final-restriction-weight1-morphisms}
\end{equation}
Let $(-)_w$ mean the component of weight $w\in\mathbb{Z}$. Then
$(A\otimes_B M_1\otimes_B A)_1\cong B\otimes_B M_1\otimes_B B\cong
M_1$. Furthermore, since $\omega$ is a bi-symplectic form of weight 2,
\eqref{final-restriction-weight1-morphisms} has weight -2, so
$(A\otimes_B M^\vee_1\otimes_B A)_{-1}=B\otimes_B M^\vee_1\otimes_B
A\cong M^\vee_1$. Therefore, we obtain the following isomorphism of
$B$-bimodules:
\begin{equation}
\label{Caligraphic R1}
\mathcal{R}\colon E^\vee_1\lra{\cong} E_1\colon \quad\widetilde{a}\longmapsto\varepsilon (a)a^*,
\end{equation}
with inverse
\begin{equation}
\label{flat1}
\flat\colon E_1\lra{\cong} E^\vee_1\colon \quad a\longmapsto \varepsilon(a)\widetilde{a^*}.
\qedhere
\end{equation}
\end{proof}

\section{Bi-symplectic $\mathbb{N}$-algebras of weight 2}
\label{sec:bisympl-weight-2}

\subsection{The graded algebra $A$}
\label{sub: bracket-poisson}


For convenience, we introduce the following.

\begin{framework}
Let $R$ be a semisimple associative algebra, $B$ a smooth associative
$R$-algebra, and $E_1$ and $E_2$ projective finitely generated
$B$-bimodules. Let
\[
A:=\T_BM
\]
be the graded tensor $\mathbb{N}$-algebra of the graded $B$-bimodule
\[
M\defeq E_1[-1]\oplus E_2[-2].
\]
Let $\omega\in \DR{2}(A)$ be a bi-symplectic form of weight $2$. Thus
the pair $(A,\omega)$ is a bi-symplectic tensor $\mathbb{N}$-algebra
of weight 2 (see Definition \ref{bi-sympldef}).
\label{framework-general}
\end{framework}

In this framework, we have
\[
A=\bigoplus_{n\in\mathbb{N}}A^n,
\]
where
\begin{equation}
A^0=B,\quad A^1=E_1,\quad A^{2}=E_1\otimes_B E_1\oplus E_2.
\label{decomposition-A-NC}
\end{equation}
By Lemma \ref{bi-symp-double-Poisson}, the bi-symplectic form $\omega$
on $A$ determines a double Poisson bracket $\lr{-,-}_\omega$ of weight
-2. This bracket satisfies the following relations:
\begin{equation}
\begin{aligned}
\lr{A^{0}, A^{0}}_{\omega}&=\lr{A^{0},A^{1}}_{\omega}=0,
\\
 \lr{A^{1},A^{1}}_{\omega}&\subset (A\otimes A)_{(0)}=B\otimes B ,
\\
 \lr{A^{2}, A^{0}}_{\omega}&\subset (A\otimes A)_{(0)}=B\otimes B,
\\
 \lr{A^{2},A^{1}}_{\omega}&\subset(A\otimes A)_{(1)}= (E_1\otimes B)\oplus (B\otimes E_1),
 \\
  \lr{A^{2},A^{2}}_{\omega}&\subset(A\otimes A)_{(2)}.
 \end{aligned}
 \label{corchete-doble-relaciones}
 \end{equation}


\subsection{The pairing}
\label{sub:pairing}

\subsubsection{A family of double derivations}

By \eqref{decomposition-A-NC}, \eqref{corchete-doble-relaciones}, $\lr{A^2, B}_{\omega}\subset B\otimes B$, so we can define 
\[
\mathbb{X}_a:=\lr{a,-}_{\omega}\vert_B\colon B\lto B\otimes B.
\]
for all $a\in A^2$. 
Since $\lr{-,-}_{\omega}$ is a double Poisson bracket, in particular,
it satisfies the graded Leibniz rule in its second argument (with
respect to the outer structure), and so
\[
\mathbb{X}_a(b_1b_2)=\lr{a,b_1b_2}_{\omega}=b_1 \lr{a,b_2}_{\omega}+\lr{a,b_1}_\omega b_2= b_1\mathbb{X}_a(b_2)+\mathbb{X}_a(b_1) b_2,
\]
for all $a\in A^{2}$, and $b_1,b_2\in B$. Therefore
$\mathbb{X}_a\in\D_R B$, and we construct a `family of double
derivations' parametrized by $A^2$, namely,
\begin{equation}
\mathbb{X}\colon A^2 \lto\D_R B\colon\quad a  \longmapsto\mathbb{X}_a:=\lr{a,-}_{\omega}\vert_B.
\label{XX}
\end{equation}

\subsubsection{A family of double differential operators}
\label{sub:double-diff-cov-op}

By \eqref{corchete-doble-relaciones}, $\lr{A^2,A^1}_{\omega}\subset
(A\otimes A)_{(1)}=E_1\otimes B\oplus B\otimes E_1$, so for all $a\in
A^2$, we can define a map
\begin{equation}
\label{DD}
\mathbb{D}\colon A^2 \lto \Hom_{R^\e}(E_1,E_1\otimes B\oplus B\otimes E)\colon \quad a \longmapsto \mathbb{D}_a:=\lr{a,-}_{\omega}\vert_{E_1}.
\end{equation}
Then, given $b\in B, e\in E_1, a\in A^2$, the graded Leibniz
rule applied to $\lr{-,-}_\omega$ yields
\begin{subequations}\label{eq:Leibniz-double-covariant}
\begin{align}
\label{eq:Leibniz-double-covariant.a}
\mathbb{D}_a(be)&=b\mathbb{D}_a(e)+\mathbb{X}_a(b)e,
\\
\label{eq:Leibniz-double-covariant.b}
\mathbb{D}_a(eb)&=\mathbb{D}_a(e)b+e\mathbb{X}_a(b),
 \end{align}
 \end{subequations}
 with $b$ acting via the outer bimodule structure on $\DD_a$ and
 $\mathbb{X}_a$. Therefore, $\DD$ can be regarded as a family of
 'covariant' double differential operators parametrized by $A^2$,
 associated to the family of double derivations $\XX$.

\subsubsection{The pairing}
\label{sub:pairing}

Given a graded algebra $C$, Van den Bergh \cite[Appendix A]{Kel11}
defines a \emph{pairing} between two graded $C$-bimodules $P$ and $Q$
as a homogeneous map of weight $n$ 
\[
\langle -,-\rangle\colon P\times Q\to C\otimes C,
\]
such that $\langle p,-\rangle$ is linear for the outer graded bimodule
structure on $C\otimes C$, and $\langle -,q\rangle$ is linear for the
inner graded bimodule structure on $C\otimes C$, for all $p\in P$,
$q\in Q$. We say that the pairing is \emph{symmetric} 
if $\langle p,q\rangle=\sigma_{(12)}\langle q,p\rangle$ (with
$\sigma_{(12)}$ as in~\eqref{eq:flip}), and \emph{non-degenerate} if
$P$ and $Q$ are finitely generated graded projective $C$-bimodules and
the pairing induces an isomorphism
\[
Q\lra{\cong} P^{\vee}[-n]\colon\quad q\longmapsto \langle-,q\rangle,
\]
with $P^\vee=\Hom_{C^\e}(P, {}_{C^\e}C^\e)$.

Consider now the Framework \ref{framework-quivers} associated to a
doubled graded quiver $\dPP$ of weight 2. Using the isomorphism
$\flat$ in \eqref{flat1}, we define
\begin{equation}
\langle -,-\rangle\colon E_1\otimes E_1\lto B\otimes B\colon\quad
 (a,b) \longmapsto \flat(a)(b)=\varepsilon(a)\widetilde{a^*}(b).
\label{inner quivers}
\end{equation}
Consider also the double Poisson bracket $\lr{-,-}_\omega$ of weight
-2 associated to the graded bi-symplectic form $\omega$ of weight 2 in
Proposition \ref{bi-symplectic-form-double-quivers}. By the third
inclusion in \eqref{corchete-doble-relaciones},
\[
\lr{-,-}_{\omega}\vert_{(E_1\otimes E_1)}\colon E_1\otimes E_1\to
B\otimes B.
\]

\begin{lemma}
\begin{enumerate}
\item [\textup{(i)}]
$\lr{a,b}_\omega=\langle a,b\rangle$, for all arrows $a,b\in\P1$ of weight 1.
\item [\textup{(i)}]
The map $\langle-,-\rangle$ in \eqref{inner quivers} is a non-degenerate symmetric pairing.
\end{enumerate}
\label{non-degenerate pairing}
\end{lemma}

\begin{proof}
To prove (i), we use the formula
\[
\lr{a,b}_\omega=i_{\frac{\partial}{\partial a}}\iota_{\frac{\partial}{\partial b}}\omega,
\]
with $\frac{\partial}{\partial a},\frac{\partial}{\partial a}\in\D_R
A$ (see \eqref{derivaciones-quivers}) and the
formula~\eqref{bi-symplectic-gaded-canonical} for the bi-symplectic
form $\omega$. Then
\begin{align*}
\lr{a,b}_\omega&=i_{\frac{\partial}{\partial a}}\iota_{\frac{\partial}{\partial b}}\omega
=i_{\frac{\partial}{\partial a}}\left(\sum_{a\in\dPP_1}\varepsilon(b^*)e_{h(b)}(\du b) e_{t(b)} \right)
\\
&=\sum_{a\in\dPP_1}i_{\frac{\partial}{\partial a^*}}\left(\varepsilon(b^*)e_{h(b)}(\du b) e_{t(b)} \right)
=\varepsilon(a)\widetilde{a^*}(b)
=\langle a,b\rangle.
\end{align*}
The fact that $\langle-,-\rangle$ is a pairing follows from (i) and
properties of double brackets. This pairing is symmetric, because
$\sigma_{(12)}\langle a,b\rangle=\varepsilon(b)e_{t(b)}\otimes
e_{h(b)}=\langle b,a\rangle$, and non-degenerate, because $\flat$ is
an isomorphism (see Theorem \ref{Theorem in weight 1}).
\end{proof}

\subsubsection{Preservation of the pairing}
\label{sub: conserv-pairing-double}

Extend the pairing \eqref{inner quivers} to two maps
\[
\langle-,-\rangle_L\colon E_1\times (A\otimes A)_{(1)} \lto B^{\otimes
  3}, 
\]
\[
\langle-,-\rangle_R\colon E_1\times (A\otimes A)_{(1)} \lto B^{\otimes 3},
\]
given, for all $e_1,e_2\in E_1, b\in B$, by
%
%
\begin{equation}
\label{pairing extendido}
\langle e_1,e_2\otimes b\rangle_L=\langle e_1,e_2\rangle\otimes b,\quad \langle e_1, b\otimes e_2\rangle_L=0,
\end{equation}
%
%
\begin{equation}
\label{extension-pairing-right}
\langle e_1,e_2\otimes b\rangle_R=0,\quad \langle e_1, b\otimes e_2\rangle_R=b\otimes \langle e_1,e_2\rangle.
\end{equation}
We extend similarly the pairing in the first argument with the inner
$\otimes$-product in~\eqref{jumping-notation}, so
\[
\langle-,-\rangle_L\colon (A\otimes A)_{(1)}\times E_1 \lto B^{\otimes 3}
\]
is given by 
\begin{equation}
\langle e_1\otimes b,e_2\rangle_L=\langle e_1,e_2\rangle\otimes_1 b,\quad \langle b\otimes e_1,e_2\rangle_L=0.
\label{extension-pairing-primer-argeumento}
\end{equation}
We will also extend double derivations $\Theta\colon B\to B^{\otimes
  2}$ to maps
\begin{equation}\label{eq:double-Der-extendida}
\Theta\colon B^{\otimes 2}\to B^{\otimes 3}\colon b_1\otimes
b_2\mapsto \Theta(b_1)\otimes b_2.
\end{equation}
Then the family $\XX$ of double derivations in \eqref{XX} and the
family $\DD$ of covariant double operators in \eqref{DD},
\emph{preserve} the pairing $\langle-,-\rangle$, that is,
\begin{equation}
\label{conservation}
\mathbb{X}_a(\langle e_1,e_2\rangle)
=\tau_{(132)}\langle e_2,\mathbb{D}_a(e_1)\rangle_L-\tau_{(123)}\langle e_1,\mathbb{D}_a(e_2)^\circ\rangle_L,
\end{equation}
for all $a\in A^2, e_1,e_2\in E_1$. This follows because
\begin{align*}
\mathbb{X}_a(\langle e_1,e_2\rangle)
&=\lr{a,\lr{e_1,e_2}_\omega}_{\omega,L}
=\tau_{(123)}\lr{e_1,\lr{e_2,a}_\omega}_{\omega,L}+\tau_{(132)}\lr{e_2,\lr{a,e_1}_\omega}_{\omega,L}
\\&
=\tau_{(132)}\langle e_2,\mathbb{D}_a(e_1)\rangle_L-\tau_{(123)}\langle e_1,\mathbb{D}_a(e_2)^\circ\rangle_L,
\end{align*}
where the first identity follows from the
definition~\eqref{eq:double-Der-extendida}, the second identity is
the graded double Jacobi identity, and the third identity follows by
graded skew-symmetry of $\lr{-,-}_\omega$.


\subsection{Twisted double Lie--Rinehart algebras}
\label{sub:A2 and Double Lie--Rinehart algebras}

A key ingredient in \v Severa--Roytenberg's characterization of
symplectic $\NN$-manifolds of weight 2 can be interpreted
algebraically as saying that the Atiyah algebroids and the commutative
analogue of $A^2$ have structures of Lie--Rinehart algebras
(\cite[Theorem 3.3]{Roy00}). In this subsection, we define a slight
generalization of Van den Bergh's double Lie algebroid
\cite[Definition 3.2.1]{VdB08a}, that fits both the underlying
algebraic structure of $A^2$ and a suitable non-commutative version of
the Atiyah algebroid that will be introduced in~\secref{sub: double
  Atiyah}.

\subsubsection{Definition of double Lie--Rinehart algebras}
\label{sub:twisted-LR}

Let $N$ be a $B$-bimodule. Following \cite[\S 2.3]{VdB08}, given
$n,n_1,n_2\in N$ and $b,b_1b_2\in B$, we define
\begin{equation}
\begin{aligned}
\lr{n_1,b\otimes n_2}_L&=\lr{n_1,b}\otimes n_2,\quad  \lr{n_1,n_2\otimes b}_L&=\lr{n_1,n_2}\otimes b,
 \\
\lr{n,b_1\otimes b_2}_L&=\lr{n,b_1}\otimes b_2, \quad
\lr{b_1,n\otimes b_2}_L&=\lr{b_1,n}\otimes b_2.
  \end{aligned}
 \label{ExtensionNormalDoubleBracket}
 \end{equation}
 The rest of combinations will be zero by definition. Given a
 permutation $s\in S_n$, we will use the notation $\tau_s$ for the map
 in~\eqref{eq:permutation-tau}, to emphasize that permutations act on
 mixed tensor products of algebras and bimodules. For instance,
 $\tau_{(12)}\colon N\otimes B\to B\otimes N\colon n\otimes b\mapsto
 b\otimes n$.

\begin{definition}
A \emph{twisted double Lie--Rinehart algebra} over $B$ is a
4-tuple 
\[
(N,\overline{N},\rho,\lr{-,-}_N),
\]
where $N$ is a $B$-bimodule, $\overline{N}\subset N$ is a
$B$-subbimodule called the \emph{twisting subbimodule}, $\rho\colon
N\to \D_R B$ is a $B$-bimodule map, called the \emph{anchor}, and 
 \[
\lr{-,-}_N\colon N\times N\lto N\otimes B\oplus B\otimes N\oplus \overline{N}^\vee\otimes \overline{N}\oplus \overline{N}\oplus\overline{N}^\vee,
\]
is a bilinear map, called the \emph{double bracket}, satisfying the
following conditions:
\begin{enumerate}
\item [\textup{(a)}]
$\lr{n_1,n_2}_N=-\tau_{(12)}\lr{n_2,n_1}_N$,
\item [\textup{(b)}]
$\lr{n_1,bn_2}_N=b \lr{n_1,n_2}_N+\rho(n_1)(b) n_2$,
\item [\textup{(c)}]
$\lr{n_1,n_2b}_N=\lr{n_1,n_2}_N  b+n_2 \rho(n_1)(b)$,
\item [\textup{(d)}]
$\displaystyle\parbox{0pt}{
\begin{align*}
0=\lr{n_1,\lr{n_2,n_3}_N}_{N,L}
+&\tau_{(123)}\lr{n_2,\lr{n_3,n_1}_N}_{N,L}
\\+&\tau_{(132)}\lr{n_3,\lr{n_1,n_2}_N}_{N,L},
\end{align*}}$
\item [\textup{(e)}]
$\rho(\lr{n_1,n_2}_N)=\lr{\rho(n_1),\rho(n_2)}_{\text{SN}}$,
\end{enumerate}
for all $n_1,n_2,n_3\in N$, $b\in B$. If the twisting subbimodule is
zero (i.e. $\overline{N}=0$), then we say that the triple
$(N,\rho,\lr{-,-}_N)$ is a \emph{double Lie--Rinehart algebra}.
\label{DoubleLieRinehart}
\end{definition}

In Definition \ref{DoubleLieRinehart}, all products involved use the
outer bimodule structure. Also, in (e), $\lr{-,-}_{\text{SN}}$ denotes
the double Schouten--Nijenhuis bracket (see \eqref{Gerstenhaber}), and
by convention, $\rho$ acts by the Leibniz rule on tensor products.

\begin{example}
By Proposition \ref{Gerstenhaber-str}, $\D_R B$ is a double
Lie--Rinehart algebra when it is equipped with the double
Schouten--Nijenhuis bracket restricted to $(\D_R B)^{\otimes 2}$ and
the identity as anchor.
 \end{example}

\subsubsection{$A^2$ as a twisted double Lie--Rinehart algebra}

As we showed in \eqref{corchete-doble-relaciones}, $\lr{A^2,A^2}_\omega\subset (A\otimes A)_{(2)}=E_2\otimes B\oplus B\otimes E_2\oplus E_1\otimes E_1$, so we can define
\begin{equation}
\lr{-,-}_{A^2}:=\lr{-,-}_{\omega}\vert_{A^2\otimes A^2}\colon A^2\otimes A^2\longrightarrow (A\otimes A)_{(2)}.
\label{corchete-A2}
\end{equation}
In Framework \ref{framework-quivers}, by Lemma \ref{non-degenerate pairing}, we know that $E_1$ is endowed with a non-degenerate symmetric pairing. Then, in particular, $E_1\cong E^\vee_1$.
\begin{proposition}
In the setting of Framework \ref{framework-quivers}, $A^2$ is a
twisted double Lie-Rinehart algebra, with the bracket
$\lr{-,-}_{A^2}$, the anchor $\rho\colon A^2\to \D_R B\colon$
$a\mapsto \mathbb{X}_a$ (see \eqref{XX}), and the twisting subbimodule $E_1$.
\label{A2LR}
\end{proposition}

\begin{proof}
Conditions (a) and (d) in Definition \ref{DoubleLieRinehart} are
automatic, because $\lr{-,-}_{A^2}$ is the restriction of a double
Poisson bracket, and (e) follows by applying Lemma \ref{intercambio}. Finally,
(b) and (c) are consequences of the graded Leibniz rule applied to
$\lr{-,-}_\omega$. For instance, (b) follows because, given
$a_1,a_2\in A^2$, $b\in B$,
\begin{align*}
\lr{a_1,ba_2}_{A^2}&= b \lr{a_1,a_2}_{\omega}+\lr{a_1,b}_{\omega} a_2
\\
&=b \lr{a_1,a_2}_{\omega}+\mathbb{X}_{a_1}(b)  a_2
=b \lr{a_1,a_2}_{A^2}+\rho(a_1)(b)  a_2
\qedhere
\end{align*}
\end{proof}

\subsubsection{Morphisms of twisted double Lie--Rinehart algebras}

Let $(N,\overline{N},\lr{-,-}_N,\rho_N)$ be a twisted double
Lie--Rinehart algebra. Suppose that $\overline{N}$ is endowed with a
non-degenerate pairing. Then $\overline{N}\cong\overline{N}^\vee$,
and we can perform the following decomposition of the double bracket
$\lr{-,-}_N$:
\begin{align}
\label{decomposition-rolls}
&\lr{n_1,n_2}_N=\lr{n_1,n_2}^l+\lr{n_1,n_2}^l+\lr{n_1,n_2}^m
\\\notag
&\;=\lr{n_1,n_2}^{l^\prime}\otimes \lr{n_1,n_2}^{l^{\prime\prime}}+\lr{n_1,n_2}^{r^\prime}\otimes \lr{n_1,n_2}^{r^{\prime\prime}}+\lr{n_1,n_2}^{m^\prime}\otimes \lr{n_1,n_2}^{m^{\prime\prime}},
\end{align}
with $\lr{n_1,n_2}^l\in N\otimes B$, $\lr{n_1,n_2}^r\in B\otimes N$, $\lr{n_1,n_2}^m\in \overline{N}\otimes \overline{N}$, and
\[
\lr{n_1,n_2}^{l^\prime}, \lr{n_1,n_2}^{r^{\prime\prime}}\in N,\quad \lr{n_1,n_2}^{r^\prime}, \lr{n_1,n_2}^{l^{\prime\prime}}\in B, \quad \lr{n_1,n_2}^{m^\prime}, \lr{n_1,n_2}^{m^{\prime\prime}}\in \overline{N}.
\]

\begin{definition}
\label{MapDoubleLieRinehart}
Let $(N,\overline{N},\lr{-,-}_N,\rho_N)$ and $(N^{\prime},\overline{N}^\prime,
\lr{-,-}_{N^{\prime}},\rho_{N^{\prime}})$ be two twisted double
Lie--Rinehart algebras over $B$, and $\langle-,-\rangle$ a
non-degenerate pairing on $\overline{N}$, so the double bracket
$\lr{-,-}_N$ admits the decomposition~\eqref{decomposition-rolls}.
Then a \emph{morphism of twisted double Lie--Rinehart algebras}
between them is a pair $\varphi=(\varphi_1,\varphi_2)$, where $\varphi_1\colon N\to
N^\prime$ and $\varphi_2\colon \overline{N}\to
(\overline{N}^\prime)^\vee$ are $B$-bimodule morphisms, such that for
all $n_1,n_2\in N$,
 \begin{enumerate}
  \item [\textup{(i)}]
$\lr{n_1,n_2}^{m^\prime}, \lr{n_1,n_2}^{m^{\prime\prime}}\in \overline{N}^\prime$;
  \item [\textup{(ii)}]
$\varphi(\lr{n_1,n_2}_N)=\lr{\varphi(n_1),\varphi(n_2)}_{N^{\prime}}$,
where in the left-hand side, by convention, 
\begin{align*}
\varphi(\lr{n_1,n_2}_N)
&=\varphi_1(\lr{n_1,n_2}^{l^\prime})\otimes\lr{n_1,n_2}^{l^{\prime\prime}}
+\lr{n_1,n_2}^{r^\prime}\otimes\varphi_1(\lr{n_1,n_2}^{r^{\prime\prime}})
\\&
+\varphi_2(\lr{n_1,n_2}^{m^\prime})\otimes\lr{n_1,n_2}^{m^{\prime\prime}}
+\lr{n_1,n_2}^{m^\prime}\otimes\varphi_2(\lr{n_1,n_2}^{m^{\prime\prime}});
\end{align*}
\item [\textup{(iii)}]
the following diagram commutes. 
\[
\xymatrix{
N \ar[rd]^{\rho_N} \ar[dd]^-{\varphi_1}
\\
& \D_RB
\\
N^{\prime}\ar[ru]^-{\rho_{N^\prime}}
}
\]
 \end{enumerate}
\end{definition}

\subsection{The double Atiyah algebra}
\label{sub: double Atiyah}

We will now define a non-commutative analogue of Atiyah algebroids,
and use the square-zero construction to show that they are twisted
double Lie--Rinehart algebras.

\subsubsection{The definition of double Atiyah algebra}

Let $R$ be an associative algebra, $B$ be an associative $R$-algebra,
and $E$ a finitely generated projective $B$-bimodule equipped with a
symmetric non-degenerate pairing $\langle-,-\rangle$ (see
\secref{sub:pairing}). Define 
\begin{equation}
\EEnd_{R}(E )\defeq\Hom_{R^\e}(E,E \otimes B \oplus B\otimes E),
\label{End-Nc}
\end{equation}
with the outer $R$-bimodule structure on $E \otimes B\oplus B\otimes
E$. The surviving inner $B$-bimodule structure on $ E\otimes B \oplus
B\otimes E$ makes $\EEnd_{R}(E)$ into a $B$-bimodule.  Its elements
will be called $R$-linear \emph{double endomorphisms}. Given $e\in E,
\mathbb{D}\in \EEnd_{R}(E)$, we will use the decomposition
\begin{equation}
\mathbb{D}(e) = \left(\mathbb{D}^l+\mathbb{D}^r \right)(e)=\left( \mathbb{D}^{l^{\prime}}\otimes\mathbb{D}^{l^{\prime\prime}}+\mathbb{D}^{r^{\prime}}\otimes\mathbb{D}^{r^{\prime\prime}}\right)(e),
\label{decomposition-D}
\end{equation}
omitting the summation symbols, where $\mathbb{D}^l(e)\in E_1\otimes
B$, $\mathbb{D}^r(e)\in B\otimes E_1$, and
\[
\mathbb{D}^{l^{\prime}}(e),\mathbb{D}^{r^{\prime\prime}}(e)\in E_1,\quad  \mathbb{D}^{l^{\prime\prime}}(e), \mathbb{D}^{r^{\prime}}(e)\in B.
\]

The following conditions (i) and (ii) should be compared
with~\eqref{eq:Leibniz-double-covariant}.

\begin{definition}
\label{double-Atiyah}
The ($R$-linear) \emph{double Atiyah algebra} $\AAt_B(E)$ is the set
of pairs $(\mathbb{X},\mathbb{D})$ with $\mathbb{X}\in\D_R B$ and
$\mathbb{D}\in\EEnd_{R}(E)$, satisfying the following conditions for
all $b\in B, e\in E$:
\begin{enumerate}
\item [\textup{(i)}]
$\mathbb{D}(b e)=b\mathbb{D}(e)+\mathbb{X}(b)e$,
\item [\textup{(ii)}]
$\mathbb{D}(e b)=\mathbb{D}(e)b+e\mathbb{X}(b)$.
\end{enumerate}
Here, all the products are taken with respect to the outer structure.
\end{definition}

It is easy to see that $\AAt_B(E_1)$ is a $B$-subbimodule of the
direct sum of the $B$-bimodules $\D_R B$ and $\EEnd_R (E)$. Using now
the symmetric non-degenerate pairing of $E$, we can impose
preservation of the pairing, in the sense of \eqref{conservation}, on
elements of $\AAt_B(E)$.

\begin{definition}
Let $E$ be a finitely generated projective $B$-bimodule equipped with
a symmetric non-degenerate pairing $\langle-,-\rangle$. The
($R$-linear) \emph{metric double Atiyah algebra} of $E$ is the
subspace $\AAt_B(E, \langle-,-\rangle)\subset\AAt_B(E)$ of pairs
$(\XX,\DD)$ that preserve the pairing, i.e. 
%
\[
\mathbb{X}(\langle e_2,e_1\rangle)=\tau_{(123)}\langle e_1,\mathbb{D}(e_2)\rangle_L-\tau_{(132)}\langle e_2,\mathbb{D}(e_1)^\circ\rangle_L,
\]
for all $e_1,e_2\in E$, $\mathbb{X}\in\D_R B$ and $\mathbb{D}\in\EEnd_{R}(E)$.
\label{metric-double-Atiyah}
\end{definition}

\subsubsection{The bracket}
\label{sub:bracket-Atiyah}

Using the square-zero construction (see, e.g., \cite[\S
3.2]{Gin05}), we define a graded associative $R$-algebra
\[
C:=B\sharp (E[-1]),
\]
with underlying graded $R$-bimodule $B\oplus(E[-1])$, and
multiplication
\[
(b,e)\cdot (b^\prime,e^\prime)=(bb^\prime,be^\prime+b^\prime e),
\]
for all $b,b^\prime\in B, e,e^\prime\in E$. Observe that $E$ is a
nilpotent ideal, i.e.  $e\cdot e'=0$, and $C$ has unit
$1_C=(1_B,0)$. Given a graded $C$-bimodule $F$, we obviously have
\begin{equation}
\Der_R(C,F)=\Bigg\{ D\colon C\lto F\Bigg| \begin{aligned}
D(bb^\prime)&=D(b)b^\prime+bD(b^\prime),
\\
D(be)&=D(b)e+bD(e),\; \text{for all }b,b^\prime\in B,\, e\in E\\
D(eb)&=D(e)b+eD(b),
\end{aligned}
\Bigg\}, 
\label{isomorphism-DerC}
\end{equation}
so when $F=(C\otimes C)_\out$, the subspace of derivations of weight 0
is
\[
(\D_R(C))_{(0)}\cong \AAt_B(E),
\]
where the isomorphism maps a derivation $\Theta\colon C\to C\otimes C$
of weight 0 into the pair $(\XX,\DD)$ consisting of its restrictions
to $B$ and $E[-1]$ (with the appropriate weight shift). The inverse
will be denoted
\begin{equation}\label{Atiyah-peso-cero}
\Xi\colon \AAt_B(E_1)\lra{\cong} (\D_R(C))_{(0)}.
\end{equation}

It also follows from \eqref{isomorphism-DerC} that the subspace of
double derivations of weitht -1 is
\[
(\D_R C)_{(-1)}\cong E^\vee.
\]
Now, we extend the symmetric non-degenerate pairing
$\langle-,-\rangle$ from $E$ to $C$ by the formulae $\langle
b,b^\prime\rangle=\langle e,b\rangle=\langle b,e\rangle =0$ (this
should be compared with Lemma \ref{non-degenerate pairing}(i), in the
Framework \ref{framework-quivers}). Then~\eqref{Atiyah-peso-cero}
restricts to another isomorphism
\begin{equation}
\overline{\Xi}\,\colon \AAt_B(E_1,\langle-,-\rangle)\lra{\cong} (\D_R(C,\langle-,-\rangle))_{(0)},
\label{Atiyah-isom-inner}
\end{equation}
where $(\D_R(C,\langle-,-\rangle))_{(0)}$ is the $C$-bimodule of
double derivations of weight 0 that preserve the pairing (extended to
$C$).

Observe now that the double Schouten--Nijenhuis bracket on $T_C \D_R
C$ preserves weights, so it restricts to another bracket on the tensor
subalgebra $T_C((\D_R C)_{(-1)})$, that also satisfies skew-symmetry
and the Leibniz and Jacobi rules. In the rest of this subsection, we
will use this restricted double Schouten--Nijenhuis bracket to
construct another double bracket on $T_B\AAt_B(E)$ via the
isomorphisms~\eqref{Atiyah-peso-cero} and~\eqref{Atiyah-isom-inner}. 


Let $T_1,T_2\in(\D_R C)_{(0)}$. Then formulae \eqref{tilde-brackets}
define weight 0 $R$-linear triple derivations
\[
\(\lr{T_1,T_2}^{\sim}_r,\lr{T_1,T_2}^{\sim}_l\colon C\to C^{\otimes
  3}\)
\in(\Der_R(C,C^{\otimes 3}))_{(0)},
\]
for the outer bimodule structure on $C^{\otimes 3}$, and
formulae~\eqref{eq:tau_12},~\eqref{eq:tau_23} define isomorphisms
\[
\tau_{(12)}\colon \Der_R (C,C^{\otimes 3})\lra{\cong} C\otimes\D_RC,
\quad
\tau_{(23)}\colon \Der_R (C,C^{\otimes 3})\lra{\cong} \D_RC\otimes C.
\]
Since these isomorphisms preserve weights, they restrict to
isomorphisms in weight 0, 
\begin{equation}
\begin{aligned}
\mathfrak{T}_{(12)}\,\colon\Der(C,C^{\otimes 3})_{(0)}&\lra{\cong} (C\otimes \D_R C)_{(0)}
\\
&\lra{\cong} C_{(0)}\otimes (\D_R C)_{(0)}\oplus C_{(1)}\otimes (\D_R C)_{(-1)}
\\
&\lra{\cong} B\otimes \AAt_B (E) \oplus E\otimes E^\vee,
\end{aligned}
\label{Tau-12}
\end{equation}
\begin{equation}
\begin{aligned}
\mathfrak{T}_{(23)}\,\colon\Der(C,C^{\otimes 3})_{(0)}&\lra{\cong} (\D_R C\otimes C)_{(0)}
\\
&\lra{\cong} (\D_R C)_{(0)}\otimes C_{(0)}\oplus (\D_R C)_{(-1)}\otimes C_{(1)}
\\
&\lra{\cong} \AAt_B (E)\otimes B\oplus E^\vee\otimes E.
\end{aligned}
\label{Tau-23}
\end{equation}
Hence the double brackets $\lr{T_1,T_2}_{l}$ and $\lr{T_1,T_2}_{r}$,
defined as in \eqref{eq:doublebracket-l}, \eqref{eq:doublebracket-r},
restrict to
\begin{align*}
\lr{T_1,T_2}_{0,l}&:=\mathfrak{T}_{(23)}\circ\lr{T_1,T_2}^\sim_l
\in \AAt_B (E)\otimes B\oplus E^\vee\otimes E,
\\
\lr{T_1,T_2}_{0,r}&:=\mathfrak{T}_{(12)}\circ\lr{T_1,T_2}^\sim_l
\in B\otimes \AAt_B (E) \oplus E\otimes E^\vee,
\end{align*}
whereas formulae~\eqref{Gerstenhaber}, for all $c,c_1,c_2\in C$,
$T,T_1,T_2\in(\D_R C)_{(0)}$, restrict to
\begin{equation}
\label{eq:restricteddoubleSchouten-Nijenhuis}
\begin{aligned}&
\lr{c_1,c_2}_0=0,
\\&
\lr{T,c}_0=T(c),
\\&
\lr{c,T}_0=-\sigma_{(12)}(T(c)),
\\&
\lr{T_1,T_2}_0=\lr{T_1,T_2}_{l,0}+\lr{T_1,T_2}_{r,0}.
\end{aligned}
\end{equation}
Now, it follows from~\eqref{isomorphism-DerC} that 
\[
(\Der_R(C,C^{\otimes 3}))_{(0)}\cong \Bigg\{ \begin{aligned}
\mathbb{X}\colon B &\lto B^{\otimes 3}
\\
\mathbb{D}\colon E_1 &\lto E_1\otimes B\otimes B+\textit{c.p.}
\end{aligned}
\Bigg|
\begin{aligned}
\mathbb{X}(bb^\prime)&=\mathbb{X}(b)b^\prime+b\mathbb{X}(b^\prime)
\\
\mathbb{D}(be)&=\mathbb{X}(b)e+b\mathbb{D}(e)
\\
\mathbb{D}(eb)&=\mathbb{D}(e)b+e\mathbb{X}(b)
\end{aligned}
\Bigg\},
\]
where ``$\textit{c.p.}$'' denotes cyclic permutations of the triple
tensor product. Therefore the restricted double Schouten--Nijenhuis
bracket~\eqref{eq:restricteddoubleSchouten-Nijenhuis} corresponds, via
the isomorphisms~\eqref{Atiyah-peso-cero}
and~\eqref{Atiyah-isom-inner}, 
to a double bracket on $T_B\AAt_B(E)$ given on generators by
\begin{align}
\label{bracket-Atiyah-generators}
\notag &
\cc{b_1,b_2}_{\AAt}=0,
\\&
\cc{(\mathbb{X},\mathbb{D}),b}_{\AAt}
=\mathbb{X}(b),
\\\notag
&
\cc{b,(\mathbb{X},\mathbb{D})}_{\AAt}
=-\sigma_{(12)}(\mathbb{X}(b)),
\\\notag
&
\cc{(\mathbb{X}_1,\mathbb{D}_1), (\mathbb{X}_2,\mathbb{D}_2)}_{\AAt}=\lr{\Xi((\mathbb{X}_1,\mathbb{D}_1)),\Xi((\mathbb{X}_2,\mathbb{D}_2))}_{l,0}+\lr{\Xi((\mathbb{X}_1,\mathbb{D}_1)),\Xi((\mathbb{X}_2,\mathbb{D}_2))}_{r,0}.
\end{align}

\subsubsection{The double Atiyah algebra as a twisted double Lie--Rinehart algebra}

\begin{proposition}
The 4-tuple $(\AAt_B(E),E, \cc{-,-}_{\AAt},\rho)$ is a twisted double
Lie--Rinehart algebra, where the bracket is defined by
\eqref{bracket-Atiyah-generators}, and the anchor is
\[
 \rho\colon \AAt_B(E)\longrightarrow \D_R B\colon\; (\mathbb{X},\mathbb{D})\longmapsto\mathbb{X}.
\]
\end{proposition}

\begin{proof}
The 4-tuple $(\AAt_B(E),E, \cc{-,-}_{\AAt},\rho)$ satisfies the
properties in Definition \ref{DoubleLieRinehart}, due to the
isomorphism \eqref{Atiyah-peso-cero} and the fact that we can restrict
the canonical double Schouten--Nijenhuis bracket on $T_C \D_R C$
(which is a double Gerstenhaber algebra) to $(\D_R C)_{(0)}$
preserving the required properties.
\end{proof}

\subsection{The map $\Psi$}
\label{sub:psipsi}

Consider the setting of Framework \ref{framework-general}. Here, we
will construct a map $\Psi$ of twisted double Lie--Rinehart algebras
between $A^2$ and $\AAt_B(E_1)$, using the isomorphism
\eqref{Atiyah-peso-cero}. Let $a\in A^2, (b,e)\in C=B\#(E_1[-1])$ (see
\secref{sub:bracket-Atiyah}). Define
\begin{equation}
\lr{a,(b,e)}_{\omega,0}:=(\lr{a,b}_\omega,\lr{a,e}_\omega).
\label{def-omega-0}
\end{equation}
Then $\lr{a,-}_{\omega,0}\in( \D_R C)_{(0)}$, because
$\lr{-,-}_\omega$ is a double Poisson bracket of weight -2 and
$\lr{a,(b,e)}_{\omega,0}=(\mathbb{X}_a(b),\mathbb{D}_a(e))$ (see
\eqref{XX} and \eqref{DD}), so we can define a map
\begin{equation}
\Psi_1\colon\, A^2\lto (\D_R C)_{(0)}\colon\quad a\longmapsto \mathbb{T}_a= \lr{a,-}_{\omega,0}.
\label{Psi}
\end{equation}
Given $c\in C$, we write
$\mathbb{T}_a(c)=\mathbb{T}^\prime_a(c)\otimes\mathbb{T}^{\prime\prime}_a(c)\in
C\otimes C$. Similarly, we define
\begin{equation}
\Psi_2\,\colon E_1\lto E^\vee_1\colon\quad e\longmapsto \mathbb{T}_e:=\lr{e,-}_\omega.
\label{Psi-2}
\end{equation}

\begin{proposition}
\label{map of LR}
The pair $\Psi=(\Psi_1,\Psi_2)$ is a morphism of twisted double Lie--Rinehart algebras.
\end{proposition}

\begin{proof}
We will partially adapt Lemma \ref{intercambio}. Since $(A^2,E_1,\lr{-,-}_{A^2},\mathbb{X})$ is a twisted double Lie--Rinehart algebra, we can perform the following decomposition
\begin{align*}
\lr{a_1,a_2}_{A^2}&=\lr{a_1,a_2}^l+\lr{a_1,a_2}^r+\lr{a_1,a_2}^m
\\
&=\lr{a_1,a_2}^{l^{\prime}}\otimes \lr{a_1,a_2}^{l^{\prime\prime}}+\lr{a_1,a_2}^{r^{\prime}}\otimes \lr{a_1,a_2}^{r^{\prime\prime}}+\lr{a_1,a_2}^{m^{\prime}}\otimes \lr{a_1,a_2}^{m^{\prime\prime}},
\end{align*}
with $\lr{a_1,a_2}^l\in A^2\otimes B$, $\lr{a_1,a_2}^r\in B\otimes
A^2$, $\lr{a_1,a_2}^m\in E_1\otimes E_1$, and hence
\[
\lr{a_1,a_2}^{l^{\prime}},\lr{a_1,a_2}^{r^{\prime\prime}}\in A^2,\quad \lr{a_1,a_2}^{r^{\prime}},\lr{a_1,a_2}^{l^{\prime\prime}}\in B, \quad
\lr{a_1,a_2}^{m^{\prime}},\lr{a_1,a_2}^{m^{\prime\prime}}\in E_1.
\]
In view of Definition \ref{MapDoubleLieRinehart}, we need to prove
\begin{align}
\label{igualdad deseada}
\begin{split}
\Psi(\lr{a_1,a_2}_{A^2})
=&\lr{\Psi_1(a_1),\Psi_1(a_2)}_0+\Psi_2(\lr{a_1,a_2}^{m^{\prime}})\otimes
\lr{a_1,a_2}^{m^{\prime\prime}}
\\&
+\lr{a_1,a_2}^{m^{\prime}}\otimes \Psi_2(\lr{a_1,a_2}^{m^{\prime\prime}}).
\end{split}
\end{align}

\begin{claim}
The bracket $\lr{-,-}_{\omega,0}$ is skew-symmetric and satisfies the double Jacobi identity.
\label{clcim-corchete-omega-0}
\end{claim}

\begin{proof}
Straightforward, because $\lr{-,.-}_{\omega,0}$ is defined in terms
of the double Poisson bracket $\lr{-,-}_\omega$ on $A$, that already
satisfies the required properties.
\end{proof}

Let $c=(b,e)\in C$, with $b\in B, e\in E_1$. Let $\mathfrak{T}_{(123)}$ and $\mathfrak{T}_{(132)}$  the permutations in $S_3$ that acts on $(\D_R C)_{0}\otimes C\otimes C+c.p$. Then by Claim \ref{clcim-corchete-omega-0}, we have the identity
\begin{equation}\label{Jacobi-Psi}
0=\!\lr{a_1,\lr{a_2,e}_{\omega,0}}_{\omega,0,L}
\!\!
+\mathfrak{T}_{(123)}\lr{a_2,\lr{c,a_1}_{\omega,0}}_{\omega,0,L}
\!\!
+\mathfrak{T}_{(132)}\lr{e,\lr{a_1,a_2}_{A^2}}_{\omega,0,L}.
\!
\end{equation}
The first summand in \eqref{Jacobi-Psi} can be written as
\begin{equation}
\begin{aligned}
\lr{a_1,\lr{a_2,c}_{\omega,0}}_{\omega,0,L}&=\lr{a_1,\mathbb{T}_{a_2}(c)}_{\omega,0,L}
\\
&=\lr{a_1,\mathbb{T}^\prime_{a_2}(c)}_{\omega,0}\otimes \mathbb{T}^{\prime\prime}_{a_2}(c)
\\
&=(\mathbb{T}_{a_1}\otimes\Id_C)\mathbb{T}_{a_1}(c).
\end{aligned}
\label{sumando-1}
\end{equation}
Using the skew-symmetry of $\lr{-,-}_{\omega,0}$, we transform the second summand:
\begin{align*}
\lr{a_2,\lr{c,a_1}_{\omega,0}}_{\omega,0,L}&=-\lr{a_2,(\mathbb{T}_{a_1}(c))^{\circ}}_{\omega,0,L}
\\
&=-\lr{a_2,\mathbb{T}^{\prime\prime}_{a_1}(c)}_{\omega,0}\otimes\mathbb{T}^\prime_{a_1}(c)
\\
&=-(\mathbb{T}_{a_2}\otimes\Id_C)(\mathbb{T}_{a_1}(c))^\circ.
\end{align*}
Consequently,
\begin{equation}
\begin{aligned}
\mathfrak{T}_{(123)}\lr{a_2,\lr{c,a_1}_{\omega,0}}_{\omega,0,L}&=-\mathfrak{T}_{(123)}\left((\mathbb{T}_{a_2}\otimes\Id_C)(\mathbb{T}_{a_1}(c))^\circ\right)
\\
&=-\mathfrak{T}_{(123)}\mathfrak{T}_{(132)}\left( (\Id_C\otimes\mathbb{T}_{a_2})\mathbb{T}_{a_1}\right)
\\
&=-(\Id_C\otimes\mathbb{T}_{a_2})\mathbb{T}_{a_1}.
\end{aligned}
\label{sumando-2}
\end{equation}
To calculate the third summand, note first that
$\lr{-,-}_{\omega,0}\vert_{A^2\otimes
  A^2}=\lr{-,-}_\omega\vert_{A^2\otimes A^2}$. Also,
$\lr{a_1,a_2}^{r^{\prime}}\in B$, and so
$\lr{e,\lr{a_1,a_2}^{r^{\prime}}}_{\omega}=0$, as $\lr{-,-}_{\omega}$
is a double Poisson bracket of weight -2. Similarly, by
\eqref{def-omega-0},
$\lr{\lr{a_1,a_2}^{m^\prime},c}_{\omega,0}=\lr{\lr{a_1,a_2}^{m^\prime},e}_\omega$.
Hence
\begin{align*}
&\!\!\!\!
\lr{c,\lr{a_1,a_2}_{A^2}}_{\omega,0,L}\\
&=-\left(\left(\lr{\lr{a_1,a_2}^{l^{\prime}},c}_{\omega,0}\right)^\circ\otimes\lr{a_1,a_2}^{l^{\prime\prime}}+\left(\lr{\lr{a_1,a_2}^{m^{\prime}},e}_{\omega,0}\right)^\circ\otimes\lr{a_1,a_2}^{m^{\prime\prime}}\right)
\\
&=-\left(\left(\mathbb{T}_{\lr{a_1,a_2}^{l^{\prime}}}(c)\right)^{\circ}\otimes\lr{a_1,a_2}^{l^{\prime\prime}}+\left(\mathbb{T}_{\lr{a_1,a_2}^{m^{\prime}}}(e)\right)^{\circ}\otimes\lr{a_1,a_2}^{m^{\prime\prime}}\right).
\end{align*}
Next,
\begin{equation}
\begin{aligned}
&\!\!\!\!
\mathfrak{T}_{(132)}\lr{e,\lr{a_1,a_2}}_{\omega,0L}\\
&=-\mathfrak{T}_{(132)}\mathfrak{T}_{(12)}\left(\mathbb{T}_{\lr{a_1,a_2}^{l^{\prime}}}(c)\otimes\lr{a_1,a_2}^{l^{\prime\prime}}+\mathbb{T}_{\lr{a_1,a_2}^{m^{\prime}}}(e)\otimes\lr{a_1,a_2}^{m^{\prime\prime}}\right)
\\
&= -\mathfrak{T}_{(23)}\left(\mathbb{T}_{\lr{a_1,a_2}^{l^{\prime}}}\otimes\lr{a_1,a_2}^{l^{\prime\prime}}+\mathbb{T}_{\lr{a_1,a_2}^{m^{\prime}}}(e)\otimes\lr{a_1,a_2}^{m^{\prime\prime}}\right).
\end{aligned}
\label{sumando-3}
\end{equation}
Summing up, from \eqref{Jacobi-Psi}, applying \eqref{sumando-1}, \eqref{sumando-2}, \eqref{sumando-3}, we obtain
\begin{equation}
\begin{aligned}
0&=\lr{a_1,\lr{a_2,c}_{\omega,0}}_{\omega,0,L}(c)+\mathfrak{T}_{(123)}\lr{a_2,\lr{c,a_1}_{\omega,0}}_{\omega,0,L}+\mathfrak{T}_{(132)}\lr{c,\lr{a_1,a_2}_{\omega,0}}_{\omega,0,L},
\\
&=\lr{\mathbb{T}_{a_1},\mathbb{T}_{a_2}}^{\sim}_l (c)-\mathfrak{T}_{(23)}\left(\mathbb{T}_{\lr{a_1,a_2}^{l^{\prime}}}(c)\otimes\lr{a_1,a_2}^{l^{\prime\prime}}+\mathbb{T}_{\lr{a_1,a_2}^{m^{\prime}}}(e)\otimes\lr{a_1,a_2}^{m^{\prime\prime}}\right)
\\
&=\lr{\mathbb{T}_{a_1},\mathbb{T}_{a_2}}_{l,0}(c)-\left(\mathbb{T}_{\lr{a_1,a_2}^{l^{\prime}}}(c)\otimes\lr{a_1,a_2}^{l^{\prime\prime}}+\mathbb{T}_{\lr{a_1,a_2}^{m^{\prime}}}(e)\otimes\lr{a_1,a_2}^{m^{\prime\prime}}\right).
\end{aligned}
\label{parte-izquierda}
\end{equation}
Therefore 
\begin{equation}
\begin{aligned}
\lr{T_{a_1},\mathbb{T}_{a_2}}_{r,0}&=-\lr{\mathbb{T}_{a_2},\mathbb{T}_{a_1}}^\circ_{l,0}
\\
&=-\left(\mathbb{T}_{\lr{a_2,a_1}^{l^\prime}}(c)\otimes \lr{a_2,a_1}^{l^{\prime\prime}}+\mathbb{T}_{\lr{a_2,a_1}^{m^\prime}}(e)\otimes \lr{a_2,a_1}^{m^{\prime\prime}}\right)^\circ
\\
&=-\left(\lr{a_2,a_1}^{l^{\prime\prime}}\otimes\lr{\lr{a_2,a_1}^{l^\prime},c}_{\omega,0}+\lr{a_2,a_1}^{m^{\prime\prime}}\otimes\lr{\lr{a_2,a_1}^{m^\prime},c}_{\omega,0}\right)
\\
&=\lr{a_2,a_1}^{l^{\prime\prime}}\otimes\left(\lr{c,\lr{a_2,a_1}^{l^\prime}}_{\omega,0}\right)^\circ+\lr{a_2,a_1}^{m^{\prime\prime}}\otimes\left(\lr{c,\lr{a_2,a_1}^{m^\prime}}_{\omega,0}\right)^\circ
\\
&=\left(\lr{c,\lr{a_2,a_1}^\circ_{A^2}}_{\omega,0,R}\right)^\circ
\\
&=\left(\lr{c,\lr{a_1,a_2}_{A^2}}_{\omega,0,R}\right)^\circ
\\
&=-\left(\lr{a_1,a_2}^{m^\prime}\otimes \lr{\lr{a_1,a_2}^{m^{\prime\prime}},c}_{\omega,0}+\lr{a_1,a_2}^{r^\prime}\otimes \lr{\lr{a_1,a_2}^{r^{\prime\prime}},c}_{\omega,0}\right)
\\
&=-\left(\lr{a_1,a_2}^{m^\prime}\otimes \lr{\lr{a_1,a_2}^{m^{\prime\prime}},e}_{\omega,0}+\lr{a_1,a_2}^{r^\prime}\otimes \lr{\lr{a_1,a_2}^{r^{\prime\prime}},c}_{\omega,0}\right)
\\
&=-\left(\lr{a_1,a_2}^{m^\prime}\otimes \mathbb{T}_{\lr{a_1,a_2}^{m^{\prime\prime}}}(e)+\lr{a_1,a_2}^{r^\prime}\otimes \mathbb{T}_{\lr{a_1,a_2}^{r^{\prime\prime}}}(c)\right).
\end{aligned}
\label{parte-derecha}
\end{equation}
Finally, \eqref{igualdad deseada} is the sum of \eqref{parte-izquierda} and \eqref{parte-derecha}, as required.
\end{proof}

\subsection{The isomorphism between $A^2$ and $\AAt(E_1)$}
\label{sub:The left isomorphism}

Consider the setting of Framework \ref{framework-general}. Here, we
will show that the map $\Psi$ constructed in~\secref{sub:psipsi} is an
isomorphism of twisted double Lie--Rinehart algebras. As a
consequence, this will imply the following non-commutative version of
\cite[Theorem 3.3]{Roy00}.

\begin{theorem}
\label{ChapterIV}
Let $(A,\omega)$ be the pair consisting of the graded path algebra of
a double quiver $\dPP$ of weight 2, and the bi-symplectic form
$\omega\in\DR{2}(A)$ of weight 2 defined in \secref{sub:bi-symplectic
  form for quivers}. Let $B$ be the path algebra of the weight 0
subquiver of $\dPP$.
Then $(A,\omega)$ is completely determined by the pair
$(E_1,\langle-,-\rangle)$ consisting of the $B$-bimodule $E_1$ with
basis given by paths in $\PP$ of weight 1, and the symmetric
non-degenerate pairing
\[
\langle-,-\rangle\defeq \lr{-,-}_\omega\vert_{E_1\otimes E_1}E_1\otimes E_1\lto B\otimes B.
\]
\end{theorem}

To prove that $\Psi$ is an isomorphism, we will construct the
following commutative diagram, where the rows are the short exact
sequences of $B$-bimodules given by the definitions of $A^2$ and
$\AAt(E_1)$ (see \eqref{decomposition-A-NC} and Definition
\ref{metric-double-Atiyah}, respectively).
\begin{equation}\label{diagrama izquierdo}
\xymatrix{
0\ar[r] & E_1\otimes_B E_1\ar[r]\ar[d]^{\Psi\vert_{E_1\otimes_B E_1}} & A^2\ar[r]\ar[d]^{\Psi} & E_2\ar[r] \ar[d]^{\widetilde{\iota}(\omega)_{(0)}}&0
\\
0\ar[r] & \aad_{B}(E_1)\ar[r] &\AAt(E_1)\ar[r]& \D_RB\ar[r] &0
}
\end{equation}
Here, using \eqref{conservation} and \eqref{pairing extendido}, we
define the \emph{double adjoint $B$-bimodule of $(E_1,\langle-,-\rangle)$} as
\begin{equation}
 \aad_{B}(E_1)\defeq\{ \mathbb{D}\in\EEnd_{B^\e}(E_1) \vert -\langle e_1,\mathbb{D}_a(e_2)\rangle_L=\sigma_{(132)}\langle e_2,\mathbb{D}_a(e_1)^\circ\rangle_L\}.
 \label{aad}
 \end{equation}
 After the construction of the above commutative diagram, it will
 follow that $\Psi$ is an isomorphism if and only if so is its
 restriction to $E_1\otimes_B E_1$ (because
 $\widetilde{\iota}(\omega)_{(0)}$ is an isomorphism by Theorem
 \ref{weight 0}). The latter fact will be proved in the setting of
 doubled graded quivers developed in \secref{sub:bi-symplectic form
   for quivers}. From now on, let $R$ be the semisimple commutative
 algebra with basis the trivial paths in $\PP$.

 The fact that the restriction of $\Psi$ to $E_1\otimes_B E_1$ is an
 isomorphism will follow by performing the following tasks:
%
%
%
%
%
%
\begin{enumerate}
\item [\textup{(i)}]
  Construction of an isomorphism $\EEnd_{B^\e}(E_1)\cong E_1\otimes_B
 E_1\oplus E_1\otimes_B E_1$ (Lemma \ref{description E}).
\item [\textup{(ii)}]
 Description of a basis of $\aad_{B}(E_1)$ (Proposition \ref{basis of aad}).
\item [\textup{(iii)}]
 Description of $\Psi\vert_{E_1\otimes_B E_1}$ in the basis of \textup{(ii)}.
 \end{enumerate}

 \subsection{Explicit description of $\EEnd_{B}(E_1)$}
 \label{Description of the double endomorphisms}

 \begin{lemma}\label{description E}
 There is a canonical isomorphism
 \[
\EEnd_{B}(E_1)\cong E_1\otimes_B E_1\oplus E_1\otimes_B E_1.
\]
\end{lemma}

\begin{proof}
We will need the  canonical isomorphisms
\begin{equation}
\begin{aligned}
\Hom_B(Be_i,M)\stackrel{\cong}{\longrightarrow} e_iM\colon \quad f\longmapsto f(e_i),
\\
\Hom_{B^\op}(e_iB,N)\stackrel{\cong}{\longrightarrow} Ne_i\colon \quad g\longmapsto g(e_i),
\end{aligned}
\label{hom-quivers}
\end{equation}
with inverse isomorphisms
\begin{equation}
\begin{aligned}
e_iM\stackrel{\cong}{\longrightarrow}\Hom_B(Be_i,M)\colon \quad e_im\longmapsto f_m,
\\
Ne_i\stackrel{\cong}{\longrightarrow}\Hom_{B^\op}(e_iB,N)\colon \quad ne_i\longmapsto g_n,
\end{aligned}
\label{isomorfismos-canonicos-isomorfismos-quivers}
\end{equation}
where, for all $b\in B$,
\begin{align*}
f_m\colon Be_i&\lto M\colon\quad be_i\longmapsto f_m(be_i)=be_im,
\\
g_n\colon e_iB&\lto N\colon \quad ne_i\longmapsto g_n(e_ib)=ne_ib.
\end{align*}
%
%
Since $E_1=\bigoplus_{\lvert c\rvert=1} BcB=\bigoplus_{\lvert
  c\rvert=1}Be_{h(c)}\otimes e_{t(c)}B$, we have 
\begin{equation}
E_1\otimes_B E_1
\cong \bigoplus_{\lvert c\rvert=\lvert d\rvert=1} Be_{h(c)}\otimes e_{t(c)}Be_{h(d)}\otimes e_{t(d)}B\cong \bigoplus_{\lvert c\rvert=\lvert d\rvert=1}BcB dB,
\label{EOE1}
\end{equation}
where we sum over arrows $c,d$ of weight 1. This explicit description of $E_1\otimes_B E_1$ in the setting of quivers enables the following explicit description of $\Hom_{B^\e}(E_1, E_1\otimes B)$:
\begin{equation}
\begin{aligned}
\Hom_{B^\e}(E_1,E_{1}&\otimes B)\cong \bigoplus_{\lvert c\rvert=\lvert d\rvert=1}\Hom_{B^\e}\left(Be_{h(c)}\otimes e_{t(c)}B, Be_{h(d)}\otimes e_{t(d)}B\otimes B \right);
\\
&\cong \bigoplus_{\lvert c\rvert=\lvert d\rvert=1}\Hom_B(Be_{h(c)},Be_{h(d)}\otimes e_{t(d)}B)\otimes \Hom_{B^\op}(e_{t(c)}B,B); 
\\
&\cong \bigoplus_{\lvert c\rvert=\lvert d\rvert=1}(e_{h(c)}Be_{h(d)}\otimes e_{t(d)}B)\otimes Be_{t(c)} ;
\\
&\cong \bigoplus_{\lvert c\rvert=\lvert d\rvert=1}Be_{t(c)}\otimes e_{h(c)}Be_{h(d)}\otimes e_{t(d)}B;
\\
&\cong \bigoplus_{\lvert c\rvert=\lvert d\rvert=1}Bc^*BdB\cong \bigoplus_{\lvert c\rvert=\lvert d\rvert=1}BcBdB,
\end{aligned}
\label{iso-222}
\end{equation}
where we used \eqref{hom-quivers}. Also, the last isomorphism is due to the fact that $\dPP$ is a doubled graded quiver of weight 2; hence there exists an isomorphism between the set of arrows $\{a\}$ such that $\lvert a\rvert=1$ and the set of reverse arrows $\{ a^*\}$. In conclusion,
\[
\Hom_{B^\e}(E_1,E_1\otimes B)\cong \bigoplus_{\lvert c\rvert=\lvert d\rvert=1}BcBdB\cong E_1\otimes_B E_1.
\]
Similarly,
\begin{equation}
\begin{aligned}
\Hom_{B^\e}(E_1,B &\otimes E_1)\cong\bigoplus_{\lvert c\rvert=\lvert d\rvert=1}\Hom_{B^\e}\left(Be_{h(c)}\otimes e_{t(c)}B,B\otimes Be_{h(d)}\otimes e_{t(d)}B\right);
\\
&=\bigoplus_{\lvert c\rvert=\lvert d\rvert=1}\Hom_B\left(Be_{h(c)},B)\otimes\Hom_{B^\op}(e_{t(c)}B,Be_{h(d)}\otimes e_{t(d)}B\right);
\\
&\cong \bigoplus_{\lvert c\rvert=\lvert d\rvert=1}e_{h(c)}B\otimes Be_{h(d)}\otimes e_{t(d)}Be_{t(c)};
\\
&\cong \bigoplus_{\lvert c\rvert=\lvert d\rvert=1} Be_{h(d)}\otimes e_{t(d)}Be_{t(c)}\otimes e_{h(c)}B ;
\\
&\cong \bigoplus_{\lvert c\rvert=\lvert d\rvert=1} BdBc^*B\cong \bigoplus_{\lvert c\rvert=\lvert d\rvert=1}BdB cB,
\end{aligned}
\label{isom-111}
\end{equation}
and we obtain that $\Hom_{B^\e}(E_1,B \otimes E_1)\cong \bigoplus_{\lvert c\rvert=\lvert d\rvert=1}BdB cB\cong E_1\otimes_B E_1$.
\end{proof}


Let $a,b $ be arrows of weight 1, and $r,q,p$ paths in $Q$ that
compose, that is,
\begin{equation}
h(p)=t(b), \quad h(b)=t(q),\quad  h(q)=h(a), \quad t(a)=h(r).
\label{natural-compatibility-cpnditions}
\end{equation}
Then we consider the path $ra^*qbp\in\bigoplus_{\lvert c\rvert=\lvert
  d\rvert=1}Bc^*BdB\cong E_1\otimes_B E_1$. By Lemma \ref{description
  E}, we need to determine an explicit basis of the $B$-bimodule
$\EEnd_{B^{\e}}(E_1)$; the image of the path $ra^*qbp$ under the
isomorphism \eqref{isom-111} (resp. \eqref{iso-222}) will be denoted
$[ra^*qbp]_2\in\Hom_{B^{\e}}(E_1 ,B\otimes E_1)$
(resp. $[ra^*qbp]_1\in\Hom_{B^{\e}}(E_1,E_1\otimes B)$). Focusing on
\eqref{iso-222},
\[
\xymatrix{
\oplus_{\lvert c\rvert=\lvert d\rvert=1}Be_{t(c)}\otimes e_{h(c)}Be_{h(d)}\otimes e_{t(d)}B \ar[d]^-{\cong}& re_{t(a)}\otimes e_{h(a)}qe_{h(b)}\otimes e_{t(b)}p \ar@{|->}[d]
\\
\oplus_{\lvert c\rvert=\lvert d\rvert=1} e_{h(c)}Be_{h(d)}\otimes e_{t(d)}B\otimes Be_{t(c)}  \ar[d]^-{\cong}& e_{h(a)}qe_{h(b)}\otimes e_{t(b)}p\otimes re_{t(a)} \ar@{|->}[d]
\\
\oplus_{\lvert c\rvert=\lvert d\rvert=1} \Hom_B(Be_{h(c)},Be_{h(d)}\otimes e_{t(d)}B)\otimes \Hom_{B^{\op}}(e_{t(c)}B,B) & f_{e_{h(a)}qe_{h(b)}\otimes e_{t(b)}p}\otimes f^\prime_{re_{t(a)}},
}
\]
where, by \eqref{isomorfismos-canonicos-isomorfismos-quivers}, we have for $s,s^\prime\in B$,
\begin{align*}
f_{e_{h(a)}qe_{h(b)}p}\colon Be_{h(c)}&\lto Be_{h(d)}\otimes e_{t(d)}B\colon\quad se_{h(a)}\longmapsto se_{h(a)}qe_{h(b)}\otimes e_{t(b)}p;
\\
f^\prime_{re_{t(a)}}\colon e_{t(c)}B&\lto B\colon\quad e_{t(a)}s^\prime\longmapsto re_{t(a)}s^\prime.
\end{align*}
Hence, the first isomorphism in \eqref{iso-222} enables us to write $[ra^*qbp]_1\in \Hom_{B^{\e}}(E_1,E_1\otimes B)$:
\begin{equation}
\left[ra^*qbp\right]_1\colon E_1\lto {}_BE_1\otimes B_B\colon\quad ses^\prime \longmapsto \delta_{ae}(se_{h(a)}qe_{h(b)}\otimes e_{t(b)}p)\otimes (re_{t(a)}s^\prime).
\label{nuevo-cuadrado-1}
\end{equation}

Finally, a generic element of a basis of $\EEnd_{B^{\e}}(E_1)$, in view of \eqref{nuevo-cuadrado-1} and \eqref{nuevo-cuadrado-2}, shall be written as
\begin{equation}
f=\sum_{\lvert a\rvert=\lvert b\rvert=1}\alpha_{ra^*qbp}[ra^*qbp]_1+\alpha^\prime_{ra^*qbp}[ra^*qbp]_2,
\label{descripcion-explicita-f-en-EEnd}
\end{equation}
where $\alpha_{ra^*qbp}, \alpha'_{ra^*qbp}\in\kk$.

\label{descripcion-explicita-f-en-EEnd}


 \subsection{Description of a basis of $\aad_{B^\e}(E_1)$}
 \label{sub:basis-in-aad}

 In this subsection, we will describe a basis of
 $\aad_{B^{\e}}(E_1)$. Note that $f\in \aad_{B^{\e}}(E_1)$ if and only
 if $f\in \EEnd_{B^{\e}}(E_1)$ (see
 \eqref{descripcion-explicita-f-en-EEnd}) satisfies the following
 additional condition for all $a,b$ arrows of weight 1:
 \begin{equation}
 \langle a,f(b)\rangle_L=-\sigma_{(132)}\langle b,f(a)^\circ\rangle_L,
 \label{objetivo}
 \end{equation}
where $\sigma_{(123)}\colon B\otimes B\otimes B\to B\otimes B\otimes B \colon b_1\otimes b_2\otimes b_3\mapsto  b_2\otimes b_3\otimes b_1$.

 \begin{lemma}
 A basis of $\aad_{B^\e}(E_1)$ consists of the elements
%
%
\begin{equation}
\varepsilon(b)[ra^*qbp]_2-\varepsilon(a)[ra^*qbp]_1,
 \label{basis of aad1}
\end{equation}
where $p,q,r$ are paths in $Q$ and $a,b$ arrows of weight 1, which satisfy the following compatibility conditions:
\[
h(p)=t(b),\quad h(b)=t(q),\quad h(q)=h(a),\quad t(a)=t(r).
\]
 \label{basis of aad}
 \end{lemma}
 \begin{proof}
To prove \eqref{objetivo}, we will write explicitly $f(b)$ and $f(a)^\circ$ for $a,b$ some arrows of weight 1. By \eqref{descripcion-explicita-f-en-EEnd}, 
 \begin{align*}
 f(b)&=\sum_{\lvert c\rvert=\lvert d\rvert=1}\alpha_{rc^*qdp}[rc^*qdp]_1(b)+\alpha^\prime_{rc^*qdp}
 [rc^*qdp]_2(b)
 \\
 &=\sum_{\lvert d\rvert=1}(\alpha_{rb^*qdp}e_{h(b)}qe_{h(d)}\otimes e_{t(d)}p\otimes re_{t(b)}+\alpha^\prime_{rb^*qdp}e_{h(b)}p\otimes re_{t(a)}\otimes e_{h(a)}qe_{t(b)}).
 \end{align*}
Next, we can compute the left hand side of \eqref{objetivo} (using \eqref{inner quivers} and \eqref{pairing extendido}):
\begin{align*}
\langle a,f(b)\rangle_L&=
\langle a, \sum_{\lvert d\rvert=1}\alpha_{rb^*qdp}e_{h(b)}qe_{h(d)}\otimes e_{t(d)}p\rangle \otimes re_{t(b)}
\\
&=\varepsilon(a) \alpha_{rb^*qap}e_{h(b)}qe_{t(a)}\otimes e_{h(a)}p\otimes re_{t(b)}.
\end{align*}
Let $h\in E_1$ and $s,s^\prime\in B$. Then using the maps
 \[
 [rc^*qdp]^\circ_{1}\colon E_1\lto B_B\otimes{}_B E_1\colon shs^\prime\longmapsto \delta_{dh}(re_{t(c)}s^\prime)\otimes (se_{h(c)}qe_{h(d)}\otimes e_{t(d)}p)
 \]
and
\[
   [rdqc^*p]^\circ_{2}\colon E_1\lto (E_1)_{B}\otimes {}_BB\colon shs^\prime\longmapsto  \delta_{dh} (re_{t(c)}\otimes e_{h(c)}qe_{t(d)}s^\prime)\otimes (se_{h(d)}p),
 \]
and the fact that $(-)^\circ$ is linear, we can calculate
\begin{align*}
f(a)^\circ &=\left(\sum_{\lvert c\rvert=\lvert d\rvert=1}(\alpha_{rc^*qdp}[rc^*qdp]_1(a)+\alpha^\prime_{rc^*qdp}[rc^*qdp]_2(a))\right)^\circ
\\
&=\sum_{\lvert c\rvert=\lvert d\rvert=1}(\alpha_{rc^*qdp}[rc^*qdp]^\circ_1(a)+\alpha^\prime_{rc^*qdp}[rc^*qdp]^\circ_2(a))
\\
&=\sum_{\lvert c\rvert=1} \alpha_{rc^*qap}re_{t(c)}\otimes e_{h(c)}qe_{h(a)}\otimes e_{t(a)}p+\alpha^{\prime}_{rc^*qap}re_{t(c)}\otimes e_{h(c)}qe_{t(a)}\otimes e_{h(a)}p.
\end{align*}
Then we compute $\sigma_{(132)}\langle b,f(a)^\circ\rangle_L$:
\begin{align*}
\sigma_{(132)}\langle b,f(a)^\circ\rangle_L&=\sigma_{(132)}\sum_{\lvert c\rvert=1}\langle b,\alpha^\prime_{rc^*qap}rc^*qe_{t(a)}\rangle \otimes e_{h(a)}p
\\
&=\sigma_{(132)}\alpha^\prime_{rb^*qap}\varepsilon(b) re_{t(b)}\otimes e_{h(b)}qe_{t(a)}\otimes e_{h(a)}p
\\
&=e_{h(b)}qe_{t(a)}\otimes e_{h(a)}p\otimes \alpha^\prime_{rb^*qap}\varepsilon(b) re_{t(b)}
\end{align*}
Therefore, by \eqref{objetivo}, we obtain the condition
\[
\alpha_{rb^*qap}=-\frac{\varepsilon(b)}{\varepsilon(a)}\alpha^\prime_{rb^*qap},
\]
from which we conclude that \eqref{basis of aad1} is a basis of $\aad_{B^\e}(E_1)$.
\qedhere

\end{proof}

Observe that the above descriptions of $\aad_{B^\e}(E_1)$ and $E_1\otimes_B E_1$ given in Lemma \ref{basis of aad} and \eqref{EOE1} provide an isomorphism between these $B$-bimodules:
\begin{equation}
\begin{aligned}
E_1\otimes_B E_1 &\lto \aad_{B^\e}(E_1)
\\
ra^*qbp&\longmapsto \varepsilon(b)[ra^*qbp]_1-\varepsilon(a) [ra^*qbp]_2.
\label{left isomorphism NC}
\end{aligned}
\end{equation}

\subsection{The isomorphism $\Psi\vert_{E_1\otimes_B E_1}$}

We can now compute $\Psi\vert_{E_1\otimes_B E_1}$ at the basis
elements.

\begin{lemma}
\label{lem:isom-restricted-Psi}
$\Psi$ restricts to an isomorphism $\Psi\colon E_1\otimes_B E_1\lra{\cong} \aad_{B^\e}(E_1)$, given by
\[
\Psi(ra^*qbp)=\varepsilon(b)[ra^*qbp]_2-\varepsilon(a)[ra^*qbp]_1,
\]
where $p,q,r$ are paths in $Q$ and $a,b$ arrows of weight 1, that
compose, that is,
\[
h(p)=t(b),\quad h(b)=t(q),\quad h(q)=h(a),\quad t(a)=t(r).
\]
\end{lemma}

\begin{proof}
First, we note that $\lr{-,-}_{\omega}$ is a double Poisson bracket of
weight -2, so $\mathbb{X}_{ra^*qbp}(b^\prime)=0$ for all $b^\prime\in B$,
because by a simple application of the Leibniz rule and \eqref{XX},
\[
\mathbb{X}_{ra^*qbp}(b^\prime)=-\sigma_{(12)}\left(ra^*\lr{b^\prime,qbp}_\omega+\lr{b^\prime,ra^*}_\omega qbp\right)=0.
\]
We can apply similarly the (graded) Leibniz rule when $c$ is an arrow of weight 1:
\begin{equation}
\begin{aligned}
\mathbb{D}_{ra^*qbp}(c)&=\lr{ra^*qbp,c}_\omega
\\
&=-\sigma_{(12)}\lr{c,ra^*qbp}_\omega
\\
&=-\sigma_{(12)}\(ra^*q\lr{c,b}_\omega p+r\lr{c,a^*}_\omega qbp\) .
\end{aligned}
\label{DDDD}
\end{equation}
To compute $\lr{c,b}_\omega$ and $\lr{c,a^*}_\omega$ we will need an
explicit description of the \emph{differential double Poisson bracket}
$P\in(\T_A\D_BA)_2$ (see \secref{sub:DDP}). 
%
Recall that in our convention, arrows compose from to left.
Applying now Propositions 
\ref{P} and \ref{Schouten-Nijenhuis-quivers}, we compute
\begin{equation}
\begin{aligned}
\lr{c,b}_\omega&=\lr{c,\{P,b\}}_L
=-\lr{c,\varepsilon(b)\frac{\partial}{\partial b^*}}_\omega
\\&
=-\varepsilon(b)\sigma_{(12)}\lr{\frac{\partial}{\partial b^*},c}_\omega
=-\varepsilon(b)e_{h(b)}\otimes e_{t(b)},
\end{aligned}
\label{Pb}
\end{equation}
%
where $\{-,-\}$ is the associated bracket to $\lr{-,-}_\omega$.
Replacing $b$ by $a^*$ in \eqref{Pb}, we obtain
\begin{equation}
\begin{aligned}
\lr{c,a^*}_\omega
=\varepsilon(a)e_{t(a)}\otimes e_{h(a)}.
\end{aligned}
\label{Pa}
\end{equation}
Using now \eqref{Pb} and \eqref{Pa} in \eqref{DDDD}, we conclude
\begin{align*}
\mathbb{D}_{ra^*qbp}(c)&=-\sigma_{(12)}(\varepsilon(b)ra^*qe_{h(b)}\otimes e_{t(b)}p-\varepsilon(a)re_{t(a)}\otimes e_{h(a)}qbp)
\\
&=\varepsilon(b)e_{t(b)}p\otimes re_{t(a)}\otimes e_{h(a)}qe_{h(b)}-\varepsilon(a)e_{h(a)}qe_{h(b)}\otimes e_{t(b)}p\otimes re_{t(a)}
\end{align*}
Therefore, we obtain the formula in the statement of Lemma~\ref{lem:isom-restricted-Psi}.
\end{proof}


\section{Non-commutative Courant algebroids}
\label{sub:DoubleCourant}

In this section, we determine the non-commutative geometric structures
associated to a bi-symplectic $\mathbb{N}Q$-algebra, namely, a graded
bi-symplectic tensor $\mathbb{N}$-algebra $(A,\omega)$ equipped with a
homological double derivation $Q$. We will focus on bi-symplectic
$\mathbb{N}Q$-algebras of weight 2 attached to a double graded quiver
$\dPP$, obtaining in this case a non-commutative analogue of
\cite[Theorem 4.5]{Roy00}, expressible in terms of the so-called
double Courant algebroids over the path algebra of the weight 0
subquiver of $\dPP$, as introduced
in~\secref{definition-Courant-double}.

\subsection{Definition of double Courant algebroids}
\label{definition-Courant-double}

Recall that, in the the setting of Framework \ref{framework-general},
the data of the bi-symplectic tensor $\mathbb{N}$-algebra $(A,\omega)$
of weight 2 is equivalent to a pair $(E,\langle-,-\rangle)$, where
$E\defeq E_1$ is a projective finitely generated $B$-bimodule and
$\langle-,-\rangle$ is the symmetric non-degenerate pairing defined in
\eqref{inner quivers} (see Theorem~\ref{ChapterIV}).

\begin{definition}
Let $R$ be a finite-dimensional semisimple associative algebra and 
$B$ a smooth $R$-algebra. A \emph{double pre-Courant algebroid
  over} $B$ is a 4-tuple $(E,\langle-,-\rangle,\rho,\cc{-,-})$
consisting of a projective finitely generated $B$-bimodule $E$ endowed
with a symmetric non-degenerate pairing, called the \emph{inner product},
\[
\langle-,-\rangle\,\colon E\otimes E\lto B\otimes B,
\]
a $B$-bimodule morphism
\begin{equation}
\rho\,\colon E\lto \D_RB,
\label{anchor-Courant}
\end{equation}
 called the \emph{anchor}, and an operation
 \begin{equation}
\cc{-,-}\,\colon E\otimes E\lto E\otimes B\oplus B\otimes E,
\label{double-Dorfman}
\end{equation}
called the \emph{double Dorfman bracket}, that is $R$-linear for the
outer (resp. inner) bimodule structure on $B\otimes B$ in the second
(resp. first) argument. This data must satisfy the following
conditions:
\begin{subequations}
\label{eq:NCPreCourant}
\begin{gather}
\label{eq:NCPreCourant.a}
\cc{e_1,be_2}=\rho(e_1)(b) e_2+b\cc{e_1,e_2},
\\
\label{eq:NCPreCourant.b}
\cc{e_1,e_2b}=e_2 \rho(e_1)(b)+\cc{e_1,e_2}b,
\\
\label{eq:NCPreCourant.c}
\rho^\vee\du\, (\langle e_2,e_2\rangle)=2\left(\cc{e_2,e_2}+\cc{e_2,e_2}^\sigma\right),
\\
\label{eq:NCPreCourant.d}
\rho(e_1)(\langle e_2,e_2\rangle)=\langle\cc{ e_1,e_2},e_2\rangle_L+\langle e_2,\cc{e_1,e_2}\rangle_R
\end{gather}
\end{subequations}
for all $b\in B$ and $e_1,e_2\in E$.  If the bracket $\cc{-,-}$ satisfies the \emph{double Jacobi identity},
\begin{equation}
\cc{e_1,\cc{e_2,e_3}}_L=\cc{e_2,\cc{e_1,e_3}}_R+\cc{\cc{e_1,e_2},e_3}_L,
\label{Jacobi-CourantNC}
\end{equation}
for all $e_1,e_2,e_3\in E$, then $(E,\langle-,-\rangle,\rho,\cc{-,-})$ is called a \emph{double Courant algebroid}.
\label{definitionpreCourantNC}
\end{definition}

As usual, the products in \eqref{eq:NCPreCourant.a} and
\eqref{eq:NCPreCourant.b} are taken with respect to the outer bimodule
structure. In \eqref{eq:NCPreCourant.c}, the universal derivation
$\du\,\colon B\to \Omega^1_RB$ acts on tensor products by the Leibniz
rule. Furthermore, $\rho^\vee\,\colon \Omega^1_R B\to E$ is composite
\[
 \xymatrix{
 \rho^\vee\,\colon\diff B\ar[rr]^-{\texttt{bidual}} && (\D_R B)^\vee\ar[rr]^-{\Hom_{B^\e}(\rho,{}_{B^\e}B^{\e})}\ar[rr]&& E^\vee\ar[rr]^-\cong && E,
 }
\]
where the first map was defined in \eqref{bidual}, and the isomorphism
between $E$ and $E^\vee$ is induced by the inner product
$\langle-,-\rangle$. As for commutative Courant
algebroids~\cite{LWX97}, given $b\in B$, $e\in E$, we use of the
identification
\[
\langle e,\rho^\vee\du b\rangle=\rho(e)(b).
\]
Finally, we use the notation $(-)^\sigma\defeq\sigma_{(12)}(-)$
(see~\eqref{eq:permutation-tau}).

Given $e_1,e_2\in E$, we perform the following decomposition of the double Dorfman bracket:
\begin{align*}
\cc{e_1,e_2}&=\cc{e_1,e_2}^l+\cc{e_1,e_2}^r
\\
&=\cc{e_1,e_2}^{l^\prime}\otimes\cc{e_1,e_2}^{l^{\prime\prime}}+\cc{e_1,e_2}^{r^\prime}\otimes\cc{e_1,e_2}^{r^{\prime\prime}}\in E\otimes B\oplus B\otimes E,
\end{align*}
with $\cc{e_1,e_2}^{l^\prime},\cc{e_1,e_2}^{r^{\prime\prime}}\in E$, and $\cc{e_1,e_2}^{r^\prime},\cc{e_1,e_2}^{l^{\prime\prime}}\in B$. Then in  \eqref{eq:NCPreCourant.c},
\[
\cc{e_2,e_2}^\sigma=\sigma_{(12)}(\cc{e_2,e_2})=\cc{e_2,e_2}^{l^{\prime\prime}}\otimes\cc{e_2,e_2}^{l^{\prime}}+\cc{e_2,e_2}^{r^{\prime\prime}}\otimes\cc{e_2,e_2}^{r^{\prime}}\in B\otimes E\oplus E\otimes B .
\]

In \eqref{eq:NCPreCourant.d}, if $\langle e_2,e_2\rangle=\langle
e_2,e_2\rangle^\prime\otimes \langle e_2,e_2\rangle^{\prime\prime}\in
B\otimes B$, $\rho(e_1)(\langle e_2,e_2\rangle)=\rho(\langle
e_2,e_2\rangle^\prime)\otimes \langle
e_2,e_2\rangle^{\prime\prime}$. The notation $\langle-,-\rangle_L$ and
$\langle-,-\rangle_R$ was defined in \eqref{extension-pairing-right}
and \eqref{extension-pairing-primer-argeumento}, respectively.

Finally, in the double Jacobi identity \eqref{Jacobi-CourantNC}, we
used the following extensions of the double Dorfman bracket
\begin{align*}
\cc{e_1,e_2\otimes b}_L&=\cc{e_1,e_2}\otimes b, \quad \cc{e_1,b\otimes e_2}_L=\cc{e_1,b}\otimes e_2,
\\
 \cc{e_1,e_2\otimes b}_R &=e_2\otimes \cc{e_1,b},\quad \cc{e_1,b\otimes e_2}_R=b\otimes\cc{e_1,e_2},
\\
\cc{e_1\otimes b,e_2}_L&=\cc{e_1,e_2}\otimes_1 b,\quad \cc{b\otimes e_1,e_2}_L=\cc{b,e_2}\otimes_1 e_2,
\end{align*}
where in the last two identities, we used the inner $\otimes$-product
in~\eqref{jumping-notation}.

\subsection{Bi-symplectic $\mathbb{N}Q$-algebras}

Let $R$ be an associative algebra, and
$B$ an $R$-algebra. We will add now more structure to the data of
Definition \ref{def-tensor-N-algebra}.

\begin{definition}
\begin{enumerate}
\item [\textup{(i)}]
An \emph{associative} \emph{$\mathbb{N}Q$-algebra} $(A,Q)$ over $B$ is an associative
$\mathbb{N}$-algebra $A$ of weight $N$ over $B$ endowed with a double
derivation $Q\colon A\to A\otimes A$ of weight +1 which is
\emph{homological}, that is,
\[
\lr{Q,Q}=0,
\]
where $\lr{-,-}$ is the double Schouten--Nijenhuis bracket. 
\item [\textup{(ii)}]
 A \emph{tensor $\mathbb{N}Q$-algebra} $(A,Q)$ over $B$ is an associative
$\mathbb{N}Q$-algebra over $B$ whose underlying associative $\mathbb{N}$-algebra
is a tensor $\mathbb{N}$-algebra over $B$.
\item [\textup{(iii)}]
A \emph{bi-symplectic $\mathbb{N}Q$-algebra $(A,\omega, Q)$ of weight}
$N$ is a tensor $\mathbb{N}Q$-algebra $(A,Q)$ over $B$, endowed with a
graded bi-symplectic form $\omega\in\DR{2}(A)$ of weight $N$, such that
\begin{enumerate}
\item [\textup{(a)}] the underlying tensor $\mathbb{N}$-algebra $A$
  over $B$, equipped with the graded bi-symplectic form $\omega$, is a
  bi-symplectic tensor $\mathbb{N}$-algebra of weight $N$ over $B$;
\item [\textup{(b)}] the homological double derivation $Q$ is
  bi-symplectic, that is,
\[
\mathcal{L}_Q\omega=0,
\]
where $\mathcal{L}_Q$ is the reduced Lie derivative.
\end{enumerate}
\end{enumerate}
\label{Defintion-NQ-algebra-bisymp}
\end{definition}

\subsection{Bi-symplectic $\mathbb{N}Q$-algebras and double Courant algebroids}

We will now explore the interplay between bi-symplectic
$\mathbb{N}Q$-algebras of weight 2 and double Courant algebroids. As
for commutative manifolds, by Lemma \ref{hamiltonian7}, a
bi-symplectic double derivation is a Hamiltonian double derivation, so
$Q=\lr{S,-}_\omega$, where $\lr{-,-}_\omega$ is the double Poisson
bracket of weight -2 induced by $\omega$ and $S\in A^3$ is a `cubic'
non-commutative polynomial. Then $S$ encodes the structure of a double
pre-Courant algebroid, recoverable via a non-commutative version of
derived brackets and, if in addition, $\{S,S\}_\omega=0$ (here
$\{-,-\}_\omega$ denotes the associated bracket to $\lr{-,-}_\omega$),
then we will have a double Courant algebroid.

Let $(A,\omega, Q)$ be a bi-symplectic $\mathbb{N}Q$-algebra of weight
2 over $B$. In particular, $Q$ is a homological double derivation,
that is,
\[
 \lvert Q\rvert=1,\quad \mathcal{L}_Q\omega=0, \quad \lr{Q,Q}=0,
\]
 where $\lr{-,-}$ is the graded double Schouten--Nijenhuis bracket on
 $\T_A\D_R A$.

\begin{lemma}
The identity $\lr{Q,Q}=0$ is equivalent to $\lr{S,S}_\omega=0$.
\label{lema-equiv-Hamilt}
\end{lemma}


\begin{proof}
  By Lemma \ref{hamiltonian7}\textup{(ii)}, $\iota_Q\omega=\du S$ for
  some $S\in A$. It is easy to see that $S\in A^3$ because
  $\lvert\lr{-,-}_\omega\rvert=-2$, whereas $\lvert Q\rvert=+1$.
%
%
  For all $a\in A$, we see that $\lr{S,a}_\omega=Q(a)$ implies $\lvert
  S\rvert=\lvert Q\rvert-\lvert\lr{-,-}_\omega\rvert=3$, that is,
  $S\in A^3$.  Now, the identity $H_{\lr{a,b}_{\omega}}=\lr{H_a,H_b}$
  of Proposition \ref{intercambio} applied to $a=b=S$ gives
\[
H_{\lr{S,S}_{\omega}}=\lr{Q,Q}.
\]
Therefore
the identity $\lr{Q,Q}=0$ is equivalent to $\lr{S,S}_\omega\in
B\otimes B$ because, by \eqref{hamiltonian-double-deriv},
$H_{\lr{S,S}_{\omega}}=0$ implies $0=\du\,(\lr{S,S}_{\omega})=
(\du\lr{S,S}_{\omega}^\prime)\otimes
\lr{S,S}_{\omega}^{\prime\prime}+\lr{S,S}_{\omega}^\prime\otimes( \du
\lr{S,S}_{\omega}^{\prime\prime})$, which implies
$\lr{S,S}_{\omega}^{\prime},\lr{S,S}_{\omega}^{\prime\prime}\in
R$. Finally, $\lvert\lr{S,S}_{\omega}\rvert=0$, as $B$ is an
associative $R$-algebra.
 However, $\lr{S,S}_\omega$ has weight 4 because $\lvert S\rvert=3$, so $\lr{S,S}_\omega=0$.
\end{proof}

By the following result, double pre-Courant algebroids can be
recovered using non-commutative derived brackets.

\begin{proposition}
Every weight 3 function $S\in A^3$ induces a double pre-Courant
algebroid structure on $(E,\langle-,-\rangle)$, given by
\begin{subequations}
\label{eq:DerivadoNC}
\begin{gather}
\label{eq:DerivadoNC.a}
\rho(e_1)(b)\defeq\lr{\{S,e_1\}_\omega,b}_\omega,
\\
\label{eq:DerivadoNC.b}
\cc{\;e_1,e_2\;}\defeq\lr{\{S,e_1\}_\omega,e_2}_\omega,
\end{gather}
\end{subequations}
for all $b\in B $ and $e_1,e_2\in E$. Here, $\{-,-\}_\omega=m\circ\lr{-,-}_\omega$ is the (graded) \emph{associated bracket} in $A$ (see \eqref{corchete asociado}).
\label{corchetes-derivados-NC}
\end{proposition}

\begin{remark}
  In this section, we will need the graded associated bracket
  $\{-,-\}\defeq\{-,-\}_\omega$, and particularly, the graded version
  of \eqref{LodayLoday} (cf. \cite{Akm97}), namely, 
\begin{equation}
\{a,\{b,c\}\}=\{\{a,b\},c\}+(-1)^{\lvert a\rvert\lvert b\rvert}\{b,\{a,c\}\}.
\label{Loday-Jacobi}
\end{equation}
Recall also that $\{a,-\}$ acts on tensors by $\{a,u\otimes v\}\defeq
\{a,u\}\otimes v+u\otimes\{a,v\}$.
\end{remark}

\begin{proof}
  To simplify the notation, the double Poisson bracket of weight -2
  induced by the bi-symplectic form $\omega$ will be denoted
  $\lr{-,-}\defeq\lr{-,-}_\omega$, and its associated bracket
  $\{-,-\}\defeq\{-,-\}_\omega$. To prove \eqref{eq:NCPreCourant.a} in
  Definition \ref{definitionpreCourantNC}, we will use the fact that
  $\lr{-,-}$ is a double derivation in its second argument with
  respect to the outer structure, so
\begin{align*}
\cc{e_1,be_2}&=\lr{\{S,e_1\},be_2}
\\
&=b\lr{\{S,e_1\},e_2}+\lr{\{S,e_1\},b} e_2
\\
&=b\cc{e_1,e_2}+\rho(e_1)(b)e_2.
\end{align*}
Similarly, we can prove that \eqref{eq:NCPreCourant.a} holds.
To prove \eqref{eq:NCPreCourant.c}, first note that $\lr{e_1,\lr{e_2,e_2}^\prime}\otimes\lr{e_2,e_2}^{\prime\prime}=0$.
Then, by \eqref{mix-corchetes}, 
\begin{align*}
 0&=\{S,\lr{e_1,\lr{e_2,e_2}^\prime}\otimes\lr{e_2,e_2}^{\prime\prime}\}
 \\
 &=\lr{e_1,\lr{e_2,e_2}^\prime}\otimes\{S,\lr{e_2,e_2}^{\prime\prime}\}+\{S,\lr{e_1,\lr{e_2,e_2}^\prime}\}\otimes\lr{e_2,e_2}^{\prime\prime}
 \\
 &=(\lr{\{S,e_1\},\lr{e_2,e_2}^\prime}-\lr{e_1,\{S,\lr{e_2,e_2}^\prime\}})\otimes\lr{e_2,e_2}^{\prime\prime}
 \\
  &=\lr{\{S,e_1\},\lr{e_2,e_2}}_L-\lr{e_1,\{S,\lr{e_2,e_2}^\prime\}\otimes\lr{e_2,e_2}^{\prime\prime}}_L
 \\
 &=\lr{\{S,e_1\},\lr{e_2,e_2}}_L-\lr{e_1,\{S,\lr{e_2,e_2}\}}_L
 \\
  &=\lr{\{S,e_1\},\lr{e_2,e_2}}_L-\lr{e_1,\lr{\{S,e_2\},e_2}-\lr{e_2,\{S,e_2\}}}_L
 \\
 &=\lr{\{S,e_1\},\lr{e_2,e_2}}_L-\lr{e_1,\lr{\{S,e_2\},e_2}}_L-\lr{e_1,\lr{\{S,e_2\},e_2}^\sigma}_L,
 \end{align*}
where
$\lr{-,-}^\sigma=\lr{-,-}^{\prime\prime}\otimes\lr{-,-}^\prime$. By
definition, if $e,e^\prime\in E$, $\lr{e,e^\prime}=\langle
e,e^\prime\rangle$, so 
\[
\rho(e_1)(\langle e_2,e_2\rangle^\prime)\otimes \langle e_2,e_2\rangle ^{\prime\prime}=\langle e_1,\cc{e_2,e_2}+\cc{e_2,e_2}^\sigma\rangle_L.
\]
Next, since $\langle-,-\rangle$ is non-degenerate, the identity
\[
\langle e_1,\rho^\vee(\du\, \langle e_2,e_2\rangle^\prime)\otimes \langle e_2,e_2\rangle^{\prime\prime}\rangle_L=\langle e_1,\cc{e_2,e_2}+\cc{e_2,e_2}^\sigma\rangle_L.
\]
implies that
\begin{equation}
\rho^\vee(\du\, \langle e_2,e_2\rangle^\prime)\otimes \langle e_2,e_2\rangle^{\prime\prime}=\cc{e_2,e_2}+\cc{e_2,e_2}^\sigma.
\label{parcial-ancla-Courant-1}
\end{equation}
Similarly, the identity $\lr{e_2,e_2}^\prime\otimes\lr{e_1,\lr{e_2,e_2}^{\prime\prime}}=0$ implies that
\[
0=\lr{\{S,e_1\},\lr{e_2,e_2}}_R-\lr{e_1,\lr{\{S,e_2\},e_2}}_R-\lr{e_1,\lr{\{S,e_2\},e_2}^\sigma}_R,
\]
which is equivalent to 
\begin{equation}
\langle e_2,e_2\rangle^\prime\otimes \rho^\vee(\du\,\langle e_2,e_2\rangle^{\prime\prime})=\cc{e_2,e_2}+\cc{e_2,e_2}^\sigma.
\label{parcial-ancla-Courant-2}
\end{equation}
The sum of \eqref{parcial-ancla-Courant-1} and \eqref{parcial-ancla-Courant-2} gives \eqref{eq:NCPreCourant.c}.

Finally, the key fact needed to prove \eqref{eq:NCPreCourant.d} is the
double Jacobi identity for $\lr{-,-}$:
\begin{align*}
\rho(e_1)(\langle e_2,e_2\rangle^\prime)\otimes \langle e_2,e_2\rangle^{\prime\prime}&=\lr{\{S,e_1\},\lr{e_2,e_2}^\prime}\otimes \lr{e_2,e_2}^{\prime\prime}
\\
&=\lr{\{S,e_1\},\lr{e_2,e_2}}_L
\\
&=\lr{\lr{\{S,e_1\},e_2},e_2}_L+\lr{e_2,\lr{\{S,e_1\},e_2}}_R
\\
&=\langle \cc{e_1,e_2},e_2\rangle_L+\langle e_2,\cc{e_1,e_2}\rangle_R.
\qedhere
\end{align*}
\end{proof}

%

In Lemma \ref{lema-equiv-Hamilt} we showed that the fact that the
identity $\lr{Q,Q}=0$ is equivalent to $\lr{ S,S}_\omega=0$. In the
next proposition, we prove that the weaker condition
$\{S,S\}_\omega=0$ implies the double Jacobi identity
\eqref{Jacobi-CourantNC}.  As before, $\{-,-\}_\omega$ is the
associated bracket.
\begin{proposition}
If $\{S,S\}_\omega=0$ then the double Jacobi identity \eqref{Jacobi-CourantNC} holds.
\label{recovering-double Courant}
\end{proposition}

\begin{proof}
  To simplify the notation, we set $\lr{-,-}\defeq\lr{-,-}_\omega$. By
  \eqref{eq:DerivadoNC.b},
\[
\cc{e_1,\cc{e_2,e_3}}_L=\lr{\{S,e_1\},\lr{\{\Theta,e_2\},e_3}}_L,
\]
where the double Jacobi identity implies
\[
\lr{\{S,e_1\},\lr{\{\Theta,e_2\},e_3}}_L=\lr{\lr{\{S,e_1\},\{S,e_2\}},e_3}_L+\lr{\{S,e_2\},\lr{\{S,e_1\},e_3}}_R
\label{applying-double-Jacobi}
\]
Regarding the term $\lr{\lr{\{S,e_1\},\{S,e_2\}},e_3}_L$, it follows
from \eqref{mix-corchetes} that 
\[
\lr{\{S,e_1\},\{S,e_2\}}=\{S,\lr{\{S,e_1\},e_2}\}-\lr{\{S,\{S,e_1\}\},e_2}.
\]
Since $(A,\{-,-\})$ is a graded Loday algebra, by \eqref{Loday-Jacobi},
\[
\{S,\{S,e_1\}\}=\frac{1}{2}\{\{S,S\},e_1\}
\]
Thus, since $\{\Theta,\Theta\}=0$ by hypothesis, we obtain the identity
\begin{align*}
 \lr{\lr{\{S,e_1\},\{S,e_2\}},e_3}_L&=\lr{\{S,\lr{\{S,e_1\},e_2}\},e_3}_L
 \\
 &=\cc{\cc{e_1,e_2},e_3}_L.
\end{align*}
Putting all together
\begin{align*}
\cc{e_1,\cc{e_2,e_3}}_L&=\lr{\lr{\{S,e_1\},\{S,e_2\}},e_3}_L+\lr{\{S,e_2\},\lr{\{S,e_1\},e_3}}_R
\\
&=\cc{\cc{e_1,e_2},e_3}_L+\lr{\{S,e_2\},\cc{e_1,e_3}}_R
\\
&=\cc{\cc{e_1,e_2},e_3}_L+\cc{e_2,\cc{e_1,e_3}}_R.
 \end{align*}
%
so \eqref{Jacobi-CourantNC} follows.
\end{proof}

In conclusion, we have proved the following. 

\begin{theorem}
Let $(A,\omega, Q)$ be a bi-symplectic $\mathbb{N}Q$-algebra of weight
2, where $A$ is the graded path algebra of a double quiver $\dPP$ of
weight 2 endowed with a bi-symplectic form $\omega\in\DR{2}(A)$ of
weight 2 defined in \textup{\secref{sub:bi-symplectic form for
    quivers}} and a homological double derivation $Q$.
Let $B$ be the path algebra of the weight 0 subquiver of $\dPP$, and
$(E,\langle-,-\rangle)$ the pair consisting of the $B$-bimodule $E$
with basis weight 1 paths in $\PP$ and the symmetric non-degenerate
pairing
\[
\langle-,-\rangle\defeq \lr{-,-}_\omega\vert_{E\otimes E}
\colon E\otimes E\to B\otimes B.
\]
%
Then the bi-symplectic
$\mathbb{N}Q$-algebra $(A,\omega,Q)$ of weight 2 induces a double
Courant algebroid
$(E,\langle-,-\rangle,\rho,\;\cc{\;-,-\;}\;)$ over $B$, where
\[
\rho(e_1)(b)\defeq\lr{\{\Theta,e_1\}_\omega,b}_\omega,\quad \cc{\;e_1,e_2\;}\defeq\lr{\{\Theta,e_1\}_\omega,e_2}_\omega,
\]
for all $b\in B $ and $e_1,e_2\in E$. Here $\Theta\in A^3$ is determined by the triple $(A,\omega,Q)$, and $\{-,-\}_\omega=m\circ\lr{-,-}_\omega$ is the \emph{associated bracket} in $A$.
\label{ChapterV}
\end{theorem}


\section{Exact non-commutative Courant algebroids}

\subsection{The standard non-commutative Courant algebroid}
\label{sub:standard}

Let $B$ be the path algebra of a quiver $Q$, i.e., $B=\kk Q=T_R V_Q$,
$R=R_Q$ the corresponding semisimple finite-dimensional $k$-algebra,
and $\du_B\,\colon B\to\Omega^1_R B$ the universal derivation. By
Lemma \ref{lema 2.6},
\[
E_1:=\Omega^1_R B=B\otimes_R V_Q\otimes_R B=\bigoplus_{a\in Q_1}B\du_B a B
\]
To simplify the notation, we define $\widehat{a}:=\du_B a$, so
$E_1=\bigoplus_{a\in Q_1}B\widehat{a}B$, and 
\[
S\defeq\Omega^\bullet_R[1]B\defeq T_B(E_1[-1]).
\]
Then we have an identification $S=\kk\PP$ with the graded path algebra
of a graded quiver $\PP$ with vertex set $\P0=\Q0$ and arrow set
$P_1=\Q1\sqcup\widehat{Q}_1$ obtained from $\QQ$ by adjoining an arrow
$\widehat{a}\in\widehat{Q}_1$ for each arrow $a\in Q_1$, with the same
tail and head, i.e., $t(\widehat{a})=t(a)$, $h(\widehat{a})=h(a)$ for
all $a\in\Q1$, and weight function given by $\lvert a\rvert=0$ and
$\lvert\widehat{a}\rvert=1$, for all $a\in Q_1$. Then the graded path
algebra of the double quiver $\dPP$ of $\PP$ of weight 2 can be written
\[
A=\kk\dPP=\T_S\(\D_R S[-2]\).
\]
Note that the arrows of $\dPP$ have weights given by 
\[
\lvert a\rvert=0,\quad \lvert a^*\rvert=2,\quad \lvert \widehat{a}\rvert=1,\quad \lvert \widehat{a}^*\rvert=1,
\]
for all $a\in Q_1$. In particular, $\lvert P\rvert=2$ if $Q_1$ is
non-empty (as we will assume). Furthermore, by Proposition
\ref{bi-symplectic-form-double-quivers}, there is a canonical
bi-symplectic form of weight 2 on $A$, given by
\begin{equation}
\omega_0=\sum_{a\in Q_1}\left(\du a\du a^*+\du\widehat{a}\du\widehat{a}^*\right)\in\DR{2}(A).
\label{bi-sympl-standard}
\end{equation}
where $\du$ is the universal derivation on $A$.

\begin{lemma}
The double derivation of weight +1
\begin{equation}
Q_0=\sum_{a\in Q_1}\left(\frac{\partial}{\partial a}\widehat{a}-a^*\frac{\partial}{\partial \widehat{a}^*}\right)
\label{Q-estandar}
\end{equation}
is a homological, bi-symplectic Hamiltonian double derivation, with Hamiltonian
\begin{equation}
S_0=\sum_{a\in Q_1}a^*\widehat{a}.
\label{Hamiltonian-standard}
\end{equation}
\label{lema-Q-estandar}
\end{lemma}

\begin{proof}
Since $\lvert \partial/\partial a\rvert=-\lvert a\rvert$, it follows
that $\lvert Q_0\rvert=+1$. Note also that
\begin{equation}
\begin{aligned}&
i_{Q_0}(\du a)=Q_0(a)=(e_{h(a)}\otimes e_{t(a)})*\widehat{a}=e_{h(a)}\otimes\widehat{a};
\\&
i_{Q_0}(\du \widehat{a}^*)=Q_0(\widehat{a}^*)=-e_{t(a)}\otimes a^*;
\\&
i_{Q_0}(\du a^*)=Q_0(a^*)=0;
\\&
i_{Q_0}(\du \widehat{a})=Q_0(\widehat{a})=0.
\end{aligned}
\label{aplicaciones-de-Q-a-flechitas}
\end{equation}
To prove that $Q_0$ is Hamiltonian (i.e. $\iota_{Q_0}\omega=\du S_0$
for some $S_0\in A^3$), we apply a graded version of
\eqref{reduced-contraction-general} and
\eqref{aplicaciones-de-Q-a-flechitas}:
\begin{align*}
\iota_{Q_0}\omega&=\iota_{Q_0}\left(\sum_{a\in Q_1}(\du a\du a^*+\du\widehat{a}\du\widehat{a}^*)\right)
\\
&=\sum_{a\in Q_1}\left(i^{\prime\prime}_{Q_0}(\du a)\du a^* i^{\prime}_{Q_0}(\du a)+  i^{\prime\prime}_{Q_0}(\du \widehat{a}^*)\du \widehat{a} i^{\prime}_{Q_0}(\du \widehat{a}^*) \right)
\\
&=\sum_{a\in Q_1}\left(\du a^*\widehat{a}+a^*\du\widehat{a}  \right)=\du\,\(\sum_{a\in Q_1}a^*\widehat{a}\).
\end{align*}
Therefore, \eqref{reduced-Cartan} implies that $Q_0$ is bi-symplectic,
because it is Hamiltonian.

It remains to show that $Q_0$ is homological. Using
\eqref{aplicaciones-de-Q-a-flechitas}, 
\begin{align*}
\lr{Q_0,Q_0}_l(a)&=\tau_{(23)}\left((Q_0\otimes \Id) Q_0(a)-(\Id\otimes Q_0)Q_0(a)\right)
\\
&=\tau_{(23)}\left(Q_0(\widehat{a})\otimes e_{t(a)}-\widehat{a}\otimes Q_0(e_{t(a)})\right)=0,
\end{align*}
so $\lr{Q_0,Q_0}_r(a)=0$ too, by skew-symmetry, and hence
$\lr{Q_0,Q_0}(a)=0$. One can show similarly that
$\lr{Q_0,Q_0}(\widehat{a}^*)=0$. Finally,
$\lr{Q_0,Q_0}(a^*)=\lr{Q_0,Q_0}(\widehat{a})=0$, by
\eqref{aplicaciones-de-Q-a-flechitas}.
\end{proof}


We define the non-commutative derived brackets (see Proposition \ref{corchetes-derivados-NC})
\begin{equation}
\begin{aligned}&
\rho_0(e)(b)=\lr{\{S_0,e\},b}_\omega;
\\&
\cc{e_1,e_2}_0=\lr{\{S_0,e\},b}_\omega,
\end{aligned}
 \label{corchete-standard}
 \end{equation}
for all $b\in B$, $e,e_1,e_2\in E$, with $S_0$ given by
\eqref{Hamiltonian-standard}. Using Theorems \ref{weight 0} and
\ref{ChapterV}, and the above results, we can conclude the following.

\begin{proposition}
Let $Q$ be a quiver with path algebra $B$. Let $A=T_BM$, with
$M=\Omega^1_RB[-1]\oplus\D_RB[-2]$. We endow the graded algebra $A$
with the graded bi-symplectic form $\omega_0$ defined in
\eqref{bi-sympl-standard} and the double derivation $Q_0$ defined in
\eqref{Q-estandar}. Then
\begin{enumerate}
\item [\textup{(i)}]
The triple $(A,\omega_0,Q_0)$ is a bi-symplectic $\mathbb{N}Q$-algebra of weight 2.
\item [\textup{(ii)}]
The 4-tuple $(\Omega^1_R B, \langle-,-\rangle, \rho_0, [-,-]_0)$ (where $\rho_0$, $\cc{-,-}_0$ were defined in  \eqref{corchete-standard} and $\langle-,-\rangle$ in \eqref{inner quivers}) is a double Courant algebroid.
\end{enumerate}
\label{proposition-standard}
\end{proposition}

The double Courant algebroid obtained in Proposition \ref{proposition-standard} will be called the \emph{standard double Courant algebroid}.

\subsection{Deformations of the standard non-commutative Courant
  algebroid}

\v Severa and Weinstein~\cite{SW01} showed that the standard Courant
bracket on $TM\oplus T^*M$, over any smooth manifold $M$, can be
`twisted' in the following way. Given a 3-form $H$, define a bracket
$\cc{-,-}_H$ on $TM\oplus T^*M$ by
$\cc{X+\xi,Y+\eta}_H\defeq\cc{X+\xi,Y+\eta}+i_Yi_XH$, for vector
fields $X,Y$ and 1-forms $\xi,\eta$ on $M$. Then $\cc{-,-}_H$ defines
a Courant algebroid structure on $TM\oplus T^*M$ (using the standard
inner product and anchor) if and only if the 3-form $H$ is
closed. Roytenberg~\cite[\S 5]{Roy00} developed an approach to the
deformation theory of Courant algebroids using the language of derived
brackets in the context of differential graded symplectic
manifolds. We will explain now how Roytenberg's approach can be
adapted to our noncommutative formalism.

Consider the Framework \ref{framework-quivers}, for a quiver $Q$ and
the double graded quiver $\dPP$ defined in~\secref{sub:standard}. 
Define an injection 
\begin{equation}
\label{injection-lambda}
\lambda\colon\Omega^n_{R}B\hookrightarrow A^n.
\end{equation}
by the formulae 
\[
\lambda(a)=a,\quad \lambda(\du\,a)=\widehat{a}, 
\]
and the rule $\lambda(a\du\,b)=\lambda(a)\lambda(\du\, b)$, for all
arrows $a,b\in Q_1$ (so $\du b\in\Omega^1_R B$).

Let $\texttt{Alg}_k$ (respectively $\texttt{CommAlg}_k$) be the
category of associative algebras (resp. the category of commutative
algebras). It is well-known that the inclusion functor
\[
\texttt{CommAlg}_k\hookrightarrow \texttt{Alg}_k
\]
has a left adjoint functor, given by the abelianization
\[
(-)_{\ab}\colon \texttt{Alg}_k\to\texttt{CommAlg}_k
\colon A\mapsto A_{\ab}:=A/(A[A,A]A).
\]
Using the Karoubi-de Rham complex~\eqref{Karoubi-de-Rham}, we get the
following by inspection.

\begin{lemma}
We have the following commutative diagram.
\[
\xymatrix{
\Omega^3_R B \ar@{^{(}->}[r]^-{\lambda}\ar@{->>}[d] & A^3 \ar@{->>}[d]
\\
\DR{3} A \ar@{^{(}->}[r]^-{\lambda_{\ab}} &A^3_{\ab}
}
\]
\label{diagrama-pequenito}
\end{lemma}

Given any $\alpha\in \Omega^\bullet_R A$, the corresponding element in
$ \DR{\bullet} A$ will be denoted $[\alpha]$. Note that
$A_{\ab}=\DR{0}A$. Let $a,b$ and $c$ be three arrows of weight 1, and
$p,q,r,s$ be three paths in $Q$ that compose, that is, 
\[
h(p)=t(a),\quad h(a)=t(q),\quad h(q)=t(b),\quad h(b)=t(r),\quad h(r)=t(c),\quad h(c)=t(s).
\]
Define a noncommutative differential 3-form on $B$, 
\[
\phi:=s(\du\, c)r(\du\, b )q(\du\, a) p\in\Omega^3_R B,
\]
with corresponding cubic polynomial on $A$, 
\[
\widehat{\phi}:=\lambda(\phi)=s\widehat{c} r\widehat{b}q\widehat{a}p\in A^3.
\]
As usual, $m\colon A\otimes A\to A$ denotes multiplication. 

\begin{lemma}
\begin{enumerate}
\item [\textup{(i)}] $\displaystyle{
\lambda(\du\phi)=m\left(Q_0(\lambda(\phi))\right);
}$
\item [\textup{(ii)}] $\displaystyle{
\lr{\widehat{a},\widehat{b}}_\omega=0.
}$
\end{enumerate}
\label{Lema-1-lambda}
\end{lemma}
\begin{proof}
To prove (i), by the definition  \eqref{injection-lambda} of
$\lambda$, it suffices to show the identity when  $\phi=a\in\Omega^1_R
B$, with $a\in Q_1$. In this case,  it is clear that the left-hand
side is $\lambda(\du\,a)=\widehat{a}$, while the right-hand side is 
\[
m\left(Q_0(\lambda(a))\right)=m\left(Q_0(a)\right)=m(e_{h(a)}\widehat{a}\otimes e_{t(a)})=\widehat{a},
\]
as required. For (ii), by \eqref{bi-sympl-standard}, we have the
differential double Poisson bracket
\[
P=\sum_{a\in Q_1}\left(\frac{\partial}{\partial a}\frac{\partial}{\partial a^*}+\frac{\partial}{\partial \widehat{a}}\frac{\partial}{\partial \widehat{a}^*} \right).
\]
Then (ii) follows by inspection, applying \eqref{subQ}.
\end{proof}

We now define the `deformed Hamiltonian'
\[
S_{\widehat{\phi}}:=S_0+\widehat{\phi}.
\]

\begin{lemma}
Let $\{-,-\}_\omega$ be the associated bracket to
$\lr{-,-}_\omega$. Then 
\[
[\{S_{\widehat{\phi}},S_{\widehat{\phi}}\}_\omega]=0
\quad\Longleftrightarrow\quad
\du_{\dR}\,[\phi]=0 \text{ in $\DR{4}A$}.
\]
\end{lemma}

\begin{proof}
By the properties of the associated bracket, we have
\[
[\{S_{\widehat{\phi}},S_{\widehat{\phi}}\}_\omega]=[\{S_0,S_0\}_\omega+\{S_0,\widehat{\phi}\}_\omega+\{\widehat{\phi},S_0\}_\omega+\{\widehat{\phi},\widehat{\phi}\}_\omega]
\]
Let $a$ be an arrow of weight 1, and $p$ a path in $Q$. Lemma
\ref{Lema-1-lambda}(ii) and the fact that
$\lr{\widehat{a},p}_\omega=0$ (since $\lvert\lr{-,-}_\omega\rvert=-2$)
imply that $\lr{\widehat{\phi},\widehat{\phi}}_\omega=0$ and,
consequently, $\{\widehat{\phi},\widehat{\phi}\}_\omega=0$. Also,
Lemmas \ref{lema-equiv-Hamilt} and \ref{lema-Q-estandar} enable us to
conclude that $\{S_0,S_0\}=0$. So,
\[
[\{S_{\widehat{\phi}},S_{\widehat{\phi}}\}_\omega]=[\{S_0,\widehat{\phi}\}_\omega+\{\widehat{\phi},S_0\}_\omega]=2[\{S_0,\widehat{\phi}\}_\omega],
\]
 where we applied \eqref{skew-sym-assoc} in the last identity. Then by Lemmas \ref{Lema-1-lambda}(i) and Lemma \ref{diagrama-pequenito}:
 \begin{align*}
[\{S_0,\widehat{\phi}\}_\omega]&=[m\left(S_0(\lambda(\phi))\right)]
\\
&=[\lambda(\du\,\phi)]
\\
&=\lambda_{\ab}(\du_{\dR}\,[\phi]).
\end{align*}
Since $\lambda_{\ab}$ is injective, $[\{S_{\widehat{\phi}},S_{\widehat{\phi}}\}_\omega]=0$ if and only if $\du_{\dR}\,[\phi]=0$ in $\DR{4}A$, as required.
\end{proof}

In conclusion, any noncommutative differential 3-form
$\lambda\in\Omega^3_R B$ that is closed in the Karoubi--de Rham
complex determines a cubic noncommutative polynomial
$\widehat{\phi}\in A^3$ and hence a deformation of the standard double
Courant algebroid.

\appendix
\section{Proof of Theorem \ref{weight 0}}

Consider the diagram \eqref{DIAGRAMA} constructed in
\secref{sub:cotangent}. Let $a^{\prime},a^{\prime\prime}\in A$ and
$\sigma\in M_N^{\vee}$. Then $\lvert\sigma\rvert=-N$ and
$a^{\prime}\otimes\sigma\otimes a^{\prime\prime}\in A\otimes_B
M^{\vee}_N\otimes_B A$ has weight $\lvert a^{\prime}\rvert +\lvert
a^{\prime\prime}\rvert-N$, when viewed as an element of the space $
A\otimes_B M^{\vee}\otimes_B A $ under the injection $A\otimes_B
M^{\vee}_N\otimes_B A\hra A\otimes_B M^{\vee}\otimes_B A$.

We will write explicitly the morphism $\nu^{\vee}$ of Proposition
\ref{DIAGRAMAGRAL} using the isomorphisms $\kappa$ and $F$ which will
be defined below. Also, these maps will make the square in the
following diagram commute.
\begin{equation}
\xymatrix{
& (A\otimes_B M\otimes_B A)^\vee\ar[r]^-{\nu^\vee} &  (\widetilde{\Omega}/Q)^{\vee}\ar[d]^{F}
\\
0\ar[r] & A\otimes_{B} M^{\vee}\otimes_{B}A \ar[r] \ar[u]^{\kappa}& \D_{R} A\ar[r] &  A\otimes_{B} \D_{R} B \otimes_{B}A \ar[r] & 0
}
\label{diagrama-kappa}
\end{equation}
Firstly, define
\[
\kappa\colon A\otimes_B M^\vee\otimes_B A \lra{\cong} (A\otimes_B M\otimes_B A)^\vee\colon \quad a^\prime\otimes \sigma\otimes a^{\prime\prime}\longmapsto \kappa(a^\prime\otimes \sigma\otimes a^{\prime\prime})
\]
by
\begin{align*}
 \kappa(a^\prime\otimes \sigma\otimes a^{\prime\prime})\colon A\otimes_B M\otimes_B A&\lto A\otimes A
 \\
a^{(1)}_1\otimes m_1\otimes a^{(2)}_1&\longmapsto \pm (a^\prime\sigma^{\prime\prime}(m_1)a^{(2)}_1)\otimes (a^{(1)}_1\sigma^{\prime}(m_1)a^{\prime\prime}),
\end{align*}
where $a^{(1)}_1,a^{(1)}_1\in A$ and $m_1\in M$. As usual, we use
Sweedler's convention. In addition, we will use the sign $\pm$ to
indicate the signs involved since we will construct a map which will
turn out to be zero hence signs can be ignored for this purpose. By
Proposition \ref{5.2.3}, the morphism $\nu$ is the natural projection:
\[
\nu\colon \widetilde{\Omega}/Q\lto A\otimes_B M \otimes_B A\colon \left( \begin{array}{c}
\overline{a}^{(1)}_2\otimes\beta_2\otimes \overline{a}^{(2)}_2  \\
+\\
a^{(1)}_2 \otimes m_2 \otimes a^{(2)}_2
 \end{array} \right)\text{mod }Q\longmapsto a^{(1)}_2\otimes m_2\otimes a^{(2)}_2,
 \]
 where $a^{(1)}_2,a^{(2)}_2,\overline{a}^{(1)}_2,\overline{a}^{(2)}_2\in A$, $m_2\in M$ and $\beta_2\in\Omega^1_R A$. Define the map
 \begin{align*}
  \varphi &\defeq \kappa(a^\prime\otimes \sigma\otimes a^{\prime\prime})\left(\nu  \left( \begin{array}{c}
\overline{a}^{(1)}_2\otimes\beta_2\otimes\overline{a}^{(2)}_2  \\
+\\
a^{(1)}_2\otimes m_2 \otimes a^{(2)}_2
 \end{array} \right)\text{mod }Q \right)\in(\widetilde{\Omega}/Q)^\vee
\\
&= \nu^\vee\left( \kappa(a^\prime\otimes \sigma\otimes a^{\prime\prime})\right)
\end{align*}



To define $F$ in \eqref{diagrama-kappa}, we will need the action of
$\widetilde{\du }$ on homogeneous generators $b\in B=\T^0_BM$ and
$m_i\in M=\T^1_B M$, for all $i=1,...,r$:
\[
\widetilde{\du }\, a=\begin{cases}
((1_A\otimes\du_B b\otimes 1_A)\oplus 0)\text{mod }Q &\mbox{if } a=b
\\
( 0\oplus\sum^{r}_{i=1}m_1\cdots m_{i-1}\otimes m_i\otimes m_{i+1}\cdots m_r)\text{mod }Q & \mbox{if } a=m_1\otimes\cdots \otimes m_r
\end{cases}
\]
Define now
\[
F\colon (\widetilde{\Omega}/Q)^\vee\lto\D_R A
\]
by
\[
 F(\varphi)(a)=\begin{cases}
0 &\mbox{if } a=b,
 \\
(\varphi(\widetilde{\du}\;a))^\circ & \mbox{otherwise},
\end{cases}
\]
where $\varphi\in (\widetilde{\Omega}/Q)^\vee$ and $a\in A$. In particular, if $a=m_1\otimes\cdots\otimes m_r$ with $r>0$:
\begin{align*}
(\varphi(\widetilde{d}\; a))^\circ&=\sum^r_{i=1}\sigma_{(12)}\left((a^{\prime}\sigma^{\prime\prime}(m_i)m_{i+1}\cdots m_r)\otimes (m_1\cdots m_{i-1}\sigma^{\prime}(m_i)a^{\prime\prime})\right)
\\
&=\sum^r_{i=1}(m_1\cdots m_{i-1}\sigma^{\prime}(m_i)a^{\prime\prime})\otimes (a^{\prime}\sigma^{\prime\prime}(m_i)m_{i+1}\cdots m_r)
\end{align*}

%

\begin{claim}
$F(\varphi)\in\D_R A$.
\end{claim}

\begin{proof}
Straightforward application of the graded Leibniz rule.
\end{proof}

Next, we will focus on the vertical arrow, $\iota(\omega)\colon\D_R A
\lra{\cong}\Omega^1_R A$ given by the bi-symplectic form $\omega$ of
weight $N$ on $A$. We use the canonical isomorphism $f^{-1}\colon
\Omega^{1}_R A\cong\widetilde{\Omega}/Q$ (see \eqref{f}), which
induces another one $\Omega^{2}_R A\lra{\cong}
(\widetilde{\Omega}/Q)^{\otimes_R 2}\colon \beta \longmapsto
\widetilde{\beta}$. In particular, the bi-symplectic form $\omega$
determines an element $\widetilde{\omega}$ that can be decomposed,
using $\widetilde{\Omega}=(A\otimes_B \Omega^1_R B \otimes_B A)\oplus
(A\otimes_R M\otimes_R A)$, as follows:
\[
\widetilde{\omega}=(\widetilde{\omega}_{\text{MM}}+\widetilde{\omega}_{\text{BB}}+\widetilde{\omega}_{\text{MB}}+\widetilde{\omega}_{\text{BM}})\;\text{mod } Q.
\]
Omitting summation signs, we write 
\begin{align*}
\widetilde{\omega}_{\text{MM}}&=(\widetilde{m}_1\otimes \widetilde{m}_2)\;\text{mod }Q,
\\
\widetilde{\omega}_{\text{BB}}&=(\widetilde{\beta}_1\otimes \widetilde{\beta}_2)\;\text{mod }Q,
\\
\widetilde{\omega}_{\text{MB}}&=(\widetilde{m}_3\otimes \widetilde{\beta}_3)\;\text{mod }Q,
\\
\widetilde{\omega}_{\text{BM}}&=(\widetilde{\beta}_4\otimes \widetilde{m}_4)\;\text{mod }Q,
\end{align*}
where $\widetilde{m}_i:=a^{(1)}_i\otimes m_i\otimes a^{(2)}_i\in
A\otimes_R M \otimes_R A$ and
$\widetilde{\beta}_i:=\overline{a}^{(1)}_i\otimes \beta_i\otimes
\overline{a}^{(2)}_i\in A\otimes_B \Omega^1_R B \otimes_B A$ for
$i=1,2,3,4$, with
$a^{(1)}_i,a^{(2)}_i,\overline{a}^{(1)}_i,\overline{a}^{(2)}_i\in A$,
$m_i\in M$ and $\beta_i\in \Omega^1_R B$ for $i=1,2,3,4$. Using this
decomposition and the previous isomorphism, we can now calculate
\begin{equation}
\iota_{F(\varphi)}\widetilde{\omega}=\iota(\widetilde{\omega})(F(\varphi))=\iota\left((\widetilde{\omega}_{\text{MM}}+\tilde{\omega}_{\text{BB}}+\tilde{\omega}_{\text{MB}}+\tilde{\omega}_{\text{BM}})\;\text{mod }Q\right)(F(\varphi))
\label{omega-iota}
\end{equation}

\begin{claim}~
\begin{enumerate}
\item [\textup{(i)}]
$ \iota_{F(\varphi)}(a^{(1)}_i\otimes m_i\otimes a^{(2)}_i) = \begin{cases} 0 & \mbox{if }\lvert m_i\rvert<N
\\
a^{(1)}_i {}^{\circ}(a^{\prime}\sigma^{\prime\prime}(m_i)\otimes \sigma^{\prime}(m_i)a^{\prime\prime})a^{(2)}_i& \mbox{if } \lvert m_i\rvert= N\end{cases}$.
\item [\textup{(ii)}]
$ \iota_{F(\varphi)}(\bar{a}^{(1)}_i\otimes \beta \otimes \bar{a}^{(2)}_i) = 0$
 \end{enumerate}
 \label{claim tecnico}
 \end{claim}

 \begin{proof}
   Observe that we do not know how the operator $\iota_{F(\varphi)}$
   acts on elements of $\widetilde{\Omega}/Q$, but we do know how it
   acts on elements of $\Omega^{1}_{R} A$. Hence we
   have to use the canonical isomorphism $f$ between these objects and
   then apply the operator $\iota_{F(\varphi)}$.

\textup{(i)} Firstly,
\begin{align*}
 \iota_{F(\varphi)}(a^{(1)}_i\otimes m_i\otimes a^{(2)}_i) &=\iota_{F(\varphi)}(a^{(1)}_i (\du_A m_i ) a^{(2)}_i)
 \\
 &=a^{(1)}_i \iota_{F(\varphi)}(\du_Am_i)a^{(2)}_i
 \\
 &= a^{(1)}_i{}^{\circ}(F(\varphi)(m_i)) a^{(2)}_i
\end{align*}
We have to distinguish two cases:
\begin{enumerate}
\item [\textup{(a)}]
\textbf{Case $\lvert m_i\rvert<N$:} Since $\sigma\in M^\vee_N$, $\sigma(m_i)=0$ since $(A\otimes A)_{(j)}=\{0\}$ with $j<0$. Thus, $ \iota_{F(\varphi)}(a^{(1)}_i\otimes m_i\otimes a^{(2)}_i) =0$.
\item [\textup{(b)}]
\textbf{Case $\lvert m_i\rvert =N$:}
\begin{align*}
 \iota_{F(\varphi)}(a^{(1)}_i\otimes m_i\otimes a^{(2)}_i) &= a^{(1)}_i{}^{\circ}(F(\varphi)(m_i)) a^{(2)}_i
 \\
 &=a^{(1)}_i {}^{\circ}(a^{\prime}\sigma^{\prime\prime}(m_i)\otimes \sigma^{\prime}(m_i)a^{\prime\prime})a^{(2)}_i\in A
 \end{align*}
 \end{enumerate}

\textup{(ii)} This case is similar. By definition of $F(\varphi)$,
\begin{align*}
  \iota_{F(\varphi)}(\overline{a}^{(1)}_i\otimes \beta_i \otimes \overline{a}^{(2)}_i) &=\iota_{F(\varphi)}(\overline{a}^{(1)}_i ( b^{(1)}_i\du_Ab^{(2)}_i )  \overline{a}^{(2)})
  \\
  &= \bar{a}^{(1)}_ib^{(1)}_i {}^{\circ}(F(\varphi)(b^{(2)}_i))\overline{a}^{(2)}_i=0
\end{align*}\qedhere
\end{proof}

Now, we will use the Claim \ref{claim tecnico} to calculcate each summand in \eqref{omega-iota}:

\begin{itemize}
\item
\textbf{Case $\tilde{\omega}_{\text{BB}}$:}
\\
As $\lvert \widetilde{\omega}_{\text{BB}}\rvert=N$ and $\lvert \beta_1\rvert=\lvert \beta_2\rvert=0$, $\lvert \overline{a}^{(1)}_i\rvert+\lvert \overline{a}^{(2)}_i\rvert=N$, with $\overline{a}^{(1)}_i,\overline{a}^{(2)}_i\in A$ for $i=1,2$. Without loss of generality, in this case, we can assume that $\lvert \overline{a}^{(1)}_{1}\rvert =N$. Then
\begin{align*}
\iota(\widetilde{\omega}_{\text{BB}})(F(\varphi))&=\iota_{F(\varphi)}(\widetilde{\omega}_{\text{BB}})
\\
&=\iota_{F(\varphi)}(\widetilde{\beta}_1\otimes\widetilde{\beta}_2)
\\
&=\left(\iota_{F(\varphi)}(\widetilde{\beta}_1)\right) \widetilde{\beta}_2+\widetilde{\beta}_1\left(\iota_{F(\varphi)}(\widetilde{\beta}_2)\right)
\\
&=\left(\iota_{F(\varphi)}(\overline{a}^{(1)}_1\otimes \beta_1 \otimes \overline{a}^{(2)}_1) \right)\widetilde{\beta}_2+\widetilde{\beta}_1\left(\iota_{F(\varphi)}(\overline{a}^{(1)}_2\otimes \beta_2 \otimes \overline{a}^{(2)}_2)\right)
\\
&=0
\end{align*}

\item
\textbf{Case $\widetilde{\omega}_{\text{MM}}$:}
\\
As $\lvert \widetilde{\omega}_{\text{MM}}\rvert=N$ and $\lvert m_i\rvert\geq 1$, then $\lvert m_i\rvert<N$ for $i=1,2$. Then, using Claim \ref{claim tecnico}:
\begin{align*}
\iota(\widetilde{\omega}_{\text{MM}})(F(\varphi))&=\iota_{F(\varphi)}(\widetilde{\omega}_{\text{MM}})
\\
&=\iota_{F(\varphi)}(\widetilde{m}_1\otimes \widetilde{m}_2)
\\
&=\left(\iota_{F(\varphi)}(\widetilde{m}_1) \right)\widetilde{m}_2+\widetilde{m}_1 \left(\iota_{F(\varphi)}(\widetilde{m}_2)\right)
\\
&=\left(\iota_{F(\varphi)}(a^{(1)}_1\otimes m_1 \otimes a^{(2)}_1) \right)\widetilde{m}_2+\widetilde{m}_1\left(\iota_{F(\varphi)}(a^{(1)}_2\otimes m_2 \otimes a^{(2)}_2)\right)
\\
&=0.
\end{align*}

\item
\textbf{Case $\widetilde{\omega}_{\text{MB}}$:}
\\
In this case, $\lvert \beta_3\rvert=0$, so $N\geq \lvert m_3\rvert \geq 1$. Again, by the Leibniz rule and Claim \ref{claim tecnico}:
\begin{align*}
\iota(\widetilde{\omega}_{\text{MB}})(F(\varphi))&=\iota_{F(\varphi)}(\widetilde{\omega}_{\text{MB}})
\\
&=\iota_{F(\varphi)}(\widetilde{m}_3\otimes \widetilde{\beta}_3)
\\
&=\left(\iota_{F(\varphi)}(\widetilde{m}_3)\right)\widetilde{\beta}_3+\widetilde{m}_3\left(\iota_{\delta(\varphi)}(\widetilde{\beta}_3)\right)
\\
&=\left(\iota_{F(\varphi)}(a^{(1)}_3\otimes m_3 \otimes a^{(2)}_3)\right)\widetilde{\beta}_3+\widetilde{m}_3\left(\iota_{F(\varphi)}(\overline{a}^{(1)}_3\otimes \beta_3 \otimes \overline{a}^{(2)}_3)\right)
\\
&=\left(a^{(1)}_3{}^{\circ}( F(\varphi)(m_3))  a^{(2)}_3\right)\widetilde{\beta}_3
\end{align*}

Now we have to distinguish two cases depending on the weight of $m_3$:

\begin{enumerate}
\item [\textup{(a)}]
\textbf{Case $\lvert m_3\rvert <N$:}
by Claim \ref{claim tecnico},
\[
\iota(\widetilde{\omega}_{\text{MB}})(F(\varphi))=0.
\]
\item [\textup{(b)}]
\textbf{Case $\lvert m_3\rvert =N$:}
by the same Claim,
\begin{align*}
\iota(\widetilde{\omega}_{\text{MB}})(F(\varphi))&=\left(a^{(1)}_3 {}^{\circ}(a^{\prime}\sigma^{\prime\prime}(m_3)\otimes \sigma^{\prime}(m_3)a^{\prime\prime})a^{(2)}_3\right)\widetilde{\beta}_3
\\
&\in \left((A\otimes_B\Omega^{1}_{R} B\otimes_B A)\oplus 0\right)\text{mod }Q \subset \widetilde{\Omega}/Q
\end{align*}
\end{enumerate}

\item
\textbf{Case $\widetilde{\omega}_{\text{MB}}$}
\\
It is similar to the previous case.
\end{itemize}

In conclusion, $\iota(\widetilde{\omega})(F(\varphi))\in (A\otimes_B
\Omega^1_R B\otimes_B A\oplus 0)\;\text{mod }Q$.  For the last step,
we define the map $g$ making the following diagram commutative:
\[
\xymatrix{
0\ar[r] &A\otimes_{B} \Omega^{1}_{R} B\otimes_{B}A \ar[r]^-{\varepsilon}  & \Omega^{1}_{R} A \ar[r]^-{\nu} \ar[d]^{\cong}& A\otimes_{B} M\otimes_{B}A\ar[r] & 0
\\
& & \tilde{\Omega}/Q \ar[ru]^{g}
}
\]
By Proposition \ref{5.2.3}, we know that $\nu$ is the projection onto the second direct summand of $\widetilde{\Omega}/Q$ so $g(\iota(\widetilde{\omega}_{\text{MB}})(F(\varphi)))$ is zero in $A\otimes_B M\otimes_B A$.
%
The universal property of the kernel allows us to conclude the existence of the dashed maps
\[
\xymatrix{
& A\otimes_B M^\vee_N\otimes_B A\ar@{^{(}->}[d]
\\
0\ar[r] & A\otimes_{B} M^{\vee}\otimes_{B}A \ar[r]^-{\nu^\vee} \ar@{-->}[d]  & \D_{R} A\ar[r]^-{\varepsilon^\vee} \ar[d]^{\iota(\omega)}&  A\otimes_{B} \D_{R} B \otimes_{B}A \ar[r]  \ar@{-->}[d] & 0
\\
0\ar[r] &A\otimes_{B} \Omega^{1}_{R} B\otimes_{B}A \ar[r]^-{\varepsilon}  & \Omega^{1}_{R} A \ar[r] ^-{\nu}& A\otimes_{B} M\otimes_{B}A\ar[r] & 0
}
\]
making this diagram commutes. Finally, it follows that we constructed the following map:
\begin{equation}
\xymatrix{
A\otimes_B M^{\vee}_N\otimes_B A \ar@{-->}[r]  & A\otimes_{B} \Omega^{1}_{R} B\otimes_{B}A
}
\label{diagrama directo}
\end{equation}




Next, we will consider the `inverse' diagram:
\begin{equation}
\xymatrix{
0\ar[r] &A\otimes_{B} \Omega^{1}_{R} B\otimes_B A\ar[r]^-{\varepsilon}& \Omega^{1}_{R} A \ar[r]^-{\nu} \ar[d]^{\iota(\omega)^{-1}}&A\otimes_{B} M\otimes_{B}A \ar[r] & 0
\\
0\ar[r] &A \otimes_{B} M^{\vee}\otimes_{B}A \ar[r]^-{\nu^\vee}& \D_{R} A\ar[r]^-{\varepsilon^\vee}  &  A\otimes_{B} \D_{R} B \otimes_{B}A  \ar[r]& 0
}
\label{inverse-diagram-prueba}
\end{equation}

In a first stage, our aim is to construct the following dashed arrow:
\[
\xymatrix{
A\otimes_B \Omega^{1}_{R} B \otimes_B A  \ar@{-->}[r] &A\otimes_B M^{\vee} \otimes_B A
}
\]
which makes the previous diagram commutative.

We begin by recalling that since $\widetilde{\Omega}=(A\otimes_B \Omega^{1}_{R} B \otimes_B A)\oplus (A\otimes_R M \otimes_R A)$, $h$ is the imbedding of the first direct summand in $\widetilde{\Omega}$ (see Proposition \ref{5.2.3}), $\texttt{proj}$ is the natural projection and the isomorphism $f$ was defined in \eqref{f},
\[
\xymatrix{
& & \tilde{\Omega} \ar@{->>}[d]^{\texttt{proj}}
\\
& & \tilde{\Omega}/Q\ar[d]^{f}
\\
0\ar[r] &A\otimes_{B} \Omega^{1}_{R} B\otimes_B A\ar[r]^-{\varepsilon}  \ar[uur]^{h} & \Omega^{1}_{R} A \ar[r]^-{\nu}& A\otimes_{B} M\otimes_{B}A \ar[r] & 0
}
\]

Let $a^{\prime},a^{\prime\prime}\in A$, $b\in B$ and $\du_Bb\in \Omega^{1}_R B$. Then $a^{\prime}\otimes \du_B b\otimes a^{\prime\prime}\in A\otimes_B \Omega^1_R B\otimes_ B A$.
It is a simple calculation that
\begin{equation}
\varepsilon\colon A\otimes_B \Omega^{1}_{R} B \otimes_B A \lto \Omega^{1}_R A\colon \quad a^{\prime}\otimes \du_Bb\otimes a^{\prime\prime} \longmapsto a^{\prime}(\du_Ab)a^{\prime\prime}
\label{varepsilon}
\end{equation}

Now, we focus on the vertical arrow of the diagram \eqref{inverse-diagram-prueba}. Observe that since $\omega$ is a bi-symplectic form of weight $N$, $\iota(\omega)^{-1}$ has weight $-N$. In fact,
%
%
%
using \eqref{10.333}, we can write this double Poisson bracket in terms of the Hamiltonian double derivation. Nevertheless, since $\lr{-,-}_\omega$ is $A$-bilinear with respect to the outer bimodule structure on $A\otimes A$ in the second argument and $A$-bilinear with respect to the inner bimodule structure on $A\otimes A$ in the first argument, it is enough to consider $a^{\prime}=a^{\prime\prime}=1_A$. Then
\begin{equation}
\left(\iota(\omega)^{-1}\circ\varepsilon\right)(1_A\otimes\du_B b \otimes 1_A)=\lr{b,-}_\omega = H_{b}
\label{antesantes}
\end{equation}
Observe that $H_{b}\in \D_R A$  has weight $-N$. Finally, since $\texttt{inj}$ is the imbedding of $A\otimes_B\Omega^{1}_R B\otimes_B A$ in the first direct summand of $\widetilde{\Omega}$, we will determine $\Psi$ to the square in the following diagram commutes: :
\[
\xymatrix{
0\ar[r] &A\otimes_{B} M^{\vee}\otimes_{B}A \ar[r]^-{\varepsilon^\vee} & \D_{R} A \ar[r]^-{\nu^\vee}\ar[d]^{\Psi} & A\otimes_{B} \D_R B\otimes_{B}A\ar[r] & 0
\\
& & (\Omega^{1}_{R} A)^{\vee}  \ar[d]^{ f^{\vee}}
\\
& & (\widetilde{\Omega}/Q)^{\vee}   \ar@{->>}[d]^{\texttt{proj}^{\vee}}
\\
& & (\widetilde{\Omega})^{\vee} \ar@{^{(}->}[r]^-{(\texttt{inj})^{\vee}}& (A\otimes_B\Omega^{1}_R A\otimes_B A)^{\vee}\ar[uuu]^{\cong}
}
\]
In this diagram, we define
\[
\Psi\colon\D_RA\lto ( \Omega^{1}_R A)^{\vee}\colon \quad \Theta \longmapsto \Psi(\Theta)
\]
given by
\[
\Psi(\Theta)\colon \Omega^{1}_R A\lto A\otimes A\colon \quad \alpha\longmapsto \Psi(\Theta)(\alpha)= (i_\Delta \alpha)^{\circ}=\pm i^{\prime\prime}_\Theta (\alpha) \otimes i^{\prime}_\Theta (\alpha),
\]
where $\pm\defeq (-1)^{(\;\mm{\;i^\prime_\Theta(\alpha)\;}\;\;\mm{\;i^{\prime\prime}_\Theta(\alpha)\;}\;+\lvert i^{\prime}_\Theta(\alpha)\rvert\lvert i^{\prime\prime}_\Theta(\alpha)\rvert)}$ (see \eqref{circulo derecha}).
When we apply $\Psi$ to the element in \eqref{antesantes}:
\begin{equation}
\Psi \colon \D_R A  \lto (\Omega^{1}_{R} A)^{\vee}\colon\quad  H_{b}  \longmapsto  i_{H_b},
\label{Psi}
\end{equation}
such that
\begin{equation}
\begin{aligned}
 i_{H_b}\colon \Omega^{1}_{R} A &\lto A\otimes A
 \\
  \overline{c}_1\du_A\overline{c}_2&\longmapsto (   \overline{c}_1H_{b}(\overline{c}_2))^\circ
\end{aligned}
\end{equation}
Next, applying $f^{\vee}$, we obtain the following:
\begin{equation}
\begin{aligned}
(f^{\vee}\circ\Psi) (H_{b} )\colon \widetilde{\Omega}/Q &\to A\otimes A
\\
\left( \begin{array}{c}
\overline{a}^{(1)}_2\otimes b^{(1)}_1\du_Bb^{(2)}_1\otimes\overline{a}^{(2)}_2
\\
+
\\
a^{(1)}_2\otimes m \otimes a^{(2)}_2
 \end{array} \right)\text{mod } Q &\mapsto \left( \begin{array}{c}
\left( \overline{a}^{(1)}_2b^{(1)}_1H_{b}(b^{(2)}_1)\overline{a}^{(2)}_2\right)^\circ
\\
+
\\
 \left( a^{(1)}_2H_{b}(m) a^{(2)}_2\right)^\circ
 \end{array} \right)
 \end{aligned}
 \label{donde-aparece-H}
 \end{equation}

 To shorten the notation, we make the following definition:
 \[
 L\defeq \left(\texttt{proj}^\vee\circ f^\vee\circ\Psi\right)(H_{b}  ).
 \]
  Finally, since  \begin{align*}
\texttt{ inj}^{\vee}\circ L \colon A\otimes_B\Omega^1_R A\otimes_B A &\lto A\otimes A
\\
 \overline{a}^{(1)}_2\otimes b^{(1)}_1\du_Bb^{(2)}_1\otimes \overline{a}^{(2)}_2 &\longmapsto  \left(\overline{a}^{(1)}_2b^{(1)}_1H_{b}(b^{(2)}_1)\overline{a}^{(2)}_2\right)^\circ
 \end{align*}
The key point is to observe that since $b, b^{(2)}_1\in B$, $\lvert b\rvert=\lvert b^{(2)}_1\rvert=0$. Thus, $\lvert H_{b^{\prime\prime}}(b^{(2)}_1)\rvert=-N<0$. Thus, $H_{b}(b^{(2)}_1)=0$ because $A$ is a bi-symplectic tensor $\mathbb{N}$-algebra so, in particular, it is non-negatively graded. By the universal property of the kernel, we conclude the existence of the dashed arrows which makes the following diagram commutes
 \[
 \xymatrix{
0\ar[r] &A\otimes_{B} \Omega^{1}_{R} B\otimes_B A\ar[r]^-{\varepsilon}  \ar@{-->}[d] & \Omega^{1}_{R} A \ar[r]^-{\nu} \ar[d]^{\iota(\omega)^{-1}}&A\otimes_{B} M\otimes_{B}A \ar[r]  \ar@{-->}[d] & 0
\\
0\ar[r] &A \otimes_{B} M^{\vee}\otimes_{B}A \ar[r]^-{\nu^\vee}& \D_{R} A\ar[r]^-{\varepsilon^\vee}  &  A\otimes_{B} \D_{R} B \otimes_{B}A  \ar[r]& 0
}
\]
In special, we are interested in the map
\begin{equation}
\xymatrix{
 A\otimes_B\Omega^1_RB\otimes_B A \ar@{-->}[r] &A\otimes_B M^{\vee}\otimes_B A
}
\label{map-intermedia-bidual}
\end{equation}
Finally, in \eqref{donde-aparece-H}, we point out that $ \left( a^{(1)}_2H_{b}(m) a^{(2)}_2\right)^\circ=0$ unless $m\in M_N$ since $\lvert H_{b}\rvert=\lvert m\rvert-N$. Hence, as a consequence of this discussion and using \eqref{map-intermedia-bidual}, we obtain:
\begin{equation}
\xymatrix{
 A\otimes_B\Omega^1_RB\otimes_B A \ar@{-->}[r] &A\otimes_B M^{\vee}_N\otimes_B A
}
\label{otraotra}
\end{equation}
By construction, \eqref{diagrama directo} and \eqref{otraotra} are inverse to each other. So, we proved the existence of the following isomorphism:
\[
 A\otimes_B\Omega^1_R B\otimes_B A \cong   A\otimes_B M^{\vee}_N\otimes_B A
 \]

 Or, equivalently using the fact that, by hypothesis, $B$ is a smooth associative $R$-algebra,
\begin{equation}
 A\otimes_B\D_RB\otimes_B A\cong   A\otimes_B M_N \otimes_B A
 \end{equation}
For reasons that we will clarify below, we make precise this isomorphism:
\begin{claim}
Let $(A,\omega)$ be a bi-symplectic tensor $\mathbb{N}$-algebra of weight $N$. Then $\iota(\omega)^{-1}$ restricts to a $B$-bimodule isomorphism
\begin{equation}
E_N \cong \D_R B.
\label{restriccionrestriccion}
\end{equation}
\label{Claimrestriction}
\end{claim}

\begin{proof}
Observe that in the commutative diagram
\[
\xymatrix{
0\ar[r] &A\otimes_{B} \Omega^{1}_{R} B\otimes_B A  \ar[r]^-{\varepsilon}  \ar@{-->}[d]& \Omega^{1}_{R} A \ar[r]^-{\nu} \ar[d]^{\iota(\omega)^{-1}}&A\otimes_{B} M\otimes_{B}A \ar[r]\ar[d] & 0
\\
0\ar[r] &A \otimes_{A} M^{\vee}\otimes_{B}A \ar[r]^-{\nu^\vee} & \D_{R} A\ar[r]^-{\varepsilon^\vee} &  A\otimes_{B} \D_{R} B \otimes_{B}A  \ar[r] & 0
}
\]
the dashed arrow is $\iota(\omega)^{-1}\circ \varepsilon$. We will see that the weight of this map is -N. This is equivalent to prove that $\lvert \varepsilon\rvert=0$ since $\omega$ is a bi-symplectic form of weight $N$. As we discussed in \eqref{varepsilon},
\[
\varepsilon\;\colon  A\otimes_B\Omega^{1}_R B\otimes_B A\lto\Omega^{1}_R A\colon\quad  a^{\prime}\otimes b^{\prime}\du_Bb^{\prime\prime}\otimes a^{\prime\prime}\longmapsto a^{\prime} (b^{\prime}\du_Ab^{\prime\prime}) a^{\prime\prime}.
\]
Thus, it is immediate that $\lvert \varepsilon\rvert=0$.

Finally, observe that $A$ is non-negatively graded and $M_N$ has weight $N$ while $\D_R B$ has weight 0. Note that the part of weight 0 of $A\otimes_B\D_R B\otimes_B A$ is $B\otimes_B\D_R B\otimes_B B$ which is isomorphic to $\D_R B$. Similarly, $(A\otimes_B M_{N}\otimes_B A)_{N}=B\otimes_BM_{N}\otimes_B B$, where $(-)_N$ denotes the part of weight $N$. Thus, we obtain the following isomorphism of $B$-bimodules,
\[
E_N \cong \D_R B\qedhere
\]
\end{proof}

Claim \ref{Claimrestriction} concludes the proof of Theorem
\ref{weight 0}.
\qed

\end{document}